\documentclass[12pt,hyperref]{article}
\usepackage{graphicx} 
\usepackage{float} 
\usepackage{amssymb}
\usepackage{amsmath}
\usepackage{mathtools}
\usepackage[thmmarks,amsmath]{ntheorem} 
\usepackage{lineno}
\usepackage{enumitem}
\usepackage{color}

\usepackage{amsmath}
\allowdisplaybreaks[4]

\usepackage{etoolbox}

\newcommand{\reduceabovedisplayskip}{%
    \setlength{\abovedisplayskip}{2.0pt plus 1pt minus 1pt}%
}
\newcommand{\reducebelowdisplayskip}{%
    \setlength{\belowdisplayskip}{2.0pt plus 1pt minus 1pt}%
}

\usepackage{bm} 
\usepackage{extarrows}

\makeatletter

{
    \theoremstyle{nonumberplain}
    \theoremheaderfont{\bfseries}
    \theorembodyfont{\normalfont}
    \theoremsymbol{\mbox{$\Box$}}
    \newtheorem{proof}{Proof}
}
\usepackage{titlesec}

\titlespacing*{\section}
  {0pt}                        
  {1.3ex plus .2ex}  
  {1.3ex plus .2ex}            

\usepackage{theorem}
\newtheorem{theorem}{Theorem}[section]
\newtheorem{proposition}{Proposition}[section]
\newtheorem{lemma}{Lemma}[section]
\newtheorem{definition}{Definition}[section]

\newtheorem{corollary}{Corollary}[section]
\newtheorem{fact}{Fact}[section]
\newtheorem{claim}{Claim}[section]

\newtheorem{problem}{Problem}[section]
\newtheorem{conjecture}{Conjecture}[section]
{
    \theoremheaderfont{\bfseries}
    \theorembodyfont{\normalfont}
    
}
\newcommand{\RNum}[1]{\uppercase\expandafter{\romannumeral #1\relax}}
\usepackage{amsmath}
\usepackage{caption}
\graphicspath{{figure/}}                                 
\usepackage[a4paper]{geometry}                        
\begin{document}
\reduceabovedisplayskip
\reducebelowdisplayskip
\title{Semi-Degree Condition for Arbitrary $H$-Linked Oriented Graphs\thanks{The author's work is supported by National Natural Science Foundation of China {(No.12571373, 12071260)}}}

\author{Jia Zhou, Jin Yan\thanks{Corresponding author. E-mail adress: yanj@sdu.edu.cn}  \unskip\\[2mm]
School of Mathematics, Shandong University, Jinan 250100, China}

\date{}
\maketitle

\begin{abstract}
Let $ H $ be a multi-digraph on $ h $ vertices with $ q $ arcs. An \textbf{$H$-subdivision} in a digraph $D$ is a subdigraph obtained by replacing every arc $uv$ of $H$ with a path from $u$ to $v$ in $D$ such that these paths are pairwise internally vertex-disjoint. A digraph $ D $ is \textbf{arbitrary $ H $-linked} if, for every injection $ f: V(H) \to V(D) $, there exists an $ H $-subdivision in $ D $ such that each vertex $ v \in V(H) $ is mapped to $ f(v) \in V(D) $, and the length of every subdivision path can be arbitrarily specified as {an integer \(l \geq 4\)}. An oriented graph is a digraph without 2-cycles. Keevash, K\"{u}hn, and Osthus proved that every sufficiently large oriented graph $ D $ of order $ n $ with $\delta^0(D) \geq \frac{3n-4}{8}$ contains a Hamilton cycle (i.e., a $\overset{\leftrightarrow}{K_2}$-subdivision). Subsequently, Kelly, K\"{u}hn, and Osthus showed that such oriented graphs {are also arbitrary $ H $-linked, where $H$ is a loop}. Motivated by these results, we establish a minimum semi-degree condition for arbitrary  $ H $-linked oriented graphs: there exists $ n_0 = n_0(h,q) $ such that every oriented graph $ D $ of order $ n \geq n_0 $ with $\delta^0(D) \geq \frac{3n + 3h + 3q - 5}{8}$ is arbitrary $ H $-linked; specifically, if $H$ is a loop, this holds under the weaker condition $\delta^0(D) \geq \frac{3n - 4}{8}$. The result provides an oriented graph analogue of Wang's conjecture on cycle-factors in graphs [J. Korean Math. Soc. 51 (2014) 919--940] and determines the tight semi-degree bounds for both strongly Hamiltonian-connected and arbitrary $ q $-linked oriented graphs.
\end{abstract}

{\noindent\small{\bf Keywords: }Oriented graphs; Semi-degree; Arbitrary  $H$-linked

\vspace{1ex}

\section{Introduction}
 The Hamilton cycle problem is one of the most well-known problems in graph theory. {A classical result is Dirac's theorem \cite{Dirac(1952)}, which establishes that every graph $G$ on $n\geq 3$ vertices with minimum degree at least $n/2$ contains a \textbf{Hamilton cycle} (i.e., a cycle {that} encounters each vertex exactly once). Extending this to digraphs, Ghouila-Houri \cite{Gho(1960)} demonstrated that every digraph $D$ on $n$ vertices with minimum semi-degree \(\delta^0(D) \geq n/2\) contains a Hamilton cycle. Here, \(\delta^0(D)\) denotes the \textbf{minimum semi-degree} of \( D \), defined as the minimum of the minimum out-degree \(\delta^+(D)\) and the minimum in-degree \(\delta^-(D)\). The situation becomes more complex for \textbf{oriented graphs} (i.e., digraphs without 2-cycles). A breakthrough result by Keevash, K\"{u}hn, and Osthus \cite{Keevash(2009)} established the tight bound $\delta^0(D) \geq (3n-4)/8$ for the existence of Hamilton cycles in sufficiently large oriented graphs.} Further, Kelly, K\"{u}hn, and Osthus \cite{Kelly(2009)} also proved that every sufficiently large oriented graph \( D \) with \(\delta^0(D) \geq (3n-4)/8\) is \textbf{almost vertex-pancyclic} (i.e., for every vertex \( v \in V(D) \), there exists a cycle of every possible length from 4 to \( n \) that includes \( v \)). {In this paper}, the \textbf{length} of a path or a cycle is defined as the number of arcs.

A generalization of the Hamilton cycle is the concept of subdivision. For a \textbf{multi-digraph} \( H \) (i.e., a digraph allowing loops and multiple arcs), an \textbf{\( H \)-subdivision} in a digraph \( D \) is a subdigraph obtained by replacing all arcs of $H$ with pairwise internally vertex-disjoint paths in $D$ {such that each replacement path is consistent in direction with its corresponding arc of $H$}. When such a subdivision spans all vertices of  \( D \), it is called \textbf{Hamiltonian}. Notably, if \( H \) is a 2-cycle, this directly corresponds to the existence of a Hamilton cycle in \( D \). We further consider a stronger concept called \( H \)-linked digraphs, which requires embedding an \( H \)-subdivision with specific vertex correspondences. A digraph \( D \) is \textbf{(Hamiltonian) \( H \)-linked} if for any injection \( f: V(H) \to V(D) \), there exists a (Hamiltonian) \( H \)-subdivision in \( D \) such that each vertex \( v \in V(H) \) is mapped to \( f(v) \in V(D) \). Furthermore, if the length of every subdivision path in the (Hamiltonian) \( H \)-subdivision can be arbitrarily specified as {an integer \(l \geq 4\), provided the total number of vertices in the $H$-subdivision does not exceed $n$}, then \( D \) is called \textbf{arbitrary (Hamiltonian) \( H \)-linked}. By definition, an arbitrary $H$-linked digraph is Hamiltonian $H$-linked, implying the former is a stronger property. Previous studies have explored \( H \)-linkage in graphs and digraphs; see \cite{Ferrara (2012), Ferrara (2006),Ferrara (2013), Gould(2006), kosto(2005), kostoc(2008), kostoc(20082)} for details.  Recently, Cheng, Wang, and Yan \cite{Wang(2024)} proved that for any digraph \( H \) with \( q \) arcs, there exists \( n_0 = n_0(q) \) such that every digraph \( D \) of order \( n \geq n_0 \) with \(\delta^0(D) \geq \frac{n}{2} + q\) is arbitrary Hamiltonian \( H \)-linked. Remarkably, when \( H \) is a single loop, arbitrary \( H \)-linked digraphs not only contain a Hamilton cycle but are also \textbf{almost vertex-pancyclic} by definition.

Our main theorem provides a semi-degree condition for arbitrary \( H \)-linked oriented graphs, extending previous result on almost vertex-pancyclicity in oriented graphs.

\begin{theorem}\label{main}
Let $ h, q $ be positive integers, and let $ H $ be a multi-digraph on $ h $ vertices with $ q $ arcs. There exists $ n_0 = n_0(h,q) $ such that for every oriented graph $ D $ of order $ n \geq n_0 $, if $ \delta^0(D) \geq \frac{3n + 3h + 3q -5}{8} $, then $ D $ is arbitrary  $ H $-linked. Specially, if $H$ is a loop, then every oriented graph $ D $ of order $ n \geq n_0 $ with $ \delta^0(D) \geq \frac{3n -4}{8} $ is arbitrary  $ H $-linked.
\end{theorem}

The semi-degree bound in Theorem \ref{main} is tight in a certain sense: Proposition \ref{prop2} constructs a family of oriented graphs achieving the threshold \(\delta^0(D) \geq \frac{3n + 3h + 3q - 13}{8}\) while failing to be arbitrary Hamiltonian \(H\)-linked for the multi-digraph $H$ {consisting of} \(q\) parallel arcs. The requirement for path lengths at least 4 in the definition of ``arbitrary  \(H\)-linked'' is necessary, as Proposition \ref{prop2} exhibits oriented graphs with $\delta^0(D) \geq \frac{13n}{32} - 3$ $>\frac{3n + 3h + 3q - 5}{8}$ (for sufficiently large $n$) failing to be arbitrary  \(H\)-linked if the length of any subdivision path is constrained to at most 3.

The proof of Theorem \ref{main} proceeds in three phases: First, all short subdivision paths are constructed through extremal structural analysis and the existence of short path with specified length between any two vertices (Lemma \ref{lemma3}), which is a key technical contribution. Second, an innovative application of the probabilistic method (Lemma \ref{lemma2}) randomly partitions the remaining oriented graph to achieve optimal vertex allocation for long subdivision paths. Finally, through a skillfully designed pinch vertex operation (Definition \ref{keydef}), the problem of embedding each long subdivision path is reduced to finding a Hamilton cycle. The proof is completed by leveraging foundational results on Hamilton cycles in oriented graphs by Keevash, K\"{u}hn, and Osthus \cite{Keevash(2009)} (Lemma \ref{4.3} and Lemma \ref{4.5}).

Interestingly, {when $H$ consists of $q$ vertex-disjoint arcs (i.e., $ h=2q $), the notion of Hamiltonian $H$-linked has been known under the name of \textbf{Hamiltonian $q$-linked}}, i.e., $D$ contains a collection of {vertex-disjoint} paths $P_1, P_2,\ldots,P_q$ such that $P_i$ is a path from $x_i$ to $y_i$ and $\bigcup_{i=1}^q V(P_i)=V(D)$, for every choice of vertices $x_1, \ldots,x_q, y_1, \ldots,y_q$ in $V(D)$. In 2008, K\"{u}hn, Osthus, and Young \cite{Kuhn(2008)} proved that every digraph $D$ on $n\geq n_0(q)$ vertices with $\delta^0(D) \geq \lceil n/2+q\rceil - 1$ is Hamiltonian $q$-linked. By substituting $h = 2q$ into Theorem \ref{main}, we derive Corollary \ref{cor1} as the oriented graph version. Indeed, the conclusion of Corollary \ref{cor1} is stronger: it is $q$-linked, and the lengths of every path can be prescribed arbitrarily as a number at least 4, i.e., \textbf{arbitrary $q$-linked}.

\begin{corollary}\label{cor1}
For any positive integer $q$, there is an integer $n_0 = n_0(q)$ such that every oriented graph $D$ on $n\geq n_0$ vertices with $\delta^0(D)\geq (3n + 9q-5)/8$ is arbitrary  $q$-linked.
\end{corollary}

{Proposition \ref{prop1} shows that Corollary \ref{cor1} is tight when $q=1$}. Specially, Hamiltonian 1-linked digraphs are \textbf{strongly Hamiltonian-connected}, that is, for any two vertices $x$ and $y$, there is a Hamiltonian path from $x$ to $y$. Berge's classical result \cite{Berge(1985)} states that every digraph \( D \) on \( n \) vertices with \(\delta^0(D) \geq \frac{n + 1}{2}\) is strongly Hamiltonian-connected. Corollary \ref{cor1} implies that every oriented graph $D$ on $n\geq n_0$ vertices with {$\delta^0(D)\geq \frac{3n + 4}{8}$} is strongly Hamiltonian-connected, thereby establishing an oriented graph analogue of this fundamental property.


On the other hand, given $q$ disjoint arcs $f_1, \ldots, f_q$ of $D$, Corollary \ref{cor1} provides a positive integer $n_0 = n_0(q)$ such that every oriented graph $D$ of order $n\geq n_0$ with $\delta^0(D)\geq (3n + 9q-5)/8$ contains $q$ disjoint cycles $C_1,C_2,\ldots, C_q$ with lengths $n_1,\ldots,n_q$ respectively, such that $f_i$ is an arc of $C_i$ for all $i\in [q]$, where $n_1, \ldots, n_q$ are $q$ integers satisfying all $n_i\geq 5$ and $\sum_{i=1}^{q} n_i\leq  n$. This result establishes an oriented graph version of a conjecture of Wang \cite{wang(2014)} in graphs.

\begin{conjecture}
\cite{wang(2014)} Let $q$ be an integer with $q\geq 2$ and $G$ be a graph of order $n = \sum^q_{i=1}n_i$, where $n_i\geq 5$ for $i \in [q]$. Let $\{f_1,\ldots,f_q\}$ be $q$ vertex-disjoint arcs in $G$. If the minimum degree of $G$ is at least $(n+2q)/2$, then $G$ can be partitioned into $q$ cycles $C_1,\ldots,C_q$ such that $f_i\in E(C_i)$ and $|C_i|=n_i$ for $i \in [q]$.
\end{conjecture}

A digraph is \textbf{regular} if all vertices has same in-degree and out-degree. Recently, Lo, Patel, and Yıldız \cite{Lo(2024)} proved that every $d$-regular sufficiently large oriented graph on $n$ vertices with $n\leq 4d+1$ is Hamiltonian. This motivates the following open problem.
\begin{problem}
Let $ h, q $ be positive integers, and let $ H $ be a multi-digraph on $ h $ vertices with $ q $ arcs. {Is it true that there exist constants} \(c_{1}\), \(c_{2}\), and \(c_{3}\) such  {that every} $d$-regular oriented graph $ D $ of order $ n $ is arbitrary  $ H $-linked, where $d=\frac{n+c_1h+c_2q+c_3}{4}$?
\end{problem}

The rest of the paper is organized as follows. {Section 2 shows that the minimum semi-degree bounds in Theorem \(\ref{main}\) and Corollary \(\ref{cor1}\) are tight}. Section 3 introduces additional tools for the main theorem. The proof of Theorem \ref{main} is presented in Section 4, which is divided into extremal case and non-extremal case. For the extremal case, we use structural analysis and random partitioning, as shown in Subsection 4.1. For the non-extremal case, we use the directed version of Szemer\'{e}di regular lemma and random partition to complete the proof, as shown in Subsection 4.2. For completeness, detailed proofs of some referenced lemmas and associated calculations are provided in the Appendix\footnote{See the file named "Appendix"}.

\section{Notation and Counterexamples}
\subsection{Notation}
The \textbf{order} of a digraph \(D\) (denoted as \(\boldsymbol{|D|}\)) is the number of vertices, and \(\boldsymbol{a(D)}\) denotes the number of arcs. Given two vertices \(x\) and \(y\) in \(V(D)\), we write \(\boldsymbol{xy}\) for the arc directed from \(x\) to \(y\). Furthermore, we refer to the arc \(xy\) as an \textbf{in-arc} of \(y\) and an \textbf{out-arc} of \(x\). For a vertex \(x\) of \(D\), we write \(\boldsymbol{N^+_D(x)} = \{y: xy \in A(D)\}\) for the \textbf{out-neighbourhood} and \(\boldsymbol{N^-_D(x)} = \{y: yx \in A(D)\}\) for the \textbf{in-neighbourhood} of $x$. We also write \(d^+_D(x) = |N^+_D(x)|\) for the \textbf{out-degree} and \(d^-_D(x) = |N^-_D(x)|\) for the \textbf{in-degree} of $x$. Given a set \(X \subseteq V(D)\), we denote \(\boldsymbol{N^+_D(X)}\) as the union of the out-neighbourhood of every vertex \(x \in X\). The subdigraph of \(D\) \textbf{induced} by \(X\) is denoted by \(\boldsymbol{D[X]}\), and \(\boldsymbol{D - X}\) denotes the digraph obtained from \(D\) by deleting \(X\) and all arcs incident with \(X\). {Let $X$ and $Y$ be two subsets of $V(D)$. Define $\boldsymbol{X \setminus Y}$ as the set obtained by removing the vertices of $Y$ from $X$. Denote $A(X,Y)$ as the set of arcs from $X$ to $Y$, and let $a(X,Y)$ denote the number of such arcs.}

A \textbf{matching} in a digraph is a set of disjoint arcs with no common endvertices. {A matching \(M\) is \textbf{maximal} if there is no arc in \(D - V(M)\) that can be added to $M$, while \textbf{keeping} it a \textbf{matching}}. Let $P = x_1 \cdots x_t$ be a directed path in $D$, denoted as an $(x_1, x_t)$-path. Here, $x_1$ is the \textbf{initial} vertex and $x_t$ is the \textbf{terminal} vertex of $P$; individually, either $x_1$ or $x_t$ may also be referred to as an \textbf{endvertex} of $P$. We denote \(zx_1P\) as the path starting at \(z\), passing through \(x_1\), and following \(P\) to the terminal vertex of $P$.  An \textbf{\(l\)-path} is a path of length \(l\). Two paths are \textbf{internally vertex-disjoint} if they do not share any interior vertices, and we abbreviate ``vertex-disjoint" as "disjoint" for brevity.

Throughout this paper, the notation \(\boldsymbol{0 < \beta \ll \alpha}\) is used to indicate that \(\beta\) can be chosen to be sufficiently small relative to \(\alpha\) so that all calculations required in our proof are valid. For positive integers \(i,t\) such that $i\leq t$, let \(\boldsymbol{[t]} = \{1, \ldots, t\}\) and \(\boldsymbol{[i, t]} = \{i, \ldots, t\}\). For convenience, let \(\boldsymbol{a \pm b}\) denote an unspecified real number in the interval \([a - b, a + b]\).

\subsection{Counterexamples}
In this subsection, we present counterexamples to illustrate  the tightness of the bounds in Theorem \ref{main} and Corollary \ref{cor1}. Key concepts include: A  \textbf{tournament} is a digraph such that there is exactly one arc between any two vertices. Define \(\boldsymbol{A \rightarrow B}\) if and only if for every \(a \in A\) and \(b \in B\), there is an arc \(a b\). Given a set $U\subseteq V(D)$, define \( \boldsymbol{d^{\pm}_D(U)}=\min \{d^+_D(x),d^-_D(x)\mid x\in U\} \).

\begin{proposition}\label{prop1}
Let $h$ be an even number, and let $n,q$ be positive integers with $n\geq 3h+7q$. Suppose $H$ is a multi-digraph on $h$ vertices with $q$ arcs that are all directed from $u$ to $v$, where $u,v\in V(H)$. There exists an oriented graph $D$ on $n$ vertices with $\delta^0(D)\geq (3n+3h+3q-13)/8$ such that $D$ is not arbitrary Hamiltonian $H$-linked.
\end{proposition}

\begin{proof}
Let $c\geq 1$ be an integer. Set $n=8c+3h+7q-9$. Define an oriented graph $D$ on $n$ vertices as follows (see Fig. \ref{tu1}).
\begin{itemize}[itemsep=0pt, topsep=0.5pt,parsep=1pt]
\item[(i)] $V(D)$ has a partition $(D_1,D_2,D_3,D_4 ,B)$, where $|D_1|=|D_3|=(n+h+q-11)/4=2c+h+2q-5$ (odd), $|D_2|=(n-3h+q+9)/4=2c+2q$ (even), $|D_4|=(n-3h-3q+13)/4$, and $|B|=h$. Furthermore, $D_2$ admits a partition $(D^1_2$, $D^2_2)$ such that $|D^1_2|=|D^2_2|$ (this can be done since \(|D_2|\) is even). Partition $B$ into $(B_1$, $B_2, \ u',\ v')$ such that $|B_1|=|B_2|=h/2-1$.
\item[(ii)] Each of $D_1$, $D^1_2$, $D^2_2$, and $D_3$ induces a tournament as regular as possible. Additionally, \(D_4\) and \(B\) form independent sets.
\item[(iii)] Let $D_1\rightarrow D_2\rightarrow D_3 \rightarrow D_4\rightarrow D_1$, $D^1_2\rightarrow D^2_2$, $D^2_2\rightarrow D_4\rightarrow D^1_2$ and $D_3\rightarrow D_1$, and $D_3\cup D_4\rightarrow u'\rightarrow D_1\cup D_2$, $D_2\cup D_3\rightarrow v' \rightarrow D_1\cup D_4$.  Also, let $D_1\cup D_4\rightarrow B_1\rightarrow D_2\cup D_3$ and let $D_1\cup D_2\rightarrow B_2\rightarrow D_3\cup D_4$.
\end{itemize}

From the construction of $D$, we have that
\begin{itemize}[itemsep=0pt, topsep=3pt,parsep=2pt]
{\item $ d^{\pm}_D(D_1)\geq \min \{\frac{|D_1|-1}{2}+|D_2\cup B_1\cup B_2|,\frac{|D_1|-1}{2}+|D_3\cup D_4\cup \{u',v'\}|\}=\frac{3n + 3h + 3q -13}{8},$
\item $ d^{\pm}_D(D^1_2)\geq \min \{\frac{|D^1_2|-2}{2}+|D^2_2\cup D_3\cup \{v'\}\cup B_2|,\frac{|D^1_2|-2}{2}+|D_4\cup D_1\cup \{u'\}\cup B_1|\}=\frac{7n+3h+7q-33}{16},$
\item $d^{\pm}_D(D^2_2)\geq \min \{\frac{|D^2_2|-2}{2}+|D_4\cup D_3\cup \{v'\}\cup B_2|,\frac{|D^2_2|-2}{2}+|D^1_2\cup D_1\cup \{u'\}\cup B_1|\}=\frac{7n+3h+7q-33}{16},$
\item $d^{\pm}_D(D_3)\geq \min \{\frac{|D_3|-1}{2}+|D_4\cup D_1\cup \{u',v'\}|,\frac{|D_3|-1}{2}+|D_2\cup B_1\cup B_2|\}=\frac{3n + 3h + 3q -13}{8},$
\item $d^{\pm}_D(D_4)\geq \min \{|D_1\cup D^1_2\cup \{u'\}\cup B_1|,|D^2_2\cup D_3\cup \{v'\}\cup B_2|\}= \frac{3n + 3h + 3q -13}{8},$
\item $ d^{\pm}_D(B_1)\geq  \min \{|D_2\cup D_3|,|D_1\cup D_4|\}= (n -h -q +1)/2,$
\item $d^{\pm}_D(B_2)\geq  \min \{|D_3\cup D_4|,|D_1\cup D_2|\}= (n -h -q +1)/2,$
\item $d^{\pm}_D(u')\geq  \min \{|D_1\cup D_2|,|D_3\cup D_4|\}= (n -h -q +1)/2,$
\item $ d^{\pm}_D(v')\geq  \min \{|D_1\cup D_4|,|D_2\cup D_3|\}= (n -h -q +1)/2.$}
\end{itemize}
 We thus can immediately infer that $\delta^0(D)=\frac{3n + 3h + 3q -13}{8}$, satisfying the required condition.

We define a bijection \( f: V(H) \to B \) such that \( f(u) = u' \) and \( f(v) = v' \). By the construction of the oriented graph \( D \), for any path from \( u' \) to \( v' \) that is internally disjoint with \( B \), the number of vertices used from \( D_2 \) is at least one more than the number of vertices used from \( D_4 \). Since \( |D_2| = |D_4| + q - 1 \), there are no $q$ internally disjoint $(u',v')$-path. Therefore, \( D \) {is} not arbitrary Hamiltonian \( H \)-linked.
\end{proof}
\begin{figure}[H]
\centering
   \includegraphics[scale=0.6]{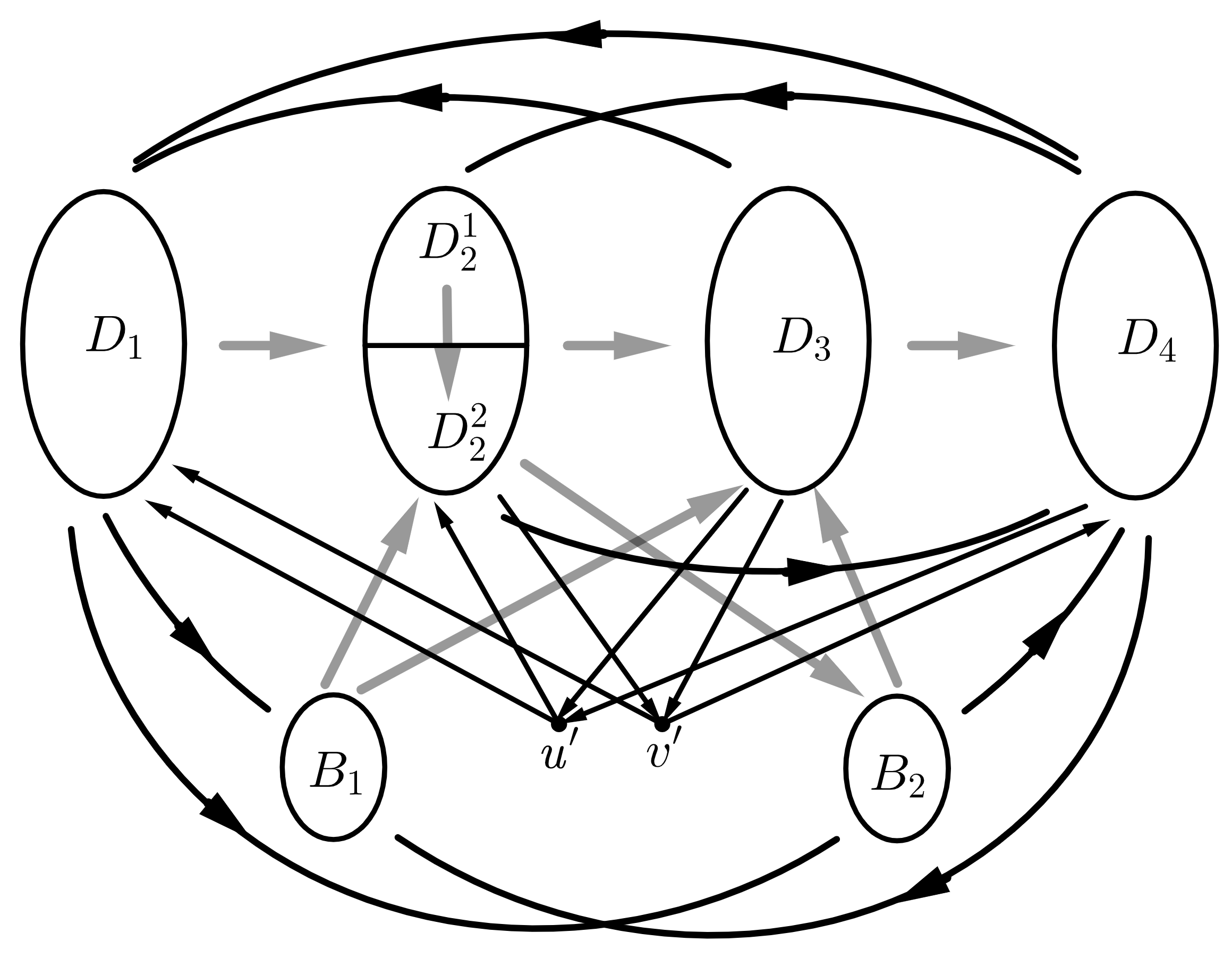}
\caption{An oriented graph $D$ on $n$ vertices with $\delta^0(D)\geq (3n+3h+3q-13)/8$ that is not arbitrary Hamiltonian $H$-linked.}
\label{tu1}
\end{figure}

Proposition \ref{prop2} explains the necessity of restricting subdivision path lengths to at least 4.

\begin{proposition}\label{prop2}
There exists a family of oriented graphs \(D\) on \(n\) vertices with \(\delta^0(D) \geq {13n/32 - 3}\) that do not contain a 3-path form $u$ to $v$ for some two vertices \(u,v \in V(D)\).
\end{proposition}
\begin{proof}
To illustrate this, we define an oriented graph \(D\) on \(n\) vertices as follows (see Fig. \ref{tu2}). Let \(D_1\), \(D_2\), \(D_3\), and \(D_4\) be disjoint vertex sets with \(|D_1| = 5n/16\), \(|D_2| = n/8\), \(|D_3| = 5n/16\), and \(|D_4| = n/4 - 2\). Add arcs such that \(D_1 \rightarrow D_2 \rightarrow D_3 \rightarrow D_4 \rightarrow D_1\) and \(D_1 \rightarrow D_3\). Each of \(D_1\), \(D_3\), and \(D_4\) {induces a tournament as regular as possible}, and \(D_2\) induces an independent set. Additionally, let \(D[D_2, D_4]\) be a bipartite tournament such that each vertex in \(D_2\) has at least \(n/8 - 1\) out-neighbours and in-neighbours in \(D_4\). Add two vertices \(u\) and \(v\) such that \(N_D^{+}(u) = D_2 \cup D_3\), \(N_D^{-}(u) = D_1 \cup D_4\), \(N_D^{+}(v) = D_3 \cup D_4\), and \(N_D^{-}(v) = D_1 \cup D_2\). Namely, $d^+_D(u), d^-_D(u)\geq 14n/32$ and $d^+_D(v), d^-_D(v)\geq 14n/32$.

Now, we verify that {$\delta^0(D)\geq 13n/32 - 3$} from the following inequalities:
\begin{itemize}[itemsep=1pt, topsep=2.5pt,parsep=2.5pt]
\item $d^{\pm}_D(D_1)\geq \min \{{\frac{|D_1|-2}{2}+|D_2\cup D_3\cup \{u,v\}|}, \ {\frac{|D_1|-2}{2}}+|D_4|\}={\frac{13n }{32}-3},$
\item $d^{\pm}_D(D_2)\geq \min  \{ {|D_3\cup \{v\}|}+n/8-1, {|D_1\cup \{u\}|}+n/8-1\}=\frac{7n}{16},$
\item $d^{\pm}_D(D_3)\geq \min  \{{\frac{|D_3|-2}{2}}+|D_4|, \ {\frac{|D_3|-2}{2}+|D_1\cup D_2\cup \{u,v\}|}\}={\frac{13n}{32}-3},$
\item $d^{\pm}_D(D_4)\geq \min  {\{\frac{|D_4|-2}{2}+|D_1\cup \{u\}|, \ \frac{|D_4|-2}{2}+|D_3\cup \{v\}|\}=\frac{7n}{16}-1}.$
\end{itemize}
By construction, it is easy to check that \(D\) contains no \((u,v)\)-path of length 3.
\end{proof}
\begin{figure}[H]
\centering
   \includegraphics[scale=0.75]{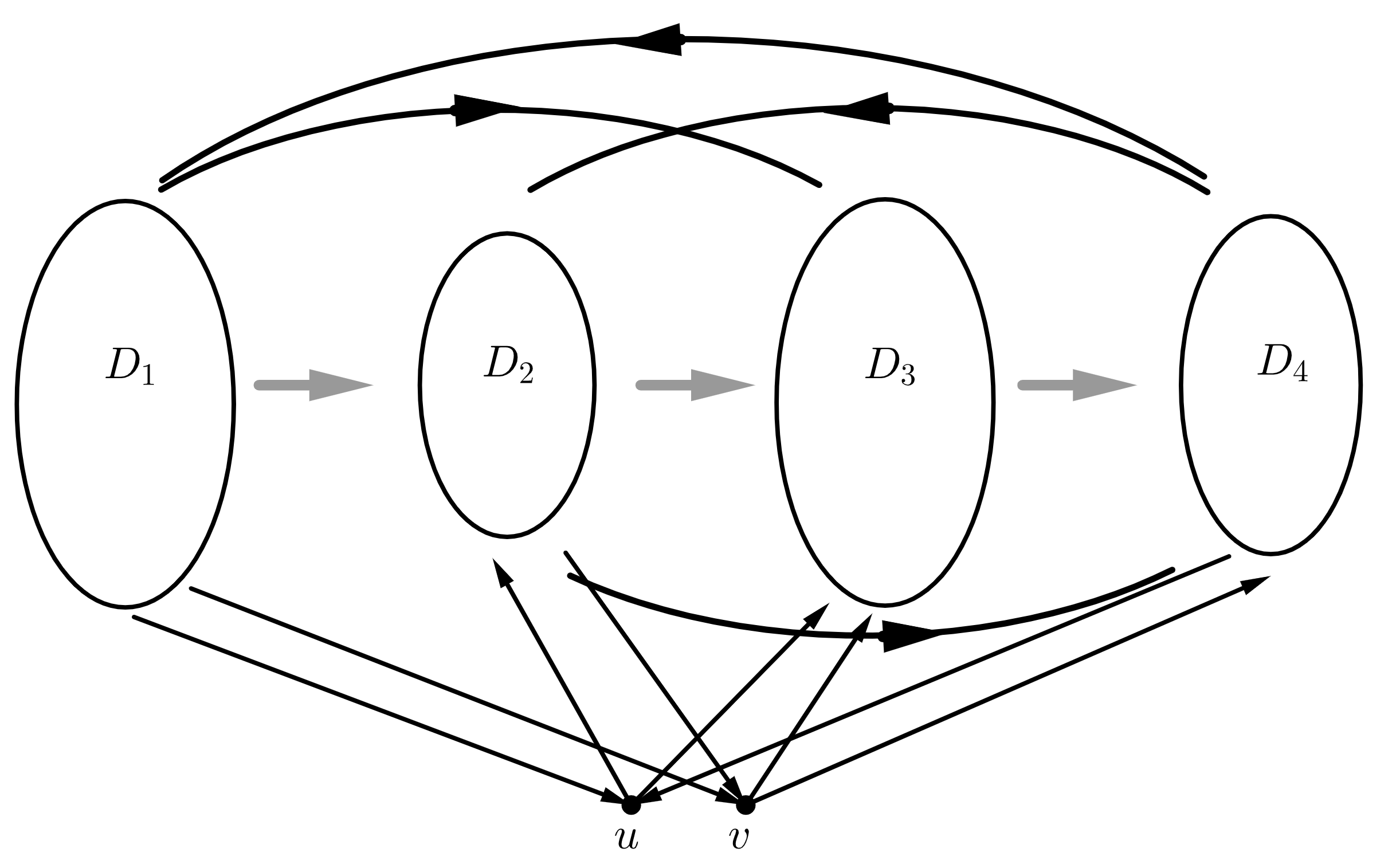}   
\caption{An oriented graph $D$ on $n$ vertices with {$\delta^0(D)\geq 13n/32 - 3$} that contains no $(u,v)$-path with length 3.}
\label{tu2}
\end{figure}

\section{Preliminaries}

\subsection{Short paths with prescribed length}
In this subsection, we prove that for $4 \leq l \leq n/10^{10}$, there is a path with length $l$ between any two vertices {in an oriented graph \(D\) of order \(n\) with \(\delta^0 (D) \geq n/3 + 6\)}. This guarantees the existence of short subdivision paths in the non-extremal case of the proof of Theorem \ref{main}.

Lemma \ref{le3} shows the existence of a path of length 4 with specified endvertices {in \(D\) under the weaker minimum semi-degree condition}.

\begin{lemma}\label{le3}
Let \(D\) be an oriented graph on \(n\) vertices with \(\delta^0 (D) \geq n/3 + 2\). Then \(D\) contains an \((x,y)\)-path of length 4 for any two vertices \(x, y \in V(D)\).
\end{lemma}

\begin{proof}
We arbitrarily select vertices $x$ and $y$ from $V(D)$ and show that there is an $(x,y)$-path of length 4. Let $X$ be a set consisting of $\frac{n}{3} + 1$ out-neighbours of $x$ in $D\setminus \{y\}$, and let $Y$ be a set consisting of $\frac{n}{3} + 1$ in-neighbours of $y$ in $D\setminus \{x\}$. {By contradiction}, suppose that \reducebelowdisplayskip \reduceabovedisplayskip
\begin{equation}\label{s4}
\text{there is no $(x, y)$-path of length 4.}
\end{equation}
This means that
\begin{equation}\label{s2}
\text{there are no vertices } u\in X,v\in Y\text{ such that } |N^+_D(u) \cap N^-_D(v)|\geq 3.
\end{equation}
Since otherwise, {for any vertex $w\in (N^+_D(u) \cap N^-_D(v))\setminus \{x,y\}$, $xuwvy$ is an $(x,y)$-path of length 4}. {Recall that \(|X|=|Y|=\frac{n}{3}+1\) and \(\delta^0(D)\ge\frac{n}{3}+2\). Consider any vertex \(u\in X\), if \(D[X]\) is arcless, then all neighbours of $u$ must in \(V(D)\setminus X\), but \(|V(D)\setminus X|=\frac{2n}{3}-1<2(\frac{n}{3}+2)\leq d^+_D(u)+d^-_D(u)\), a contradiction. Thus \(D[X]\) has an arc, the same holds for \(D[Y]\). Namely,}
\begin{equation}\label{s3}
\text{both $D[X]$ and $D[Y]$ have an arc.}
\end{equation}
Let $Z := X \cap Y$. Next, our proof is divided into three cases according to the {size} of $Z$.

\textbf{Case 1. $\boldsymbol{|Z|= n/3+1}$, that is, $\boldsymbol{X=Y= Z}$.}

By (\ref{s3}), $D[Z]$ contains an arc, say $uv$. Consider the out-neighbours of $v$ and the in-neighbours of $u$. According to the assumption (\ref{s4}), we have $N^+_D(v) \cap Z = \emptyset$ and $N^-_D(u) \cap Z = \emptyset$. Hence $|N^+_D(v)\cup N^-_D(u)|\leq n-|Z|=2n/3-1$. By a simple calculation,
$$
|N^-_D(u) \cap N^+_D(v)| = |N^-_D(u)| + |N^+_D(v)| - |N^-_D(u) \cup N^+_D(v)| \geq 2\delta^0(D) - 2n/3 + 1 \geq 5,
$$
where $u\in X\cap Z, v\in Y\cap Z$. This contradicts (\ref{s2}).

\textbf{Case 2. $\boldsymbol{3 \leq |Z| \leq \frac{n}{3}}$.}

In this case, we first assert that either $D[X \setminus Z]$ or $D[Y \setminus Z]$ has no arc. {By contradiction}, assume that both $D[X \setminus Z]$ and $D[Y \setminus Z]$ have arcs. We prove the existence of a vertex \( x' \in X \setminus Z \) satisfying \( d^+_{X \setminus Z}(x') \leq \frac{|X \setminus Z|}{2} \) and {\( d^-_{X \setminus Z}(x')\geq 1 \)}. Suppose, for contradiction, that no such vertex exists. Then, all vertices with out-degree \( \leq \frac{|X \setminus Z|}{2} \) induce an arcless subdigraph. Let \( x' \in X \setminus Z \) be {a} vertex with the minimum out-degree; thus, \( d^+_{X \setminus Z}(x') \leq \frac{|X \setminus Z|}{2} \) and {\( d^-_{X \setminus Z}(x')=0 \)}. Consequently, there exists another vertex \( x'' \in X \setminus (Z \cup \{x'\}) \) satisfying \( d^+_{X \setminus Z}(x'') \leq \frac{|X \setminus Z|}{2} \) and {\( d^-_{X \setminus Z}(x'')=0 \)}. Repeating this process, {we can extract a maximal subset \(S\subseteq X\setminus Z\) where every vertex $s\in S$ has \(d^+_{X\setminus Z}(s)\leq \frac{|X\setminus Z|}{2}\) and \(d^-_{X\setminus Z}(s)=0\), so $S$ is arcless. Since \(D[X\setminus Z]\) contains at least one arc, we obtain that $(X\setminus Z)\setminus S\neq \emptyset$. Hence, there is a vertex $x'''\in (X\setminus Z)\setminus S$ with $d^+_{X\setminus Z}(x''')\leq \frac{|(X\setminus Z)\setminus S|}{2}$, and so \( d^-_{X \setminus Z}(x''')=0 \), which contradicts the maximality of $S$. Therefore, there must exist a vertex \( w \in X \setminus Z \) with \( d^+_{X \setminus Z}(w) \leq \frac{|X \setminus Z|}{2} \) and an in-arc \( uw \) of $w$ in \( D[X \setminus Z] \).} Similarly, since there are arcs within $D[Y \setminus Z]$, there exists a vertex $u' \in Y \setminus Z$ whose in-degree is less than $\frac{|Y \setminus Z|}{2}$ and which has out-arcs within $D[Y \setminus Z]$, say $u'w'$.

The fact (\ref{s4}) indicates that $N^+_D(w) \cap Y = \emptyset$ and $N^-_D(u') \cap X = \emptyset$. Recall that $d^+_{X \setminus Z}(w) \leq \frac{|X \setminus Z|}{2}$, $d^-_{Y \setminus Z}(u') \leq \frac{|Y \setminus Z|}{2}$, and $|X|=|Y|$. In addition,
$$
\begin{aligned}
n &\geq |(X\cup Y)\setminus (N^+_D(w)\cup N^-_D(u'))| + |N^+_D(w)\cup N^-_D(u')| \\
&\geq |X| + |Y \setminus Z| - \frac{|X \setminus Z| + |Y \setminus Z|}{2} + |N^+_D(w)\cup N^-_D(u')| \\
&\geq |X| + |N^+_D(w)\cup N^-_D(u')| \\
&\geq |X| + |N^+_D(w)| + |N^-_D(u')| - |N^+_D(w)\cap N^-_D(u')| \\
&\geq 3(n/3 + 1) - |N^+_D(w)\cap N^-_D(u')|.
\end{aligned}
$$
This implies that $|N^+_D(w) \cap N^-_D(u')| \geq 3$, which contradicts (\ref{s2}). Hence, either $D[X \setminus Z]$ or $D[Y \setminus Z]$ contains no arc. W.l.o.g., assume that there is no arc in $D[X \setminus Z]$.

We also claim that there is no arc from $Z$ to $X \setminus Z$. Suppose that there is such an arc, say $uv$. According to the fact that $D[X \setminus Z]$ is arcless and the assumption (\ref{s4}), we can conclude that $N^+_D(v) \cap (X \cup Y) = \emptyset$. Moreover, (\ref{s3}) yields an arc in $D[Y]$, say $wz$. Then, by (\ref{s4}), we have $N^-_D(w) \cap X = \emptyset$. Consequently, we can calculate that
$$
\begin{aligned}
n &\geq |(X\cup Y)\setminus (N^+_D(v)\cup N^-_D(w))| + |N^+_D(v)\cup N^-_D(w)| \\
&={|X| + |(Y\setminus Z)\setminus ( N^+_D(v)\cup N^-_D(w))| + |N^+_D(v)\cup N^-_D(w)|}\\
&\geq |X| + |Y \setminus Z| - |Y \setminus Z| + |N^+_D(v)\cup N^-_D(w)| \\
&\geq |X| + |N^+_D(v)\cup N^-_D(w)| \\
&\geq |X| + |N^+_D(v)| + |N^-_D(w)| - |N^+_D(v)\cap N^-_D(w)| \\
&\geq 3(n/3 + 1) - |N^+_D(v)\cap N^-_D(w)|.
\end{aligned}
$$
This implies that $|N^-_D(w) \cap N^+_D(v)| \geq 3$, which contradicts (\ref{s2}).

{Similarly, $D[Z]$ is also arcless, since otherwise, there is an arc $uv$ in $D[Z]$. According to the fact that there is no arc from $Z$ to $X \setminus Z$ and the assumption (\ref{s4}), we can conclude that $N^+_D(v) \cap (X \cup Y) = \emptyset$ and $N^-_D(u) \cap X = \emptyset$. Similarly, we can obtain that $|N^-_D(u) \cap N^+_D(v)| \geq 3$, which contradicts (\ref{s2}).}  Now, both $D[X \setminus Z]$ and $D[Z]$ are arcless, and there is no arc {from $Z$ to $X \setminus Z$}. Therefore, there must be an arc $uv$ from $X \setminus Z$ to $Z$ by (\ref{s3}). Together with (\ref{s4}), $N^+_D(v)\cap (X\cup Y) = \emptyset$. Also, since $D[Y]$ has an arc, say $wz$, we have $N^-_D(w) \cap X = \emptyset$. Likewise, we can conclude that $|N^+_D(v) \cap N^-_D(w)| \geq 3$, which leads to a contradiction.

\textbf{Case 3. $\boldsymbol{Z}$ has at most two vertices.}

In this case, $|V(D)\setminus (X\cup Y)| = n - |X| - |Y| + |Z| \leq n/3$. Suppose that both of the following conditions hold:
\begin{itemize}[itemsep=0pt, topsep=0.5pt, parsep=1pt]
  \item[(i)] There exists $x' \in X$ such that $\left|N^+_D\left(x'\right) \setminus (X \cup Y)\right| \geq \frac{n/3 + 1}{2}$.
  \item[(ii)] There exists $y' \in Y$ such that $\left|N^-_D\left(y'\right) \setminus (X \cup Y)\right| \geq \frac{n/3 + 1}{2}$.
\end{itemize}
If these conditions are met, then $|N^{+}_D\left(x'\right) \cap N^-_D\left(y'\right)| \geq 3$, which contradicts (\ref{s2}). So w.l.o.g., assume that (i) does not hold. (The case when (ii) does not hold is similar.) Let $X'$ be the set of vertices $x' \in X$ with $d_X^-\left(x'\right) > 0$. Note that $X' \neq \emptyset$ since $D$ is an oriented graph. Let $x' \in X'$ be a vertex such that $d_{X'}^+\left(x'\right)$ is minimal. Since $N^{+}_D\left(x'\right) \cap (X \setminus X') = \emptyset$, the minimality implies that
$$
\left|N^{+}_D\left(x'\right) \backslash X'\right| \geq \delta^0(D) - \frac{|X'|}{2} \geq \delta^0(D) - \frac{|X|}{2} \geq \frac{n/3 + 1}{2},
$$
a contradiction. This completes the proof.
\end{proof}
\begin{figure}[H]
\centering
   \includegraphics[scale=0.75]{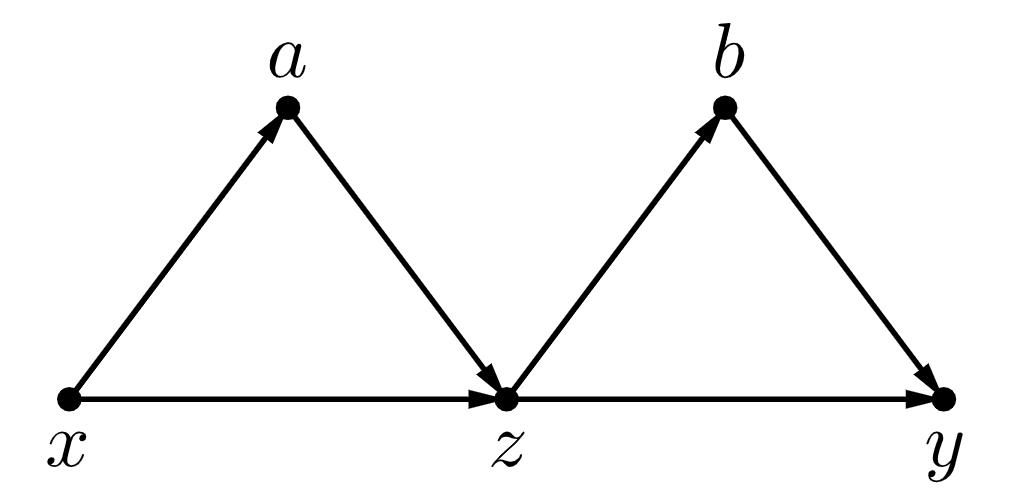}
\caption{An \(xy\)-butterfly.}
\label{tu4}
\end{figure}

To find a path with specified endvertices of prescribed length $l \geq 7$, we introduce the following special structure, {which has already been utilized in \cite{Kelly(2009)}}. An \textbf{\(xy\)-butterfly} is an oriented graph with vertices \(x, y, z, a, b\) and arcs \(xa, xz, az, zb, zy, by\), see Fig. \ref{tu4}. It is easy to check that there are three $(x,y)$-paths of length 2,3,4, respectively in an \(xy\)-butterfly. This structure plays a crucial role in adjusting the length of a path.

\begin{fact} \label{fac2}
\cite{Kelly(2009)} If \(D\) is an oriented graph on \(n\) vertices with \(\delta^0(D) \geq n/3 + 1\), then for any vertex \(x \in V(D)\), there exists a vertex \(y\) such that \(D\) contains an \(xy\)-butterfly.
\end{fact}

{Note that the vertex $y$ in Fact \ref{fac2} is existential rather than arbitrary, and thus Lemma \ref{sublemma1} is essential for constructing a path with specified endvertices.}
\begin{lemma} \label{sublemma1}
\cite{Kelly(2009)} Let \(C\) be a positive integer. If \(D\) is an oriented graph on \(n \geq 8 \cdot 10^9 C\) vertices with \(\delta^0(D) \geq n/3 - C + 1\), then for any pair of distinct vertices \(s\) and \(t\), there exists an \((s,t)\)-path of length 3, 4, or 5.
\end{lemma}

We now establish Lemma \ref{lemma3}, which guarantees the existence of short path with specified length between any two vertices.
\begin{lemma}[Short-path existence lemma]\label{lemma3}
For \(l \geq 4\), if \(D\) is an oriented graph on \(n \geq 10^{10}l\) vertices with \(\delta^0 (D) \geq n/3+6\), then \(D\) contains an \((x,y)\)-path of length \(l\) for any two vertices \(x, y \in V(D)\) (\(x\) and \(y\) are not necessarily distinct).
\end{lemma}
\begin{proof}
Let \(D\) be an oriented graph on \(n\) vertices with \(\delta^0(D) \geq n/3 + 6\), and let \(x, y \in V(D)\) be distinct vertices. We aim to construct an \((x,y)\)-path of length \(l\).

For \(l \leq 6\), by the minimum semi-degree condition, we can construct a path \(Q\) of length \(l-4 \leq 2\) from \(x\) to some vertex \(v\) avoiding \(y\). Applying Lemma \ref{le3} to \((D \setminus V(Q)) \cup \{v\}\), we find a \((v,y)\)-path of length 4, which combined with \(Q\) yields an \((x,y)\)-path of length \(l\).

For \(l \geq 7\) and \(n \geq 10^{10}l\), by Fact \ref{fac2}, there exists an \(xu\)-butterfly {with vertex set $\{a,b,z,x,u\}$} for some vertex \(u \in V(D) \setminus \{y\}\). We greedily construct a path \(P\) of length \(l-7\) from \(u\) to some vertex \(v\), avoiding \(a, b, z, x, y\). The minimum semi-degree condition ensures the existence of such a path. Next, we apply Lemma \ref{sublemma1} to \(D - (\{a, b, z, x, u\} \cup (V(P) \setminus \{v\}))\) with \(C = l-7\), {$s:=v$ and $t:=y$}. Since \(n \geq 10^{10}l\), the conditions \(n - l - 2 \geq 8 \cdot 10^9 C\) and $\delta ^0(D  \setminus (\{a, b, z, x, u\} \cup V(P))\cup \{v\})\geq (n-C)/3-C+1 $ hold, and thus we obtain a \((v,y)\)-path of length 3,4 or 5. Finally, we select a path from \(x\) to \(u\) in the \(xu\)-butterfly of appropriate length \(k \in \{2, 3, 4\}\) such that \(k + (l - 7) + m = l\), where \(m \in \{3, 4, 5\}\) is the length of the \((v,y)\)-path. This completes the construction of an \((x,y)\)-path of length \(l\).
\end{proof}

\subsection{Partitioning tools}

In this subsection, we introduce two partitioning tools. One of them is directed version of Szemer\'{e}di Regular Lemma. The density of a bipartite graph $G$ with vertex classes $A$ and $B$ is defined to be $\boldsymbol{d_G(A, B)}= \frac{e(A, B)}{|A||B|}$. We often write $\boldsymbol{d(A, B)}$ if this is unambiguous. Given $\varepsilon >0$, we say that $G$ is \textbf{$\varepsilon $-regular} if for all subsets $X\subseteq A$ and $Y\subseteq B$ with $|X|>\varepsilon |A|$ and $|Y|>\varepsilon |B|$ we have that $|d(X, Y )-d(A, B)| < \varepsilon$. Given $d\in [0, 1]$ we say that $G$ is \textbf{$\boldsymbol{(\varepsilon, d)}$-super-regular} if it is $\varepsilon$-regular and furthermore $d_G(a)\geq (d-\varepsilon)|B|$ for all $a\in A$ and $d_G(b)\geq (d-\varepsilon)|A|$ for all $b\in B$.

In 1978, Szemer\'{e}di proposed Regular Lemma on graphs, and later Alon and Shapira \cite{Alon(2004)} extended it to the digraph version.

\begin{lemma}\label{lem1}
\cite{Alon(2004)} \textbf{\emph{(}Diregularity Lemma\emph{)}} For every $\varepsilon \in (0, 1)$ and $M'>0$ there are numbers $M(\varepsilon,M')$ and $n_0$ such that if $D$ is a digraph on $n\geq n_0$ vertices, and $d\in [0, 1]$ is any real number, then there is a partition of the vertices of $D$ into $V_0, V_1,\ldots , V_l$ and a spanning subdigraph $D'$ of $D$ such that the following holds:
\begin{itemize}[itemsep=0pt, topsep=0.5pt,parsep=1pt]
\item $|V_0|\leq \varepsilon n$, $|V_1| = \cdots = |V_l|$ with $M' \leq l \leq M$,
\item for each $\sigma \in \{+,-\}$, $d^{\sigma}_{D'}(x) > d^{\sigma}_D(x)- (d + \varepsilon)n$ for all vertices $x\in D$,
\item for all $i\in [l]$ the digraph $D'[V_i]$ is arcless,
\item for all $1\leq i, j\leq l$ with $i \neq j$ the bipartite graph, whose vertex classes are $V_i$ and $V_j$ and whose arcs are all directed from $V_i$ to $V_j$ in $D'$, is $ \varepsilon$-regular and has density either $0$ or density at least $d$.
\end{itemize}
\end{lemma}

Given $(V_0, V_1,\ldots , V_M)$ and the spanning digraph $D'$ given by the Diregularity lemma, the \textbf{reduced digraph} $R'$ with parameters $(\varepsilon, d)$ is the digraph whose vertex set is $[M]$ and in which $ij$ is an arc if and only if the bipartite digraph whose vertex classes are $V_i$ and $V_j$ and whose arcs are all the arcs from $V_i$ to $V_j$ in $D'$ is $\varepsilon$-regular and has density at least $d$. Note that $R'$ is not necessarily an oriented graph even if $D$ is. The next lemma shows that there is a reduced spanning oriented {subgraph} of $ R'$, which still almost inherits the minimum semi-degree of $R'$.

\begin{lemma}\label{lem2}
\cite{Keevash(2009)} For every $\varepsilon\in (0, 1)$, there exist numbers $M' = M'(\varepsilon)$ and $n_0 = n_0(\varepsilon)$ such that the following holds. Let $d\in [0, 1]$ with $\varepsilon \leq d/2$. Let $D$ be an oriented graph of order $n\geq n_0$ and let $R'$ be the reduced digraph with parameters $(\varepsilon, d)$ obtained by applying the Diregularity lemma to $D$ with $M'$ as the lower bound on the number of clusters. Then $R'$ has a spanning oriented subgraph $R$ such that $\delta^0(R)\geq (\delta^0(D)/|D|- (3\varepsilon + d))|R|$.
\end{lemma}

We present another partitioning tool by generalizing the results of Keevash and Sudakov \cite{sudakov(2009)} to multi-part random partitions, namely Lemma \ref{lemma2}, which is particularly useful for optimizing vertex allocation in long subdivision paths. This process relies on the following essential lemma.

\begin{lemma}\label{chva} (Chv\'{a}tal \cite{chva}). Let
$$
F(n, \gamma, k, \delta)=\binom{n}{k}^{-1} \sum_{i=0}^{(\gamma-\delta) k}\binom{\gamma n}{i}\binom{(1-\gamma) n}{k-i}
$$
denote the probability that a random subset of $k$ elements out of $n, \gamma n$ of which are marked, contains at most $(\gamma-\delta) k$ marked elements; then
$$
F(n, \gamma, k, \delta) \leqslant \mathrm{e}^{-2 \delta^2 k}.
$$
\end{lemma}

In \cite{Alon(1996)}, Alon and Fischer gave a random partition of a graph such that each vertex has a proportional distribution of neighbours in each part. In \cite{sudakov(2009)}, Keevash and Sudakov provided the corresponding digraph version. Lemma \ref{key1} further generalizes these results by offering a random partition of a large subset $U$ such that each vertex with large neighbours in $U$ has a proportional distribution of neighbours in each part.

\begin{lemma}\label{key1}
Let $\varepsilon$ and $ \beta$ {be} two real numbers such that $0< \varepsilon \leq  \beta <1$. There exists a constant $n_0 = n_0(\varepsilon)$ such that if $D$ is an oriented graph on $n > n_0$ vertices with a subset $U\subseteq V(D)$ of order at least $\beta n$, and $m,l$ are two integers satisfying $m \geqslant \varepsilon |U|, l \geqslant \varepsilon |U|$ and $m + l = |U|$, then there exists a partition of $U$ into two sets $A, B$ of orders $m, l$, respectively, such that for any vertex $v\in V(D)$ with $d^+_{U}(v)\geqslant \beta |U|$, we have $d^+_{A}(v)\geqslant \beta m- n^{2/3} $ and  $d^+_{B}(v)\geqslant \beta l- n^{2/3} $. Moreover, for any vertex $v\in V(D)$ with $d^-_{U}(v)\geqslant \beta |U|$, we have  $d^-_{A}(v)\geqslant (\beta - n^{-1 / 3}) m$  and $d^-_{B}(v)\geqslant (\beta - n^{-1 / 3}) l$.
\end{lemma}

\begin{proof}
We randomly choose a partition of the vertices of \(U\) into sets \(A\) and \(B\) with orders \(m\) and \(l\), respectively, such that each possible partition is equally probable. Let \(v\) be a vertex in \(V(D)\) with \(d^+_{U}(v) \geq \beta |U|\). Applying Lemma \ref{chva} with \(n := |U|, \gamma :=\beta, k:=m,\) and \( \delta:=n^{-1/3}\), the probability that fewer than \(\left(\beta - n^{-1 / 3}\right) m\) out-neighbours of \(v\) lie in \(A\) is at most \(\mathrm{e}^{-2 n^{-2 / 3} m} \leq \mathrm{e}^{-2 \varepsilon \beta n^{1 / 3}}\leq \mathrm{e}^{-2 \varepsilon ^2  n^{1 / 3}}\). A similar estimate applies to the out-neighbours of \(v\) in \(B\). Thus, the probability that any of these conditions is not satisfied for some vertex in \(V(D)\) is at most \(2 n \mathrm{e}^{-2 \varepsilon^2 n^{1 / 3}}\). We can choose \(n_0\) such that for \(n > n_0\), this probability is less than \(1/2\), ensuring the existence of the required partition.
\end{proof}

By repeatedly applying Lemma \ref{key1}, we can obtain the following more generalized result. This provides a random partition of a digraph such that each vertex has a proportional distribution of in- and out-neighbours in each part.

\begin{lemma}\label{lemma2}
Let $\varepsilon, \beta$ be two positive real numbers such that $0< \varepsilon \leq  \beta <1$, and let $s$ with $s\leq 1/\varepsilon$  and $M:=M(\varepsilon)$ be two positive integers. There exists a constant $n_0 = n_0(s,M,\varepsilon)$, such that the following conclusion holds. {Let $D$ be} an oriented graph on $n > n_0$ vertices with a partition $(V_0,V_1,V_2,\ldots, V_M)$, where each $V_i$ $(i\in [M])$ satisfies $| V_i|\geqslant n/2M$. Let $m^i_l, i\in [s], l\in [M]$, be integers satisfying $ m^i_l\geq \varepsilon |V_l|$ for all $m^i_l$ and $  \sum_{i=1}^s m^i_l=\vert V_l\vert $ for all $l\in [M]$. Then there exists a partition $(S^1_l,\ldots,S^{s}_l)$ of each part $V_l, l\in [M]$, satisfying
\begin{itemize}[itemsep=0pt, topsep=0.5pt,parsep=1pt]
\item[(i)] $\vert S^{i}_l\vert = m^i_l$ for $i\in [s] $;
\item[(ii)] For every vertex $v\in V(D)$ with $d^+_{V_l}(v)\geqslant \beta |V_l|$ (resp., $ d^-_{V_l}(v)\geqslant \beta |V_l|$), the number of out-neighbours (resp. in-neighbours) of $v$ in $S^{i}_l$ is at least $\beta m^i_l - 2sn^{2/3}$.
\end{itemize}
\end{lemma}

\begin{proof}
For each \( l \in [M] \), we iteratively partition \( V_l \) into \( s \) subsets \( S^1_l, \ldots, S^s_l \) as follows: Initially, set \( R^0_l := V_l \). For each \( i = s, s-1, \ldots, 2 \), apply Lemma \ref{key1} to the subset \( U := R^{s - i}_{l} \) with parameters
     \[
     m := m^i_l, \quad l := \sum_{j=1}^{i-1} m^j_l, \quad \beta := \beta, \quad \varepsilon := \varepsilon,
     \]
we obtain a partition \( (R^{s - i + 1}_{l}, S^i_l,\ldots,S^s_l) \), where all \( |S^i_l| = m^i_l \) and $|R^{s - i + 1}_{l}|=\sum_{j=1}^{i-1} m^j_l$. Update $i:=i-1$. Repeat the above process until \( i = 1 \). At this point, \( V_l \) has a partition $(R^{s -1}_{l}, S^2_l,\ldots,S^s_l )$, and let \( S^1_l := R^{s-1}_{l} \). This process ensures \( |S^i_l| = m^i_l \) for all \( i \in [s] \).

In this process, each application of Lemma \ref{key1} introduces an additive error of at most \( n^{2/3} \) in the number of out-neighbours (resp., in-neighbours) for any vertex \( v \). Since the partition is constructed via \( s - 1 \) iterations, the total error accumulates to \( 2(s - 1)n^{2/3} \), which is bounded by \( 2sn^{2/3} \). We thus verify (ii). For any vertex \( v \) with \( d^+_{V_l}(v) \geq \beta |V_l| \), the number of out-neighbours of $v$ in each \( S^i_l \) is at least $\beta m^i_l - 2sn^{2/3}.$ Similarly, for in-neighbours with \( d^-_{V_l}(v) \geq \beta |V_l| \), the same bound holds.
\end{proof}

\section{Proof of Theorem \ref{main}}
In this section, we prove Theorem \ref{main}. Before that, we introduce necessary definitions and lemmas.
\subsection{Preliminaries}
To utilize the extremal structure proposed by Keevash, K\"{u}hn, and Osthus \cite{Keevash(2009)}, we adopt their notation. {Before discussing the extremal case, we need} the following notation. Let $\alpha <1$ be a real number. Consider an oriented graph \(D\) with a vertex partition \((D_1, D_2, D_3, D_4)\). {Given a vertex \( x \in V(D) \), we will use the compact notation \( x \) has property \( W: (D_1)_{\tau_1}^{\sigma_1} (D_2)_{\tau_2}^{\sigma_2} (D_3)_{\tau_3}^{\sigma_3} (D_4)_{\tau_4}^{\sigma_4} \) as follows. The notation starts with some \( W \in \{D_1,D_2,D_3,D_4\} \), namely the set that \( x \) belongs to. Next, for each of \( D_i, i\in [4] \), its superscript marker \(\sigma_i\) describes the size of the intersection \(D_i \cap N^+_D(x)\), and the subscript marker \(\tau_i\) describes the size of the intersection \(D_i \cap N^-_D(x)\). A symbol \(>\alpha\) indicates that the corresponding intersection has size at least \(\alpha |D_i|\), while a symbol \(<\alpha\) indicates that the intersection has size at most \(\alpha |D_i|\). The absence of a symbol means that no statement is made about that particular intersection, and we can omit any of \( D_1,D_2,D_3,D_4 \) if it has no superscript marker or subscript marker. For example, to say \( x \) has property \( D_2:(D_1)^{>1/4}(D_3)_{<1/3}^{<1/2} \) means that \( x \in D_2 \), \( |N^+_D(x)\cap D_1| > |D_1|/4 \), \( |N^+_D(x) \cap D_3| < |D_3|/2 \) and \( |N^-_D(x) \cap D_3| < |D_3|/3. \)} Based on this notation, a vertex in $D$ is defined to be $\boldsymbol{\alpha}$-\textbf{acceptable} if it satisfies one of the following 16 properties:
\begin{itemize}[itemsep=0pt, topsep=0.5pt,parsep=1pt]
\item $D_1:(D_2)^{>\alpha}(D_4)_{>\alpha}, D_1:(D_1)^{>\alpha}(D_4)_{>\alpha}, D_1:(D_1)_{>\alpha}(D_2)^{>\alpha}, D_1:(D_1)^{>\alpha}_{>\alpha},$
\item $D_2:(D_1)_{>\alpha}(D_3)^{>\alpha}, D_2:(D_1)_{>\alpha}(D_4)^{>\alpha}, D_2:(D_3)^{>\alpha}(D_4)_{>\alpha}, D_2:(D_4)^{>\alpha}_{>\alpha},$
\item $D_3:(D_2)_{>\alpha}(D_4)^{>\alpha}, D_3:(D_2)_{>\alpha}(D_3)^{>\alpha}, D_3:(D_3)_{>\alpha}(D_4)^{>\alpha}, D_3:(D_3)^{>\alpha}_{>\alpha},$
\item $D_4:(D_1)^{>\alpha}(D_3)_{>\alpha}, D_4:(D_1)^{>\alpha}(D_2)_{>\alpha}, D_4:(D_2)^{>\alpha}(D_3)_{>\alpha}, D_4:(D_2)^{>\alpha}_{>\alpha}.$
\end{itemize}
{In other words, a vertex is acceptable if it has a significant out-neighborhood in one of its two out-classes and a significant in-neighborhood in one of its two in-classes, where `out-classes' and
`in-classes' are to be understood with reference to the extremal partition in Fig. \ref{tu3}.}

\begin{definition}\label{def2}
Let \(c \leq 1\) be a real number. A vertex \(x\) is \textbf{\(c\)-circular}  in $D$ if \(x \in D_i\), \(x\) has at least \(c|D_{i+1}|\) out-neighbours in \(D_{i+1}\), and \(x\) has at least \(c|D_{i-1}|\) in-neighbours in \(D_{i-1}\) for some $i\leq 4$ (counting modulo 4).
\end{definition}

It is worth noting that the concept of $c$-circular is closely related to $\alpha$-acceptability. Given a partition, a $c$-circular vertex is $\alpha$-acceptable if $c\geq \alpha$. Now we give the definition of extremal case.

\begin{definition}\label{def4}
A digraph \(D\) is an \textbf{extremal case with parameters \((c_1, c_2, c_3, c_4)\)} if there exists a partition \((D_1, D_2, D_3, D_4)\) of \(V(D)\) such that:
\begin{itemize}[itemsep=0pt, topsep=0.5pt,parsep=1pt]
\item[$(i)$] $|D_i|=(1/4\pm c_1\mu)n$, for $i\in [4]$.
\item[$(ii)$] $a(D_i, D_{i+1}) > (1-c_2\mu)n^2/16$, for $i\in [4]$ $(mod \  4)$; $a(D_i)>(1/2- c_2\mu)n^2/16$, for $i\in \{1,3\}$;  $a(D_i, D_j) > (1/2- c_2\mu)n^2/16$, for $i,j\in \{2,4\}$ with $i\neq j$; $a(D_4) + a(D_4, D_3)\leq c_2\mu n^2$;
\item[$(iii)$] Every vertex of $D$ is \(c_3\)-acceptable, and the number of vertices that are not \(c_4\)-circular is at most \(\sqrt{\mu} n\) in \(D\).
\end{itemize}
Moreover, we also call the partition $(D_1,D_2,D_3,D_4)$ the \textbf{extremal partition with parameters $(c_1, c_2, c_3, c_4)$}.
\end{definition}

Generally speaking, an extremal oriented graph has an extremal partition $(D_1,D_2,D_3,D_4)$ where the four parts are almost equal in order, both $D_1$ and $D_3$ almost induce tournaments, and the in-neighbours and out-neighbours of all vertices have specific distributions, see Fig. \ref{tu3}.

\begin{figure}[H]
\centering
   \includegraphics[scale=0.75]{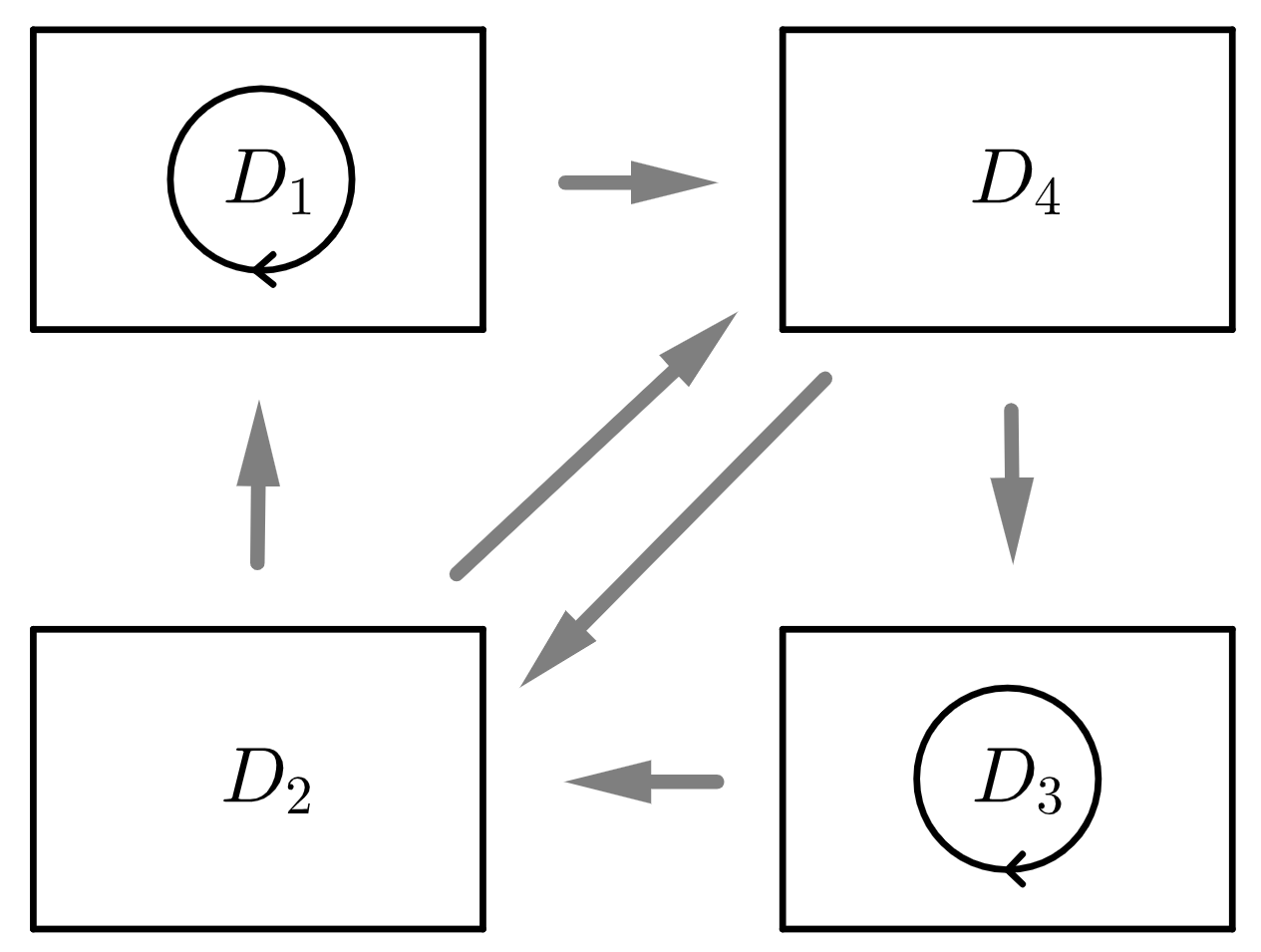}
\caption{An extremal partition \((D_1, D_2, D_3, D_4)\).}
\label{tu3}
\end{figure}
The following lemma provides key properties of extremal oriented graphs, which  is helpful for proving the extremal case.

\begin{lemma}\label{4.1}
Let $h,q$ be two positive integers and let $\mu$ and $\rho$ be two real numbers such that $0<\mu \ll \rho \ll q^{-1}h^{-1}$. There exists a constant $n_0 = n_0(h,q)$, such that the following conclusion holds. Suppose \(D\) is an oriented graph on \(n \geq n_0\) vertices with \(\delta^0(D) \geq 3n/8 - 2q\rho n\). Let \(S \subseteq V(D)\) with \(|S| \leq 2q\rho n\) and \(B \subseteq V(D)\) with \(|B| = h\). If \(D \setminus B\) has an extremal partition \((D_1, D_2, D_3, D_4)\) with parameters $(32$, $10^4$, $1/300$, $1-10^4\sqrt{\mu })$, then the following statements hold:
\begin{itemize}[itemsep=0pt, topsep=0.5pt,parsep=1pt]
\item[(i)] If $M$ is a subset of $D_1$ or $D_3$ with order at least $ 200\sqrt{ \mu}n$, then $M$ contains a 3-path.
\item[(ii)] For any two vertices $u,v\in B$ (not necessarily distinct), there exist two indices $i, j\in [4]$ with $i\equiv j + 1$ or $i\equiv j + 2$ \emph{(mod 4)} such that $d^+_{D_i\setminus S}(u)\geq n/48$, $d^-_{D_j\setminus S}(v)\geq n/48$.
\end{itemize}
\end{lemma}
\begin{proof}
To prove (i), w.l.o.g., suppose that \(M \subseteq D_1\) with \(|M| = 200\sqrt{\mu}n\) and \(D[M]\) contains no 3-path (the case \(M \subseteq D_3\) follows analogously). Recall that every  tournament on 4 vertices contains a 3-path. {If \(D_1\) induces a tournament, it has exactly \(\frac{|D_1|(|D_1| - 1)}{2}\) arcs, and every induced subgraph of \(D_1\) on 4 vertices would contain a 3-path. Hence, every induced subgraph of \(M\) on 4 vertices must miss at least one arc to avoid forming a tournament of order 4. The number of subsets of \(M\) with 4 vertices is \(\binom{|M|}{4}\), and each missing arc lies in at most \(\binom{|M| - 2}{2}\) such subsets. It follows that \(D[M]\) has at least \(\frac{\binom{|M|}{4}}{\binom{|M| - 2}{2}}\) missing arcs compared to a tournament of the same order.}

Recall that \(|D_1|\leq  (1/4 + 32\mu)(n-h)\leq (1/4 + 32\mu)n\). The total number of arcs in \(D_1\) is thus at most:
\[
\frac{(|D_1|)(|D_1| - 1)}{2} - \frac{\binom{|M|}{4}}{\binom{|M| - 2}{2}} \leq \frac{(n/4 + 33\mu n)(n/4 + 33\mu n - 1)}{2} - \frac{\binom{200\sqrt{\mu}n}{4}}{\binom{200\sqrt{\mu}n - 2}{2}}.
\]
Simplifying this expression, the number of arcs in \(D_1\) would be less than \((1/2 - 2\cdot 10^4\mu)n^2/16\), contradicting the given condition \(a(D_1) > (1/2 - 10^4\mu)(n-h)^2/16\geq (1/2 - 2\cdot 10^4\mu)n^2/16\)  (let $n_0$ be a positive integer such that $h,q\ll n_0$). Therefore, \(M\) must contain a path of length 3.

For (ii), consider any two vertices \(u, v \in B\). Since \(\delta^0(D) \geq 3n/8 - 2q\rho n\), \(u\) has at least \((3/8 - 4q\rho)n\) out-neighbours in \(D \setminus S\), while \(v\) has at least \((3/8 - 4q\rho)n\) in-neighbours in \(D \setminus S\). Note that the extremal partition \((D_1, D_2, D_3, D_4)\) has \(|D_i| \leq (1/4 + 32\mu)n\). By the pigeonhole principle, there exist two {parts} \(D_{i'}\) and \(D_{i''}\) ($i',i''\in [4]$) such that both \(d^+_{D_{i'}}(u)\) and \(d^+_{D_{i''}}(u)\) are at least \(\frac{(3/8 - 4q \rho )n - h - (1/4 + 32\mu )n}{3} \geq n/48\) (as $h\ll n_0\leq n$). Similarly, for \(v\), there exist two {parts} \(D_{j'}\) and \(D_{j''}\) such that \(d^-_{D_{j'}}(v) \geq n/48\) and \(d^-_{D_{j''}}(v) \geq n/48\). Note that there are four parts in total. Assume no \(i \in \{i', i''\}, j \in \{j', j''\}\) satisfy \(i \equiv j + 1 \pmod{4}\). Then at least one pair \((i, j)\) must have $i\equiv j + 2$ (mod 4), completing the proof.
\end{proof}

The subsequent lemma indicates that, under certain semi-degree conditions, by adjusting  a small number of vertices of an extremal partition and contracting a few short paths, there exists an extremal partition with a specific structure.
\begin{lemma}\label{4.2}
Let $ h,q,b$ be positive integers with $b \leq q$ and $\mu \ll 1$ be a real number. There is an integer $n_0:=n_0(\mu, q)$ such that the following holds. Let $ D $ be an oriented graph of order $ n \geq n_0 $. Suppose $ B \subseteq V(D) $ is a subset of order $h$. If $ D \setminus B $ admits an extremal partition $ (D^*_1, D^*_2, D^*_3, D^*_4) $ with parameters $(32$, $10^4$, $1/300$, $1-10^4\sqrt{\mu })${, then the following holds.} 
\begin{itemize}[itemsep=0pt, topsep=0.5pt,parsep=1pt]
\item[(i)] If $ \delta^0(D) \geq \frac{3n + 3h + 3q - 5}{8} $, then there exists an extremal partition $ (D_1, D_2, D_3, D_4) $ of $V(D)\setminus B$ with parameters $ (10^3, 10^6, 1/400,\ 1-10^6\sqrt{\mu}) $ such that {$ ||D_2| - |D_4|| = b $}. This partition is obtained by adjusting the location of at most $ 65\mu n $ vertices from $ (D^*_1, D^*_2, D^*_3, D^*_4) $ and contracting a collection of disjoint paths $ Q_1, \ldots, Q_r $ in $ D \setminus B $ with $ \left|\bigcup_{i=1}^r V(Q_i)\right| \leq 650\mu n $.
\item[(ii)]  {If $B=\{u\}$ \emph{(}$h=1$\emph{)} and $ \delta^0(D) \geq \frac{3n  - 4}{8} $}, then there exists an extremal partition $ (D_1,D_2,D_3,D_4) $ of $V(D)$ with parameters $ (10^3, 10^6, 1/400,\ 1-10^6\sqrt{\mu}) $ such that $ |D_2| = |D_4| $. This partition is obtained by adjusting the location of at most $ 65\mu n $ vertices from $ (D^*_1\cup B, D^*_2, D^*_3, D^*_4) $ and contracting a collection of disjoint paths $ Q_1, \ldots, Q_r $ in $ D $ with $ \left|\bigcup_{i=1}^r V(Q_i)\right| \leq 650\mu n $.
\end{itemize}
\end{lemma}

 In the proof of Lemma \ref{4.2}, contractions of paths will play an important role. Given a partition $(D_1,D_2,D_3,D_4)$ of $V(D)$ and a path $P$ in $D$ whose initial vertex $p_{1}$ and final vertex $p_{2}$ lie in the same part $D_i$, the contraction of $P$ yields the following oriented graph $D'$: we add a new vertex $p$ to the part $D_i$ and remove (the vertices of) the path $P$ from $D$; let $N_{D'}^{+}(p)=N^{+}_{D_{i+1}}\left(p_{2}\right)$, $N_{D'}^{-}(p)=N^{-}_{D_{i-1}}\left(p_{1}\right)$, and $ A(D'\setminus \{p\})=A(D'\setminus V(P))$. Note that any cycle in $D'$ that includes $p$ corresponds to a cycle in $D$ that includes $P$.

\begin{proof}\textbf{of Lemma \ref{4.2}.}
First, we prove (i). The lemma is trivial if {\(||D^*_2| - |D^*_4|| = b\)}. Consequently, we assume that \(|D^*_2| - |D^*_4| > b\), \(|D^*_4| - |D^*_2| > b\), \(|D^*_2| - |D^*_4| < b\) or \(|D^*_4| - |D^*_2| < b\). {We next consider the cases \(|D^*_2| - |D^*_4| > b\) or \(|D^*_2| - |D^*_4| < b\), as the remaining two cases are entirely symmetric.} Our aim is to rearrange the positions of a few vertices to reduce the difference between \(|D^*_2|\) and \(|D^*_4| + b\), while ensuring that the newly obtained partition remains extremal with a slight change in its parameters. Given a partition $(O_1,O_2,O_3,O_4)$ of $V(D)$, we classify vertices as follows with $\gamma = 1/400$: if $|O_2| - |O_4| > b$, we call a vertex \textbf{good} if it is $\gamma$-acceptable, and it belongs to $O_4$ or has one of the properties
$O_1:(O_2)_{<\gamma}(O_3)_{<\gamma}$, $O_2:(O_1)^{<\gamma}(O_2)^{<\gamma}_{<\gamma}(O_3)_{<\gamma}$, $O_3 :(O_1)^{<\gamma}(O_2)^{<\gamma}$. In this case, $O_4$ is called the \textbf{good set} and $O_2$ is called the \textbf{bad set}. And if $|O_2| - |O_4| < b$, we call a vertex \textbf{good} if it is $\gamma$-acceptable, and it belongs to $O_2$ or has one of the properties
$O_1:(O_4)^{<\gamma}(O_3)^{<\gamma}$, $O_3 :(O_1)_{<\gamma}(O_4)_{<\gamma}$, $O_4:(O_1)_{<\gamma}(O_4)^{<\gamma}_{<\gamma}(O_3)^{<\gamma}$. In this case, $O_4$ is called the \textbf{bad set} and $O_2$ is called the \textbf{good set}. In either case, we call a vertex \textbf{bad} if it is not good.

 The following shows that the goal can be achieved by relocating the bad vertices. Now we consider the bad vertices in the initial partition $(D^*_1, D^*_2, D^*_3, D^*_4)$ of $V(D)$.

\begin{claim}\label{claim5}
Each bad vertex in the bad set can become \(\gamma\)-acceptable by arranging it in \(D^*_1\cup D^*_3\).
\end{claim}

\begin{proof}
W.l.o.g., assume $D^*_2$ is the bad set and $v$ is a bad vertex in $D^*_2$.
Then $v$ satisfies at least one of properties $D^*_2:(D^*_1)^{>\gamma}$, $D^*_2:(D^*_2)^{>\gamma}$, $D^*_2:(D^*_2)_{>\gamma}$ and $D^*_2:(D^*_3)_{>\gamma}$. By the $1/300$-acceptability of $v$, we get that $v$ has properties $D^*_2:(D^*_1)_{>1/300}$ or $D^*_2:(D^*_4)_{>1/300}$ and $D^*_2:(D^*_4)^{>1/300}$ or $D^*_2:(D^*_3)^{>1/300}$. Hence, if $v$ has properties $D^*_2:(D^*_1)^{>\gamma}$ or $D^*_2:(D^*_2)^{>\gamma}$, then we move the vertex $v$ to the part $D^*_1$. This means that $v$ has property $D^*_1:(D^*_1)^{>\gamma}(D^*_1)_{>1/300}$ or $D^*_1:(D^*_1)^{>\gamma}(D^*_4)_{>1/300}$ or $D^*_1:(D^*_2)^{>\gamma}(D^*_4)_{>1/300}$ or $D^*_1:(D^*_2)^{>\gamma}(D^*_1)_{>1/300}$. By the definition of $\gamma$-acceptable, $v$ is $\gamma$-acceptable after moving. And if $v$ has properties $D^*_2:(D^*_2)_{>\gamma}$ or $D^*_2:(D^*_3)_{>\gamma}$, then $v$ becomes $\gamma$-acceptable after moving the vertex $v$ into $D^*_3$. The case where \( D^*_4 \) is the bad set is analogous.
\end{proof}

\begin{claim}\label{claim4}
We can arrange every bad vertex in $D^*_1\cup D^*_3$ to be a good vertex by moving it to the good set.
\end{claim}

\begin{proof}
Here, we present the proof for the case where \( D^*_4 \) is the good set, as a similar analysis can be applied if \( D^*_2 \) is the good set. If $v$ is a bad vertex in $D^*_1$, we can obtain that $D^*_1:(D^*_2)_{>\gamma} $ or $D^*_1:(D^*_3)_{>\gamma} $. Since $v$ is a $1/300$-acceptable vertex, it follows that $D^*_1:(D^*_1)^{>1/300} $ or $D^*_1:(D^*_2)^{>1/300} $. It is easy to check that we move the vertex $v$ to $D^*_4$ then $v$ becomes $\gamma$-acceptable. Furthermore, it is good. The same argument applies for the case $v\in D^*_3$.
\end{proof}

Let $n_0$ be an integer such that $\mu n_0\geq q $. Define $r^*=\left||D^*_2| - (|D^*_4| + b)\right|$, then $r^*\leq 64\mu n+q \leq 65\mu n$ as $b\leq q$ and $\mu n\geq \mu n_0\geq hq$. {According to the definition of a good set, removing vertices from a bad set or adding vertices to a good set reduces the gap between \(|D^*_2|\) and \(|D^*_4|+b\). Specifically, each application of either Claim \ref{claim5} or Claim \ref{claim4} decreases the gap between \(|D^*_2|\) and \(|D^*_4|+b\) by 1. Thus, if the number of bad vertices in the union of the bad set and \(D^*_1\cup D^*_3\) is at least \(r^*\), then applying Claims \ref{claim5}–\ref{claim4} a total of \(r^*\) times—equivalently, redistributing exactly \(r^*\) bad vertices in the union—yields a new partition \((D_1,D_2,D_3,D_4)\) such that \(|D_2|=|D_4|+b\). By Definition \ref{def4}, it is straightforward to verify that this partition is extremal with parameters \((100, 10^5, 1/350,1-10^5\sqrt{\mu})\)\footnote{For detailed calculations, refer to Property A.1 (II) in the Appendix}, as desired. We therefore assume that the number of bad vertices in the union of the bad set and \(D^*_1\cup D^*_3\) is at most \(r^*\). Redistributing all bad vertices in this union via the application of Claims \ref{claim5}–\ref{claim4} produces a new extremal partition \((D_1,D_2,D_3,D_4)\) with parameters \((100, 10^5, 1/350,1-10^5\sqrt{\mu})\)\footnote{For detailed calculations, refer to Property A.1 (II) in the Appendix}, and every vertex in \(D\setminus B\) is good. Furthermore, we assume that $|D_2|\neq |D_4|+b$, as the desired conclusion then follows immediately.}

Now, define $r=\left||D_2| - (|D_4| + b)\right|$, then $r\leq 65\mu n$. {Next, we claim that there exist $r$ disjoint paths such that contracting these paths reduces the gap between \(|D_2|\) and \(|D_4|+b\) by $r$, where each path is obtained by extending a `special' arc. To this end,} we construct an oriented graph $H^{\prime}$ with $V(H^{\prime})= V(D\setminus B)$ as follows:

\indent $ \bullet$ $A(H^{\prime})=A(D_2, D_1)\cup A(D_2)\cup A(D_3, D_1)\cup A(D_3, D_2)$, when $|D_2|- |D_4|> b$;

\indent $ \bullet$ $A(H^{\prime})=A(D_1, D_3)\cup A(D_1, D_4)\cup A(D_4, D_3)\cup A(D_4)$, when $|D_2|- |D_4|< b$.

\begin{figure}[htbp]  
    \centering  
    \begin{minipage}[t]{0.43\textwidth}
        \centering
        \includegraphics[width=\linewidth]{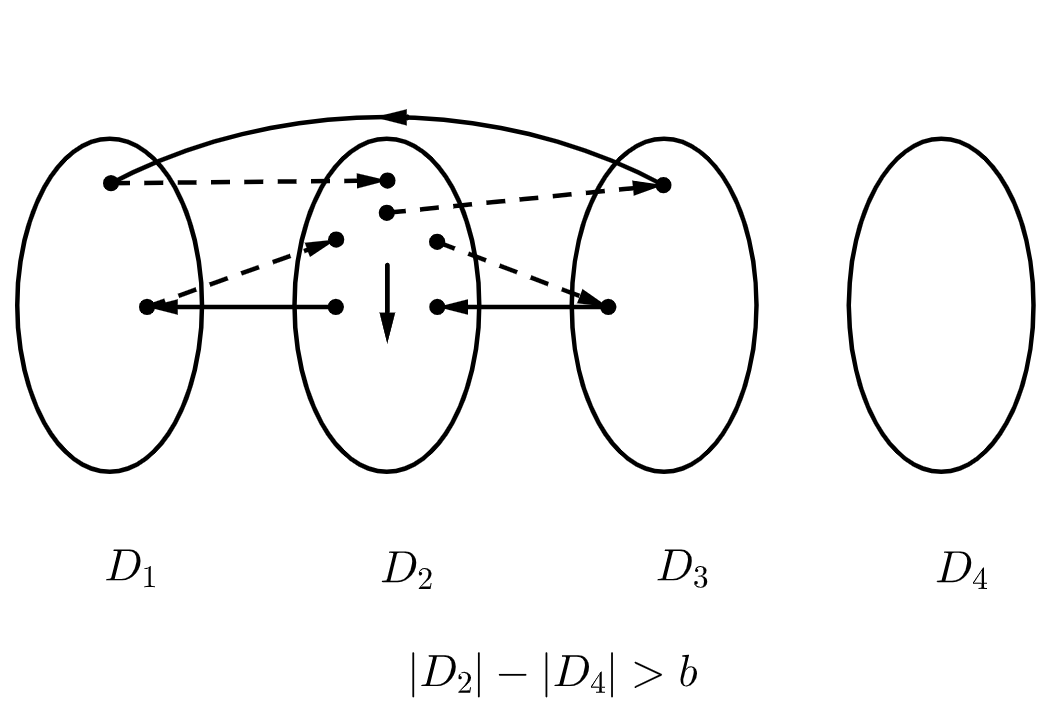}  
    \end{minipage}
    \hspace{0.1\textwidth}  
    \begin{minipage}[t]{0.43\textwidth}
        \centering
        \includegraphics[width=\linewidth]{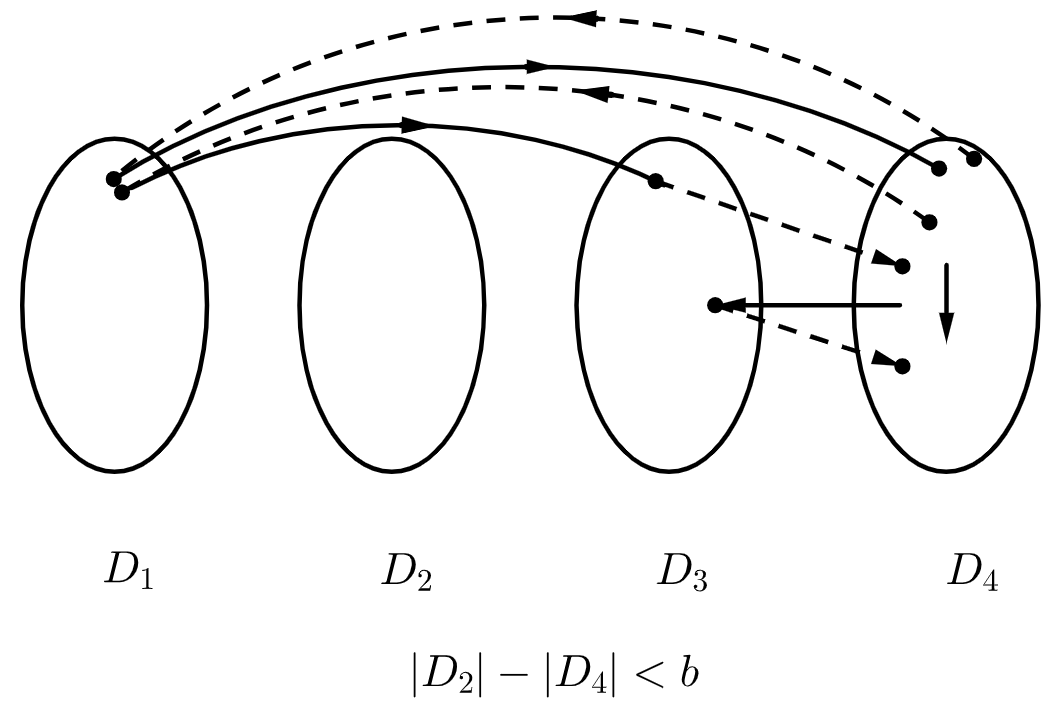}  
    \end{minipage}
    
    \caption{{Illustration of \(H'\): Solid lines denote arcs in \(H'\) (if such arcs exist in $D$), while dashed lines denote arcs that lie in the digraph $D$ but not in \(H'\); the paths formed by the concatenation of solid and dashed lines are exactly the paths \(\{Q'_i \mid i\in [r]\}\) we seek.}}
    \label{tu5}  
\end{figure}
\noindent In either case, choose $M$ to be a maximum matching in $H^{\prime}$. {Then we consider the case $ a(M)\leq  r$ and $a(M)>  r$ separately.} Suppose that $a(M)\geq  r$. {Hereafter, the bad set and good sets are defined with respect to the partition $(D_1,D_2,D_3,D_4)$. Since every vertex in $\bigcup_{i=1}^4 D_i$ is $1/350$-acceptable, it is readily verified that there exist $r$ disjoint paths $\{Q'_i \mid i\in[r]\}$), each containing exactly one arc in $M$ and with endvertices lying in the bad set (see Fig. \ref{tu5}). Subject to this, we choose each path $Q'_i$ to be minimal, and wherever feasible, select its endvertices to be $(1-10^5\sqrt{\mu})$-circular vertices. It should be noted that it may happen that not all endvertices of $Q'_i$ are $(1-10^5\sqrt{\mu})$-circular vertices; when this occurs, we extend $Q'_i$ to a minimal path $Q_i$ in a manner that winds around $D_1,D_2,D_3,D_4$, such that both endvertices of $Q_i$ are $(1-10^5\sqrt{\mu})$-circular vertices in the bad set. This construction implies that each path $Q_i$ has length at most 8, using exactly two more vertices from the bad set than from the good set.} Thus $|\bigcup_{i\in [r]}V(Q_i)|< 10r<650 \mu n$ (here, $9r$ is enough, but we take $10r$ for computational convenience). Further, after contracting these paths, there is a new partition $(D'_1,D'_2,D'_3,D'_4)$ such that $|D'_2|=|D'_4| + b$. By taking $\mu$ small enough and $n_0$ large enough, it is easy to check that $(D'_1,D'_2,D'_3,D'_4)$ is extremal with parameters $ (10^3, 10^6, 1/400,1-10^6\sqrt{\mu}) $\footnote{for a detailed calculation, see Property A.1 (II) in Appendix}. Then these disjoint paths $Q_1,\ldots,Q_r$ and the partition $(D'_1,D'_2,D'_3,D'_4)$ as desired. 

Now assume $a(M)<  r$. Next, we show that this contradicts the assumption $\delta ^0(D)\geq \frac{3n + 3h + 3q - 5}{8}$. Let $L_i = V(M)\cap D_i$ for each $i\in [4]$. We divide the proof into two cases: $|D_2| - |D_4| > b$ or $|D_2| - |D_4| < b$.

\textbf{Case 2.1: $|D_2| - |D_4| > b$}

The maximality of $M$, $\sum_{i\in [4]}|L_i|\leq 2a(M)$, and the definition of the good vertex yield that
\begin{equation*}\label{9}
\begin{aligned}
a(D_3, D_1)&\leq a(L_3, D_1) + a(D_3, L_1 )\leq (|L_3| + |L_1|)\gamma (n/4 + 32\mu n)\leq a(M)n/300.
\end{aligned}
\end{equation*}
Similarly, we also can obtain that $a(D_2,D_1)\leq a(L_2, D_1) + a(D_2, L_1) < a(M)n/300,$ $a(D_3,D_2)\leq a(L_3, D_2) + a(D_3, L_2 ) < a(M)n/300,$ and $a(D_2)\leq 2|L_2|\cdot \gamma (n/4+32\mu n)<a(M)n/150.$
Hence, it follows that
$$\sum_{v\in D_1} d_{D\setminus B}^-(v)\leq \frac{|D_1|\cdot (|D_1-1|)}{2}+ a(M)n/150+ |D_1|\cdot |D_4|,$$
$$\sum_{v\in D_2} d_{D\setminus B}(v)\leq (|D_1|+|D_3|)\cdot |D_2| + 2a(M)n/150+|D_2|\cdot |D_4|,$$
$$\sum_{v\in D_3} d_{D\setminus B}^+(v)\leq \frac{|D_3|\cdot (|D_3-1|)}{2}+ a(M)n/150+ |D_3|\cdot |D_4|.$$

 By pigeonhole principle, there are three vertices $u\in D_1, v\in D_2, w\in D_3$ such that
$$d_{D\setminus B}^-(u)\leq \frac{|D_1|-1}{2}+|D_4|+\frac{a(M)n}{150 |D_1|}<\frac{|D_1|-1}{2}+|D_4|+\frac{a(M)}{30},$$
$$d_{D\setminus B}(v)<|D_1|+|D_3|+|D_4|+ \frac{a(M)}{15},\ \text{ and } d_{D\setminus B}^+(w)<\frac{|D_3|-1}{2} + |D_4| + \frac{a(M)}{30}.$$
Together $|D_2| - |D_4|-b=r>a(M)$ with $|D_1|+|D_2|+|D_3|+|D_4|= n-h$, it yields that
\begin{equation}\label{2}
\begin{aligned}
4\times \frac{3n+3h+3q-5}{8}-3h&\leq d_{D\setminus B}^-(u) + d_{D\setminus B}^+(w) + d_{D\setminus B}(v) \\
&<\frac{3}{2}(n-h-|D_2| + |D_4|)+ \frac{2a(M)}{15} -1.
\end{aligned}
\end{equation}
 Thus (\ref{2}) gives that $0\leq 41(|D_2|-|D_4|)/30+2b/15<-(3q-3)/2.$ This is impossible.

\textbf{Case 2.2: $0\leq |D_2|- |D_4|< b$.}

The similar calculations as above yields that $a(D_1, D_3)< a(M)n/300, $ $a(D_1,D_4)< a(M)n/300,$ $a(D_4,D_3) < a(M)n/300,$ and $a(D_4) <a(M)n/150.$ By Pigeonhole principle, there are three vertices $u\in D_1, v\in D_4, w\in D_3$ such that
$$d_{D\setminus B}^+(u)<\frac{|D_1|-1}{2}+|D_2|+\frac{a(M)}{30},$$
$$d_{D\setminus B}(v)<|D_1|+|D_3|+|D_2|+ \frac{a(M)}{15},\ \text{ and }d_{D\setminus B}^-(w)<\frac{|D_3|-1}{2} + |D_2| + \frac{a(M)}{30}. $$
Together $|D_4|+b-|D_2|=r>a(M)$ with $|D_1|+|D_2|+|D_3|+|D_4|= n-h$, it yields that
\begin{equation*}
\begin{aligned}
4\times \frac{3n+3h+3q-5}{8}-3h&\leq d_{D\setminus B}^+(u) + d_{D\setminus B}^-(w) + d_{D\setminus B}(v) \\
&<\frac{3}{2}(n-h+|D_2| - |D_4|) + \frac{2a(M)}{15}-1.
\end{aligned}
\end{equation*}
This gives $|D_4|-|D_2| < 1-q+4 a(M) / 45.$ On the other hand, we previously observed that $r=|D_4|-|D_2| +b \geqslant a(M)+1$. Combining this gives $41a(M)/45< b-q\leq 0$ (as $b\leq q$), which is impossible. This completes the proof of $(i)$.

{The proof of (ii) follows similar, thus we provide a sketch here.\footnote{The choice of $b$ differs between (i) and (ii), and in (ii), the vertex in $B$ can be both adjusted and contracted,  see Lemma A.1 in Appendix for details} For \( u \in B \), we assign \( u \) to the partition \( (D^*_1, D^*_2, D^*_3, D^*_4) \) of \( D \setminus B \). By Lemma \ref{4.1}(ii), there exist indices \( i, j \in [4] \) with \( i \equiv j + 1 \) or \( i \equiv j + 2 \pmod{4} \) such that \( d^+_{D^*_i \setminus S}(u) \geq n/48 \) and \( d^-_{D^*_j \setminus S}(u) \geq n/48 \), allowing \( u \) to be assigned as a \( (1/100) \)-acceptable vertex to some \( D^*_k \) ($k\in [4]$). Updating the partition \( (D^*_1, D^*_2, D^*_3, D^*_4) \) to cover \( V(D) \), we assume \( |D^*_2| \neq |D^*_4| \) (the case \( |D^*_2| = |D^*_4| \) trivially satisfies the conclusion). Applying Claims \ref{claim5}–\ref{claim4}, we reassign bad vertices so that all vertices become good vertices ($b=0$), thus forming an extremal partition \( (D_1, D_2, D_3, D_4) \) with parameters \( (100, 10^5, 1/350, 1 - 10^5\sqrt{\mu}) \).  Suppose \( |D_2| > |D_4| \) (the dual case is symmetric). Let \( r'' = \big| |D_2| - |D_4| \big| \) and construct an oriented graph \( H' \) on \( V(D) \). Let $M$ be a maximum matching in $H'$. Similarly, we obtain $a(M)<  r''$, then this would contradict the assumption $\delta ^0(D)\geq \frac{3n -4}{8}$. W.l.o.g., we assume that \(|D_2| > |D_4|\) (the dual case is symmetric).

Through analogous computations, we establish the bounds: $a(D_3, D_1)< a(M)n/300,$ $a(D_2,D_1) < a(M)n/300,$ $a(D_3,D_2)< a(M)n/300,$ and $a(D_2)<a(M)n/150.$ By pigeonhole principle, there are three vertices $u\in D_1, v\in D_2, w\in D_3$ such that
$$d_{D}^-(u)\leq \frac{|D_1|-1}{2}+|D_4|+\frac{a(M)n}{150 |D_1|}<\frac{|D_1|-1}{2}+|D_4|+\frac{a(M)}{30},$$
$$d_{D}(v)<|D_1|+|D_3|+|D_4|+ \frac{a(M)}{15},\ \text{ and } d_{D}^+(w)<\frac{|D_3|-1}{2} + |D_4| + \frac{a(M)}{30}.$$
Together with $|D_1|+|D_2|+|D_3|+|D_4|= n$, it yields that
\begin{equation*}
\begin{aligned}
4\times \frac{3n-4}{8}&\leq d_{D}^-(u) + d_{D}^+(w) + d_{D}(v) <\frac{3}{2}(n-|D_2| +|D_4|) -1+ \frac{2a(M)}{15}.
\end{aligned}
\end{equation*}
 Thus $|D_2| - |D_4|=r''>a(M)$ gives that $41a(M)/30<-1/2.$ This is impossible.}
\end{proof}

Lemma \ref{4.3} provides a sufficient condition for Hamiltonicity in an extremal oriented graphs, which is derived from the work of  \cite{Keevash(2009)}. Although \cite{Keevash(2009)} does not explicitly state this result as a lemma or theorem, the essence of the result is embedded within their proof. For the clarity and completeness, we include a detailed proof of Lemma \ref{4.3} in Appendix\footnote{Please refer to Lemma B.1 in Appendix}.
 
 \begin{lemma}\label{4.3}
\cite{Keevash(2009)} There exist a real number \( \eta_0 < 1 \) and an integer $n_0=n_0(\eta_0)$ such that for any $0< \xi  \ll \eta_0$, the following statement holds. If an oriented graph \( D \) of order \( n \geq {n_0} \) has a partition \( (D_1, D_2, D_3, D_4) \) such that
\begin{itemize}[itemsep=0pt, topsep=0.5pt,parsep=1pt]
\item[$(i)$] \( |D_2| = |D_4| \), and $|D_i|=(1/4 \pm \xi  ) |D|$ for all $i\in [4]$;
\item[$(ii)$] if $M$ is a subset of $D_3$ with order at least $ \xi  n$, then $M$ contains a 3-path;
\item[$(iii)$] all but at most $\xi  n$ vertices in $D_4$ have at least $|D_2|/32$ out-neighbours and at least $|D_2|/32$ in-neighbours in $D_2$;
\item[$(iv)$] for the partition \( (D_1, D_2, D_3, D_4) \), every vertex is $1/600$-acceptable and the number of non $(1 - 20q \xi )$-circular vertices is at most $\xi  n$,
\end{itemize}
 then \( D \) contains a Hamilton cycle.
\end{lemma}

K\"{u}hn, Osthus and Treglown \cite{Kuhn(2010)} first proposed the following definition, {which} was already used implicitly in \cite{Keevash(2009)}. This definition has strong structural properties.

\begin{definition}\emph{(Robust $(\mu,\tau)$-outexpander)} \label{def1}
Let $\mu <1$ and $ \tau <1$ be positive reals. A digraph \(R\) is a \textbf{robust \((\mu, \tau)\)-outexpander} if \(|N^+_R(S)| \geq |S| + \mu |R|\) for all \(S \subseteq V(R)\) with \(\tau |R| \leq |S| \leq (1-\tau)|R|\).
\end{definition}

Lemma \ref{4.4} {states} that if $D$ is not an extremal case, then its reduced digraph is a robust outexpander. While \cite{Keevash(2009)} does not explicitly formulate Lemma \ref{4.4} as a {standalone} result, it is implicitly used in their proofs. We provide a full proof of {Lemma \ref{4.4}} in the Appendix\footnote{Please refer to Lemma B.3 in Appendix} for completeness.

\begin{lemma}\label{4.4}
\cite{Keevash(2009)} Let $ M', n_0 $ be positive numbers and $ \varepsilon, d, \mu $ be positive constants satisfying $
1/n_0 \ll 1/M' \ll \varepsilon \ll d \ll \mu  \ll  1.
$
Let $ D $ be an oriented graph on $ n \geq n_0 $ vertices with $
\delta^0(D) \geq (3/8 - 2d)n$, and let $ W \subseteq V(D) $ be a vertex subset with $ |W| \leq dn $.
Applying the Diregularity lemma to $ D \setminus W $ {on input $\varepsilon$ and $M'$}, we obtain the reduced digraph $ R' $ with parameters $ (\varepsilon, d) $. Suppose $ R $ is a spanning oriented subgraph of $ R' $ with parameters $ (\varepsilon, d) $ satisfying $
\delta^0(R) \geq (3/8 - 6d)|R|.$ If $ D $ is not an extremal case with parameters $(31$, $10^3$, $1/200$, $1-10^3\sqrt{\mu })$, then $ R $ is a robust $ (\mu, 1/3) $-outexpander.
\end{lemma}

The following fact reveals the relationship between the parameters of a robust outexpander.

\begin{fact}\label{fac1}
Let $d, \mu$, and $ \tau $ be positive real numbers with $d\ll \mu\leq \tau \leq 1/100$. Suppose $R$ is an oriented graph satisfying $\delta^0(R)\geq (3/8-6d)|R|$. If $R$ is a robust $(\mu,1/3)$-outexpander, then it is also a robust $(\mu, \tau)$-outexpander.
\end{fact}
\begin{proof}
Let $|R|=k$. By the definition of robust $(\mu,1/3)$-outexpander, we have that $|N^+_R(S)|\geq |S| + \mu k$ for every $S\subseteq V(R)$ satisfying $ k/3 \leq |S| \leq 2k/3$. For a subset $S$ of $V(R)$ satisfying $\tau k\leq |S| \leq k/3$, it is easy to check that
 $$|N^+_R(S)|\geq (3/8-6d)k\geq k/3+k/100\geq |S| + \mu k \text{ (as $d\ll \mu\leq \tau \leq 1/100$)}.$$
 And if $ 2k/3 \leq |S| \leq (1-\tau)k$, then $|S| + |N^-_R(v)| > k$, that is, $S\cap N^-_R(v)\neq \emptyset$ for any $v\in R$. Thus $N^+_R(S) = V(R)$. So $$|N^+_R(S)|=|R|=(1-\tau)k+\tau k \geq |S| + \tau k\geq |S| + \mu k \text{ (as $d\ll \mu\leq \tau \leq 1/100$)}.$$ This implies that $R$ is a robust $(\mu, \tau)$-outexpander, as desired.
\end{proof}

Lemma \ref{4.5} gives a sufficient condition for Hamiltonicity of oriented graphs.

\begin{lemma}\label{4.5}
\cite{Keevash(2009)} Let $M',n_0$ be positive numbers and let $\varepsilon,d,\eta, \nu,\tau$ be positive constants such that $1/n_0\ll 1/M'\ll \varepsilon \ll d\ll \mu \ll \tau \ll \eta <1$. Let $D$ is an oriented graph on $n\geq n_0$ vertices such that $\delta^0(D)\geq 2\eta n$. Let $R'$ be the reduced digraph of $D$ with parameters $(\varepsilon, d)$ and $|R'|\geq M'$. If $R'$ has a spanning oriented subgraph $R$ with $\delta^0(R)\geq \eta |R|$ such that $R$ is a robust $(\mu, \tau)$-outexpander, then $D$ contains a Hamilton cycle.
\end{lemma}

Here, we introduce the pinch vertex operation, which plays a crucial role in our proof.
\begin{definition}[Pinch vertex operation]\label{keydef}
Given an oriented graph $D$, two vertices \(u, v \in V(D)\), and two subsets \(X, Y \subseteq V(D)\): If \(N_X^+(u) \cap N_Y^-(v) \neq \emptyset\), partition \(N_X^+(u) \cap N_Y^-(v)\) into two sets \(N_u, N_v\) of size as equal as possible; Otherwise, set \(N_u = N_v = \emptyset\). Define \(N_1 = N_X^+(u) \setminus N_v\) and \(N_2 = N_Y^-(v) \setminus N_u\). The pinch vertex operation (on \(u, v\) with respect to \(X, Y\)) constructs an oriented graph \(D_{uv}\) as follows: {$V(D_{uv})=(V(D)\setminus \{u,v\})\cup \{x_{u,v}\}$, where $x_{u,v}$ is a new vertex; $A(D_{uv})=A(D-\{u,v\})\cup \{x_{u,v}w\mid w\in N_1\} \cup \{wx_{u,v}\mid w\in N_2\} $.}
\end{definition}

 {Note that any cycle in $D_{uv}$ that includes $x_{u,v}$ corresponds to a $(u,v)$-path in $D$.} Also, this can make the contraction vertex $ x_{u,v}$ retain the out-degree of $u$ in $X$ and the in-degree of $v$ in $y$ as much as possible.

\vspace{10pt}

 {\textbf{Parameters selection:}}
 Throughout the proof process, suppose parameters satisfy
\begin{equation}\label{para1}
1/M' \ll \varepsilon \ll \theta \ll d \ll \mu \ll\{\rho , \tau \}\ll \xi \ll \eta  \ll \min\{\eta _0, k^{-1}q^{-1}\},
\end{equation}
where $\eta _0$ as stated in Lemma \ref{4.3}. For brevity, we denote \(n_0^{(a)}\) as the constant \(n_0\) in Lemma \(a\) or Theorem \(a\). To ensure that the oriented graphs in Theorem \ref{main} is large enough so that the above conclusions can be utilized, we select \(n_0^{(1.1)}\) to satisfy
\begin{equation}\label{para2}
n_0^{(1.1)}\gg \max \{n_0^{(3.8)},n_0^{(4.1)},n_0^{(4.2)},n_0^{(4.3)},n_0^{(4.4)},n_0^{(4.5)} \},
 \end{equation}
 where \(n_0^{(3.8)} = n_0(s,M,\varepsilon)\), in which \(s = q\) and \(M\) is obtained by applying Lemma \ref{lem1} with \(\varepsilon, d, M'\) (as in (\ref{para1})). Additionally, let $n_0:=n_0^{(4.6)}=n_0^{(4.7)}=n_0^{(1.1)}$.

The proof of Theorem \ref{main} is divided into two cases: the extremal case (Lemma \ref{4.6}) and the non-extremal case (Lemma \ref{4.7}). Thus, it suffices to prove Lemmas \ref{4.6}-\ref{4.7}. {In subsequent proofs, given a multidigraph $H$ and an injection $f: V(H)\to V(D)$, let $A(H)=\{u_iv_i \mid i \in [q]\}$ denote its arc set (Note that two arcs may share common vertices here; this notation is solely used to distinguish distinct arcs), and assume $f(v)=v$ for all $v\in V(H)$. Let $\mathcal{N} = \{n_i \geq 4 \mid i \in [q]\}$ be the set of prescribed path lengths. Here, $\mathcal{N}$ consists of $q$ integers satisfying $\sum_{i \in [q]} n_i \leq n - h + q$ (i.e., the total number of vertices in the $H$-subdivision is at most $n$).} So our goal is to find a set $\mathcal{P} = \{P_{i}, i\in [q] \mid P_{i} \text{ is a } (u_i, v_i)\text{-path of length } n_{i}\}$ of $q$ internally disjoint paths. For the purpose of calculation, we assume that $n_1\geq \cdots \geq n_s>\rho n \geq n_{s + 1}\geq \cdots\geq n_q$ for Lemma \ref{4.6}, and assume $n_1\geq \cdots \geq n_s>\theta n \geq n_{s + 1}\geq \cdots\geq n_q$ for Lemma \ref{4.7}, where $\theta \ll d \ll \mu \ll \rho $ in (\ref{para1}).

\subsection{Extremal case}

\begin{lemma}\label{4.6}
Suppose \(H\) is a multi-digraph with \(h\) vertices and \(q\) arcs, and \(D\) is an extremal oriented graph of order \(n \geq n_0\) with parameters $(31$, $10^3$, $1/200$, $1-10^3\sqrt{\mu })$. If $D$ satisfies {\(\delta^0(D)\geq \frac{3n + 3h + 3q - 5}{8} \)}, then \(D\) is arbitrary  \(H\)-linked. Specially, if $H$ is a loop and $D$ satisfies {\(\delta^0(D)\geq \frac{3n -4}{8} \)}, then $ D $ is arbitrary  $ H $-linked.
\end{lemma}

\textbf{Overview of Lemma \ref{4.6}.} {Let \(f:V(H)\to V(D)\) be an injection, and let \(B = f(V(H))\) denote the image of $f$. First, according to the extremal partition of $D$ and the small size of $B$, we readily derive an initial extremal partition of \(D \setminus B\). Applying Lemma \ref{4.2} to refine this partition yields a refined extremal partition \((D_1, D_2, D_3, D_4)\) of \(D \setminus B\) such that \(|D_2| = |D_4| + b\), where $b$ is determined by the neighbour distribution of branch vertices.} {The structure of this partition ensures the existence of all short subdivision paths. Apply Lemma \ref{lemma2} to the remaining subdigraph of $D\setminus B$ --- obtained by removing vertices from these short subdivision paths --- to yield $s$ subgraphs $G_1,\ldots,G_s$, where $s$ denotes the number of required long subdivision paths.} Pinch the branch vertices \( u_i \) and \( v_i \) into a single vertex and added to \( G_i \), followed by adjusting a few vertices to ensure \( |V(G_i) \cap D_2| = |V(G_i) \cap D_4| \) (feasible due to \( |D_2| = |D_4| + b \)) and $|G_i|=n_i$. Finally, we verify that each oriented graph $G_{i}$ satisfies the condition of Lemma \ref{4.3}. This implies the existence of a Hamilton cycle in \( G_i \), thereby yielding a subdivision $(u_i,v_i)$-path of length \( n_i \) in \( D \), which completes the proof.

\begin{proof}\textbf{of Lemma \ref{4.6}.}
{Constants are selected as in \eqref{para1} and \eqref{para2}, and we suppose that $n_1\geq \cdots \geq n_s>\rho n \geq n_{s + 1}\geq \cdots\geq n_q$ holds for this case.} {Let \(f:V(H)\to V(D)\) be an injection, and let \(B = f(V(H))\) denote the image of $f$, so $|B|=h$.} Since $D$ is an extremal case with parameters $(31$, $10^3$, $1/200$, $1-10^3\sqrt{\mu })$ and $n_0\gg h$, it follows that $ D \setminus B $ admits an extremal partition $(D_1^*, D_2^*, D_3^*, D_4^*)$ with parameters $(32$, $10^4$, $1/300$, $1-10^4\sqrt{\mu })$ (please refer to Property A.1 (I) in Appendix for details).

We will utilize Lemma \ref{4.2} to optimize the present extremal partition. For the case that \(H\) is a loop, if \(n_1\leq \rho n\), one can directly use Lemma \ref{lemma3} to find the desired $H$-subdivision; otherwise, apply Lemma \ref{4.2} (ii),  by adjusting the location of at most $ 65\mu n $ vertices from \( (D^*_1\cup B, D^*_2, D^*_3, D^*_4) \) and contracting a collection of disjoint paths \( Q_1, \ldots, Q_r \) in $D$ with \( \left|\bigcup_{i=1}^r V(Q_i)\right| \leq 650\mu n \), there exists an extremal partition $ (D_1, D_2, D_3, D_4) $ of $V(D)$ with parameters $ (10^3, 10^6, 1/400,\ 1-10^6\sqrt{\mu}) $ such that $ |D_2| = |D_4| $. For the case that \(H\) is not merely a loop, to utilize Lemma \ref{4.2} (i), determine the parameter \( b \) as follows. For each pair of vertices $u_i,v_i\in B$, fix two indices $i^+,i^-\in [4]$ such that $i^+\equiv i^- + 1$ or $i^+\equiv  i^- + 2$ (mod 4) and $d^+_{D^*_{i^+ }}(u_i)\geq n/48$, $d^-_{D^*_{i^-}}(v_i)\geq n/48$. The existence of these indices is guaranteed by Lemma \ref{4.1} (ii). Then for each $i\in [q]$, set
$$b(P_{i})=\begin{cases}
           0, & \mbox{if } i^+\in \{1,2\} \mbox{ and } i^-\in \{1,4\}, \mbox{ or } i^+\in \{3,4\} \mbox{ and } i^-\in \{2,3\}, \\
           -1, & \mbox{if } i^+\in \{3,4\} \mbox{ and } i^-\in \{1,4\}, \\
          1, & \mbox{otherwise}.
         \end{cases}$$
 Indeed, $b(P_{i})$ denotes that the $(u_i, v_i)$-path that we construct next  use \(b(P_{i})\) more vertices from \(D^*_2\) than in \(D^*_4\). That is, $b(P_{i})=|(V(P_{i})\cap D^*_2|-|V(P_{i})\cap D^*_4)| $. Let $b=\sum_{i\in [q]}b(P_{i}),$ then $|b|\leq q$. W.l.o.g., assume $b\geq 0$. Apply Lemma \ref{4.2}, by adjusting the location of at most $ 65\mu n $ vertices from \( (D^*_1, D^*_2, D^*_3, D^*_4) \) and contracting a collection of disjoint paths \( Q_1, \ldots, Q_r \) with \( \left|\bigcup_{i=1}^r V(Q_i)\right| \leq 650\mu n \), there exists an extremal partition \( (D_1, D_2, D_3, D_4) \) with parameters $ (10^3, 10^6,  1 /400,1-10^6\sqrt{\mu}) $ such that \( |D_2| - |D_4| = b \). It not hard to see that every pair of vertices $u_i,v_i\in B$ still satisfies $d^+_{D_{i^+ }}(u_i)\geq n/100$, $d^-_{D_{i^-}}(v_i)\geq n/100$. In either case, let $C$ be the set of all contraction vertices that {are} obtained by contracting paths \( Q_1, \ldots, Q_r \), so $|C|\leq 650\mu n$.

We now construct the short paths \( P_{s+1}, \ldots, P_q \) inductively. Let \( S^* \) denote the non-\((1 - 10^6\sqrt{\mu})\)-circular vertices in the partition \((D_1, D_2, D_3, D_4)\), so \( |S^*| \leq \sqrt{\mu}n \). We will assign the vertices in \( S^* \cup C \) to long subdivision paths. In fact, we further require that each short path \( P_i \) (\( s+1 \leq i \leq q \)) has the form \( P_i = u_i u_i^+ \cdots v_i^- v_i \), where:
\begin{itemize}[itemsep=0pt, topsep=0.5pt,parsep=1pt]
\item \( u_i^+ \in N^+_{D_{i^+}}(u_i) \) and \( v_i^- \in N^-_{D_{i^-}}(v_i) \);
\item All internal vertices are \((1 - 10^6\sqrt{\mu})\)-circular;
\item \(  |V(P_i) \cap D_2| - |V(P_i) \cap D_4|  = b(P_i) \).
\end{itemize}
Assume \( P_{q}, \ldots, P_{j+1} \) exist, and let \( W = \bigcup_{i=j+1}^q V(P_i) \) (\( |W| < q\rho n \)). Then we find the $(u_j, v_j)$-path $P_{j}$. Define \( S = S^* \cup C \cup W \) (\( |S| \leq 2q\rho n \)). For \( u_j \) and \( v_j \), a $(1 - 10^6\sqrt{\mu })$-circular vertex $u^+$ is selected from \(N^+_{D_{j^+}\setminus S}(u_j)\), and a $(1 - 10^6\sqrt{\mu })$-circular vertex $v^-$ is selected from \(N^-_{D_{j^-}\setminus S}(v_j)\). Inequalities $|N^+_{D_{j^+}}(u_j)|-|S|\geq n/200$ and $ |N^-_{D_{j^-}}(v_j)|-|S|\geq n/200$ guarantee the existence of $u^+$ and $v^-$. For any two $(1 - 10^6\sqrt{\mu })$-circular vertices \( x \in D_l \), \( y \in D_{l-2} \) (where indices are modulo 4), there is:
   \[
   |N^+_D(y) \cap N^-_D(x) \cap D_{l-1}| - |S| \geq (\frac{1}{4} - 10^3\mu  - 2 \times 10^6\sqrt{\mu})(n-h) - 2q\rho n \stackrel{(\ref{para1},\ref{para2})}> 200\sqrt{\mu}n.
   \]
Together with Lemma \ref{4.1} (i) and the fact that $j^+\equiv j^- + 1$ or $j^+\equiv j^- + 2$ (mod 4), by selecting a subpath of an appropriate length in $D_1$ or $D_3$, we obtain a $(u^+, v^-)$-path \( P' \) of length \( n_j - 2 \); except for possibly containing a 3-length subpath in \( D_1 \) or \( D_3 \), all other vertices of this path are distributed cyclically in the order of \( D_1, D_2, D_3, D_4 \). Define \( P_j = u_j P' v_j \). By the circularity of the path $P'$ and {the previous definition of $ b(P_j) $}, the selection of $ u^+ \in D_{j^+} $ and $ v^- \in D_{j^-} $ enforces:
$
  |V(P_j) \cap D_2| - |V(P_j) \cap D_4|   = b(P_j).
$
This completes the construction of the short paths by induction.

If \( s = 0 \), the result holds; otherwise, \( s \geq 1 \). To construct long subdivision paths, we apply Lemma \ref{lemma2} to partition $ D\setminus (\bigcup_{i=s+1}^q V(P_i)\cup B\cup C) $ into subgraphs $ \{G_i\}_{i=1}^s $ of specified orders with large semi-degree. Then pinch vertices $ u_i $ and $ v_i $ into a vertex $ x_{u_i,v_i} $, then update $G_i = G_i \cup \{x_{u_i,v_i}\}.$ This reduces finding $(u_i, v_i)$-subdivision path to finding Hamilton cycle in $ G_i $. Then, adjust a small number of vertices to make $|G_i| = n_i$ (the vertices of the contracted path are also taken into account). Finally, Lemma \ref{4.3} is applied to each $ G_i $ to get a Hamilton cycle.

Let $D'=D\setminus (\bigcup_{i=s+1}^q V(P_i)\cup B\cup C)$, and let $D'_l=D_l\setminus (\bigcup_{i=s+1}^q V(P_i)\cup C)$ for all $l\in [4]$. When randomly partitioning the remaining oriented graph via {Lemma \ref{lemma2}}, the proof splits into two: $\sum_{i \in [q]} n_i \geq n - 100q\rho^2 n$ or $\sum_{i \in [q]} n_i < n - 100q\rho^2 n$. Now, assume that $\sum_{i \in [q]} n_i \geq n - 100q\rho^2 n$. To apply Lemma \ref{lemma2}, for each $l\in [4]$, let $m^{i}_l=\lfloor (n_i - 1)/4\rfloor$ for all $ i\in [2,s]$, and then let $m^{1}_l= |D'_l| - \sum^s_{i=2}m^{i}_l$. It is not difficult to find that $m^i_l\geq \rho n/5$ since $n_i\geq \rho n$, for all $i \in [2,s]$. Then we bound $m^1_{l}$ to show that $m^{1}_l\geq 20q\rho n$. Since if $n_1\leq 10^5q\rho n$, it follows that $\sum_{i=1}^q n_i\leq 10^5q^2 \rho n< n- 100q\rho^2 n$ (as (\ref{para1}) and (\ref{para2})), which contradicts with $ \sum_{i=1}^{q}n_i\geq n- 100q\rho^2 n$.  Hence, $n_1\geq 10^5q\rho n$, this yields that
\begin{equation}\label{11}
\begin{aligned}
m^{1}_l= |D'_l| - \sum^s_{i=2}m^{i}_l& \geq (1/4-10^3\mu) (n-h)-650\mu n -q\rho n-\sum_{i=2}^s\lfloor\frac{n_{i}-2}{4}\rfloor\\
&\geq n/4-10^4\mu n -2q\rho n-\frac{n-n_1-2q}{4}\stackrel{(\ref{para1})(\ref{para2})}>20q\rho n.
\end{aligned}
\end{equation}
Applying Lemma \ref{lemma2} with $D:=D'$, $\varepsilon :=\varepsilon,\ s:=s,\ M:=4,\{m^{i}_l, i\in [s], l\in [4]\}$, $V_0:=\emptyset$, and $V_l:=D'_l$ for all $l\in [4]$, there is a partition $ (S^1_l,\ldots,S^s_l )$ of $D'_l$ for each $l\in [4]$, satisfying
\begin{itemize}[itemsep=0pt, topsep=0.5pt,parsep=1pt]
\item[$(A1)$] $|S^{i}_l| = m^i_l$, for all $i\in [s] $;
\item[$(A2)$] For each $\sigma \in \{ +, - \}$, if $v\in V(D)$ has $d^{\sigma}_{V_l}(v)\geqslant \varepsilon |V_l|$, then $
|N^{\sigma}_D(v) \cap S^i_l| \geq \frac{ d^{\sigma}_{V_l}(v) |S^i_l|}{|V_l| } - 2sn^{2/3} $ for all $i\in [s] $.
\end{itemize}
{(For the case $\sum_{i \in [q]} n_i < n - 100q\rho^2 n$, let $m^{i}_l=\left\lfloor \frac{n_i - 1}{4}\right\rfloor$ for each $l\in [4]$ and $i\in [s]$, and set $m^{s+1}_l= |D'_l| - \sum_{i=1}^s m^{i}_l$. By Lemma \ref{lemma2}, each $D'_l$ admits a partition $(S^1_l,\ldots,S^{s+1}_l )$ for $l\in [4]$ that satisfies $(A1)$–$(A2)$. For each $l\in [4]$, select $s$ vertices from the subset $S_l^{s+1}$, add them to $S_l^1$, and ensure that $|D_2|=|D_4|+b$. Retain only the subsets $S_l^1,\dots,S_l^{s}$ for subsequent proofs, discarding the remaining vertices in $S_l^{s+1}$. In particular, if $H$ is a loop (i.e., $V(H)=\{u\}$), we partition $D$ into two subgraphs $G_1$ and $G_2$ by applying Lemma \ref{lemma2}. We then place the vertex $u$ (or the contracted vertex corresponding to the path $Q_i$ containing $u$) into $G_1$ such that $|G_1|\geq n_1$, with all remaining vertices discarded. The subsequent proof for this case is identical to that for the case where $\sum_{i \in [q]} n_i \geq n - 100q\rho^2 n$. Hence, we present only the proof for the case  $\sum_{i \in [q]} n_i \geq n - 100q\rho^2 n$.)}

Define $G_{i} = D'[\bigcup_{l=1}^{4} S_l^i] \text{ for all } i \in [2, s],$ and $G_{1} =D'[C\cup  \bigcup_{l=1}^{4} S_l^1].$  For each $i\in [s]$, pinch vertices $u_i,v_i$ with $X=D_{i^+},Y=D_{i^-}$ into a vertex $ x_{u_i,v_i}$, then put $x_{u_i,v_i}$ into $S^i_{t_i}$ to get an oriented graph, still call it $G_i$. The selection of $t_i$ depends on the indices $i^+,i^-$ as previously fixed. Namely,
\begin{itemize}[itemsep=0pt, topsep=0.5pt,parsep=1pt]
\item when $i^+\in \{1,2\}$, if $i^-\in \{1,4\}$ then let $t_i = 1$, and if $i^-\in \{2,3\}$, let $t_i= 4$;
\item when $i^+\in \{3,4\}$, if $i^-\in \{2,3\}$ then let $t_i = 3$, and if $i^-\in \{1,4\}$, let $t_i = 2$.
\end{itemize}
 This leads to \( ||S^{i}_2| - |S^{i}_4||\leq  1\) for all $i\in [2,s]$, then we adjust at most one vertex between $S^{i}_2$ and $S^{1}_2\setminus C$ or between $S^{i}_4$ and $S^{1}_4\setminus C$ to ensure \( |S^{i}_2| = |S^{i}_4|\). The selection of \(t_i\) compensates for the portion of the discrepancy between \(|D_2|\) and \(|D_4|\)  equivalent to $b(P_i)$, together with $|D_2|=|D_4|+b$, we naturally have \( |S^{1}_2| = |S^{1}_4|\) now. Also, by the definition of $1/400$-acceptable, \(x_{u_i,v_i}\) is a $1/400 $-acceptable vertex for the partition $( S_1^i,S_2^i,S_3^i,S_4^i )$ if \(x_{u_i,v_i}\) is put into \(S_{t_i}^{i}\) (these are why we choose such \(t_i\)).  {(When \(H\) is a loop, no vertex pinching or vertex placement operations are required.} Moreover, since we apply Lemma \ref{4.2} (ii), it naturally holds that \(|S_1^2| = |S_1^4|\).)

For each $i\in [2,s]$, it follows from $(A1)$ that $n_i - 4\leq |G_{i}| = 4\lfloor\frac{n_{i} - 1}{4}\rfloor + 1 \leq n_i $. To ensure that $|G_{i}| = n_{i}$, we remove at most 4 vertices from $S^1_{1}\setminus C$ and add them to $S^{i}_1$. Notice that $\sum_{i \in [q]} n_i \geq n - 100q\rho^2 n$ and some disjoint paths may have been contracted. Thus, to get \(|G_1|=n_1\) (the vertices of the contracted path are also taken into account), we remove some redundant vertices from each set \(S_l^1\), $l\in [4]$ with equality as much as possible. And remove at most $100q\rho^2 n $ vertices in total, while keeping \(|S^1_2| = |S^1_4|\). This results in the number of vertices removed from different sets differing by at most two. Then $|G_1|\geq n_1-650\mu n-100q\rho^2 n\geq 50q\rho n$ as (\ref{para1}).


Now, for each $i\in [s]$, we verify that $( S_1^i,S_2^i,S_3^i,S_4^i )$ satisfies the condition of Lemma \ref{4.3} (i)-(iv). It is easy to check that $ |S^i_{l}|=(1/4\pm 4\mu ) |G_i|$ for all $i\in [2,s]$. According to the construction of $ (S^{1}_1,S^1_{2},S^1_{3},S^1_{4})$, $$\left||S^1_{l'}|-|S^1_{l''}|\right|\leq 200\mu n+650\mu n+65\mu n+4s+2<10^4q\mu n\stackrel{(\ref{para1})}\leq \xi |G_1|,$$ for any $l',\ l''\in [4]$. Hence,  $ |S^1_{l}|=(1/4\pm \xi) |G_1|$ for all $l\in [4]$, then ($i$) holds. If $M$ is a subset of $S^i_3$ with order at least $ \xi n_i $, then $M$ contains a subset of $D^*_3\setminus \bigcup_{i=1}^r V(Q_i)$ with order at least $200\sqrt{\mu}n$ as $\xi \gg \sqrt{\mu}$, so $M$ contains a 3-path by Lemma \ref{4.1} (i), then ($ii$) holds. Now, we verify ($iv$). Since $(D_1,D_2,D_3,D_4)$ is an extremal partition with parameters $ (10^3, 10^6,  1 /400,1-10^6\sqrt{\mu}) $, every vertex $v\in V(D\setminus B) $ is $1/400$-acceptable for the partition $(D_1,D_2,D_3,D_4)$. Assume $v$ satisfies $D_1:(D_2)^{>1/400}(D_4)_{>1/400}$, so the number of out-neighbours of $v$ in the set $D'_2$ is at least
$$ |D_2|/400-|\bigcup_{i=s+1}^q V(P_i)\cup C|\geq (1/400-10^3q\rho )|D_2|\geq (1/400-10^3q\rho )|D'_2|.$$
 According to (A2), (\ref{para1}) and (\ref{para2}), the number of out-neighbours of $v$ in each $S^2_i$ ($i\in [2,s]$) is at least $(1/400-10^3q\rho )m^2_i - 2sn^{2/3}-4s \geq  m^2_i/600$ (the number of out-neighbours of $v$ in each $S^2_1$ is at least $(1/400-10^3q\rho )m^2_1 - 2sn^{2/3}-4s-100q\rho^2 n \geq  m^2_1/600$). Similarly, the number of in-neighbours of $v$ in each $S^4_i$ ($i\in [s]$) is at least $ m^4_i/600$. This implies that every vertex in $G_{i}$ is $1/600$-acceptable for $( S_1^i,S_2^i,S_3^i,S_4^i )$. Likewise, (\ref{para1}) yields that $$(1 - 10^6\sqrt{\mu })m^l_i - 2sn^{2/3}-4s-100q\rho^2 n\geq (1 - 20q \xi)m^l_i,$$
thus every $(1-10^6\sqrt{\mu })$-circular vertex $v\in V(D)$ is $(1 - 20q \xi)$-circular for $( S_1^i,S_2^i,S_3^i,S_4^i )$ if $v\in V(G_{i})$. Until now, for $( S_1^i,S_2^i,S_3^i,S_4^i )$, every vertex in $G_{i}$ is $1/600$-acceptable and the number of non $(1 - 20q \xi)$-circular vertices is at most $\sqrt{\mu } n\leq \sqrt{\mu } |G_{i}|/\rho\leq \xi |G_{1}|$.

In order to verify ($iii$) in Lemma \ref{4.3}, {let $L$ be the subset of $D_4$ such that every vertex in $L$ has at least $ n/32$ out-neighbours in $D_3\cup D_4$. It follows that $|L|n/32 \leq a(D_4) + a(D_4, D_3)\leq 10^6\mu n^2$. This implies that $L$ has at most $10^8\mu n $ vertices. Since $|D_1|\leq n/4 +  10^3\mu n$ and $\delta^0(D\setminus B)\geq (3n - 5k-5)/8$, every vertex in $D_4 \setminus L$ has at least $ n/16$ out-neighbours and in-neighbours in $D_2$. Therefore, all but at most $10^8\mu n$ vertices in $D_4$ have at least $|D_2|/32$ out-neighbours and at least $|D_2|/32$ in-neighbours in $D_2$.} Similarily, by (A2), all but at most $10^8\mu n\leq \xi |G_i|$ vertices in $S^i_4$ have at least $|S^i_2|/32$ out-neighbours and at least $|S^i_2|/32$ in-neighbours in $S^i_2$.   In summary, $( S_1^i,S_2^i,S_3^i,S_4^i )$ satisfies the condition of Lemma \ref{4.3}.

Lemma \ref{4.3} yields a Hamilton cycle $C_i$ in $G_i$ which encounters some contraction vertices. Thus there is a new cycle $C'_i = c_1c_2\cdots c_{j - 1}x_{u_i, v_i}c_{j + 1}\cdots c_1$ containing the special vertex $x_{u_i, v_i}$ with length $n_i - 1$ by replacing each contracting vertex $q_j$ by the path $Q_j$ in $C_i$. Then the path $P_i = u_ic_{j + 1}\cdots c_1c_2\cdots c_{j - 1}v_i$ is the desired subdivision path with length $n_i$. In this way, we get all long subdivision paths. This completes the proof of this case.
\end{proof}

\subsection{Non-extremal case}

\begin{lemma}\label{4.7}
Let $ H $ be a multi-digraph with $ h $ vertices and $ q $ arcs. Let $ D $ be an oriented graph of order $ n \geq n_0$ with $\delta^0(D) \geq \frac{3n - 4}{8}$. If $ D $ is not an extremal case with parameters $(31$, $10^3$, $1/200$, $1-10^3\sqrt{\mu })$, then \(D\) is arbitrary  \(H\)-linked.
\end{lemma}

\begin{proof}
{In accordance with the specifications in (\ref{para1}) and (\ref{para2}) for the constants, we further assume the ordering}. We start by constructing iteratively a sequence of subdivision paths \(P_{q}, \ldots, P_{s+1}\) of length at most \(\theta n\) using Lemma~\ref{lemma3}: At iteration $ i $, let $ W = \bigcup_{j = q-i+1}^q P_j $ then $ |W| < q\theta n .$ For $ \widehat{D} = (D \setminus (W \cup B)) \cup \{u_{q - i},v_{q - i}\} $,
$$
\delta^0(\widehat{D}) \geq \tfrac{3n-4}{8} - (h + q\theta n) \stackrel{(\ref{para1})(\ref{para2})}> \tfrac{n}{3} + 6.
$$
 Apply Lemma~\ref{lemma3} with $ l:=|P_{q - i}| < \theta n \stackrel{(\ref{para1})}< |\widehat{D}|/10^{10}$ to get the path $P_{q - i}$ in $\widehat{D}=(D\setminus (W\cup B))\cup \{u_{q - i},v_{q - i}\}$. Therefore, we obtain all short subdivision paths. Update $W := \bigcup^{q}_{i = s + 1} P_{i}$, and then we have $|W|< q\theta n $.

For the construction of long paths, we first apply {Diregularity Lemma} and random partition to divide the remaining oriented graph into $s$ subgraphs, each with a specified order. Through pinch vertex operation, we transform the problem of constructing long subdivision paths into finding Hamilton cycles in these subgraphs. Finally, we complete the proof by using Lemma \ref{4.5}.

 Let $ D^* = D - (W \cup B) $, then
$
\delta^0(D^*) \geq \tfrac{3n-4}{8} - 2q\theta n \geq (\tfrac{3}{8} - 2d)n
$ as (\ref{para1}).
Apply the Diregularity Lemma with inputs $\varepsilon, d$ and $M'$ to obtain a partition $(V_0,V_1,V_2,\ldots,V_M)$ of $V(D^*)$. As $\varepsilon  \ll d$, Lemma \ref{lem2} immediately leads to a reduced oriented graph $R$ of order $M$ with $\delta^0(R)\geq (3/8 - 6d)|R|.$ According to Lemma \ref{4.4}, $R$ is robust $(\mu,1/3)$-outexpander{. Then} $R$ is also a robust $(\mu, \tau)$-outexpander by Fact \ref{fac1}.

Assume that $\sum_{i \in [q]} n_i \geq n - 100q\theta n$. Define for $ 1 \leq l \leq M $:
$$
m_l^i = \lfloor \tfrac{n_i - 1}{M} \rfloor \quad (2 \leq i \leq s), \quad m_l^1 = |V_l| - \sum_{i=2}^s m_l^i.
$$
So $ m^i_l\geq \theta n/2M$ for all $m^i_l$ ($m^1_l>\theta n/2M$ similarly as  (\ref{11})). By Lemma \ref{lemma2} with $D:=D^*$, $\varepsilon:=\varepsilon,\ s:=\varepsilon,\ M:=M,\{m^{i}_l, i\in [s], l\in [M]\}$, $V_l:=V_l$ for all $0\leq l\leq M$, there exists a partition $ (S^1_l,\ldots,S^s_l) $, for each $V_l\ (l\in [M])$, satisfying
\begin{itemize}[itemsep=0pt, topsep=0.5pt,parsep=1pt]
\item[(C1)] $ |S_l^i| = m_l^i $ for all $i\in [s]$;
\item[(C2)] For each $\sigma \in \{ +, - \}$, if $v\in V(D)$ has $d^{\sigma}_{V_l}(v)\geqslant \varepsilon |V_l|$, then $
|N^{\sigma}_D(v) \cap S_l^i| \geq \frac{ d^{\sigma}_{V_l}(v) |S_l^i|}{|V_l| } - 2sn^{2/3}$ for all $i\in [s]$.
\end{itemize}
 {(For the case \(\sum_{i \in [q]} n_i < n - 100q\theta n\), define \(m_l^i = \lfloor \frac{n_i - 1}{M} \rfloor \quad (1 \leq i \leq s),\quad
m_l^{s+1} = |V_l| - \sum_{i=1}^s m_l^i.\) By Lemma \ref{lemma2}, each \(V_i\) is partitioned into \(s+1\) subsets \((S_l^1,\dots,S_l^{s+1})\). For each \(l\in [M]\), we select $s$ vertices from each vertex subset \(S_l^{s+1}\) and add them to the set \(S_l^1\), and $(C2)$ still holds. We then retain only the subsets \(S_l^1,\dots,S_l^{s}\) for subsequent proofs, discarding all remaining vertices in \(S_l^{s+1}\). The subsequent proof for this case is identical to that for the case where \(\sum_{i \in [q]} n_i \geq n - 100q\theta n\). Hence, we present only the proof of the case \(\sum_{i \in [q]} n_i \geq n - 100q\theta n\))}

Initially, set $S^{i}_0: = \emptyset$  for each $i \in [s]$ and set $\widetilde{S}_{i} := S^{i}_0 \cup \bigcup_{l \in [M]} S^{i}_l$. Now we adjust some vertices so that $ |\widetilde{S}_i| = n_i $. Recall that $u_i$ and $v_i$ are endvertices of path $P_i$ of length $n_i$ that we need to construct. For each $i\in [s]$, pinch vertices $u_i,v_i$ with $X=Y=V(D)$ into a vertex $ x_{u_i,v_i}$, then update $S^{i}_0:=S^{i}_0\cup \{x_{u_i,v_i}\}$. Presently, for $i \in [2, s]$, we have $n_i-M \leq |\widetilde{S}_{i}| \leq n_i $. To guarantee that $|\widetilde{S}_{i}| = n_{i} $ for all $i \in [2,s]$, we extract $n_{i}- |\widetilde{S}_{i}| < M$ vertices from $V_0 \cup \widetilde{S}_{1}$ and append them to $S^{i}_0$ for each $i \in [2, s]$. By the setting of $m^i_l$, we have that $|S^{i}_1|=\cdots =|S^{i}_M|$ for all $i \in [2, s]$. We still denote the sets after vertex adjustment by $V_0$ and $\widetilde{S}_i \ (i \in [s])$. 

{At most $sM$ vertices were first removed from $\widetilde{S}_{1}$, yielding $|\widetilde{S}_{i}| = n_{i}$ for each $i\in [2,s]$. To ensure $|S_1^1|=\cdots =|S_M^1|$, we move at most $sM$ vertices from each $S_l^1$ to $S_0^1$. Then let \(S^1_0:=S^1_0\cup V_0\). However, required subdivision may not be Hamiltonian: $\widetilde{S}_{1}$ may contain redundant vertices unused in the subdivision. By the inequality $\sum_{i \in [q]} n_i\geq n - 100q\theta n$, at most $100q\theta n$ vertices are removed from $\widetilde{S}_{1}\setminus \{x_{u_1,v_1}\}$ to eliminate redundancy and achieve $|\widetilde{S}_1|=n_1$. Notably, the number of remaining vertices in each $S^1_l$ ($l\in [M]$) is kept as equal as possible during this elimination step.  Finally, at most one vertex is moved from each $S^1_l$ ($l\in [M]$) to $S^1_0$, ensuring $|S^{1}_1|=\cdots =|S^{1}_M|$.} Hence, $ |S^{i}_0|\leq 1+sM^2+\varepsilon n+M <\sqrt{\varepsilon} |\widetilde{S}_{i}|$ (since $|\widetilde{S}_{i}|=n_i\geq \theta n \gg \varepsilon n$). Until now, we have that for each $i\in [s]$, $|\widetilde{S}_{i}|= n_i$, $|S^{i}_1|=\cdots =|S^{i}_l|$ and $ |S^i_0|\leq \sqrt{\varepsilon} |\widetilde{S}_{i}|$.

To utilize Lemma \ref{4.5}, we assert that $ \delta ^0( \widetilde{S}_{i})\geq 4\eta |\widetilde{S}_{i}|$. By $(C2)$, for every vertex $v \in V(D)$, $i\in [s]$ and $\sigma \in \{ +, - \}$, we derive that:
\begin{align*}
d^{\sigma}_{\widetilde{S}_{i}}(v) \geq \frac{\sum_{l \in [M]} d^{\sigma}_{S^i_l}(v)- sM^2}{2}&\geq \sum_{l \in [M] : d^{\sigma}_{V_l}(v)\geq \varepsilon |V_l|} \left(\frac{ d^{\sigma}_{V_l}(v) |S_l^i|}{|V_l| }  -2sn^{2/3}\right)/2- sM^2   \\
&\geq \frac{ |S_l^i|}{2|V_l|}\cdot \sum_{l \in [M] : d^{\sigma}_{V_l}(v)\geq \varepsilon |V_l|} d^{\sigma}_{V_l}(v)-sMn^{2/3}- sM^2 \\&\geq \frac{ |S_l^i|}{2|V_l|}(3/8 - q\theta-\varepsilon - M\varepsilon)n - 2sMn^{2/3} \stackrel{(\ref{para1})}> 4\eta n_i.
\end{align*}
For each $i \in [s]$, define $R_{i}$ as a digraph with vertex set $V(R)$. An arc {$uv$} belongs to $R_i$ if and only if it is an arc of $R$. Evidently, $R_i$ and $R$ are isomorphic. Therefore, $R_i$ is a robust $(\mu, \tau)$-outexpander, and $\delta^0(R_i) = \delta^0(R ) \geq (3/8 - 6d)|R|$, as required. Also, $R_i$ is an oriented graph.
\begin{claim}
For all \(i \in [s]\), the oriented graph \(R_i\) is a reduced oriented graph of \(\widetilde{S}_{i}\) corresponding to the partition \((S^0_{i},S^1_{i}, \ldots,S^M_{i})\), {with parameters \(\sqrt{\varepsilon}\) and \(5d/6\) in place of \(\varepsilon\) and $d$ respectively.}
\end{claim}
\begin{proof}
Note that for every arc $l'l'' \in A(R)$, the pair $(V_{l'}, V_{l''})$ is an $\varepsilon$-regular pair with density at least $d$. A straightforward calculation follows that for each $i \in [s]$,
$$|S^{l'}_{i}| = \frac{n_i - |S^0_i|}{M} \geq \frac{(1 - \sqrt{\varepsilon})n_i}{M} \geq \frac{(1 - \sqrt{\varepsilon})\theta n}{M}\stackrel{(\ref{para1})} > \theta |V_{l'}|/2 > \varepsilon |V_{l'}|.$$
Similarly, $|S^{l''}_{i}| > \theta |V_{l'}|/2 $. The $\varepsilon$-regularity of $(V_{l'}, V_{l''})$ implies that $|d(S^{l'}_{i}, S^{l''}_{i}) - d(V_{l'}, V_{l''})| < \varepsilon$. Hence, $d(S^{l'}_{i}, S^{l''}_{i}) \geq d - \varepsilon \geq 5d/6$. For any two vertex sets $U'\subseteq S^{l'}_{i}$ of order $\sqrt{\varepsilon}|S^{l'}_{i}| $ and $U''\subseteq S^{l''}_{i}$ of order $ \sqrt{\varepsilon}|S^{l''}_{i}|$, because $ \sqrt{\varepsilon}|S^{l'}_{i}|> \sqrt{\varepsilon}\theta |V_{l'}|/2 >\varepsilon |V_{l'}|,$ {we have that} $|d(U', U'') - d(V_{l'}, V_{l''})| < \varepsilon$. This implies that $|d(U', U'') - d(S^{l'}_{i}, S^{l''}_{i})| < 2\varepsilon<\sqrt{\varepsilon}$, then by definition, $(S^{l'}_{i}, S^{l''}_{i})$ is $\sqrt{\varepsilon}$-regular pair. In light of the construction of $R$, if $(V_{l'}, V_{l''})$ is $\varepsilon$-regular with density at as least $d$, then $(S^{l'}_{i}, S^{l''}_{i})$ is $\sqrt{\varepsilon}$-regular pair with density at least $5d/6$. This leads to the conclusion that $R_{i}$ is the reduced oriented graph of $\widetilde{S}_{i}$ corresponding to the partition $(S^0_{i},S^1_{i}, \ldots,S^M_{i})$, {with parameters \(\sqrt{\varepsilon}\) and \(5d/6\) in place of \(\varepsilon\) and $d$ respectively,} as desired.
\end{proof}

Applying Lemma \ref{4.5} with $D := \widetilde{S}_i$ and reduced oriented graph $R_i$, the oriented graph $\widetilde{S}_i$ contains a Hamilton cycle $C_i$. Assume that $C_i = c_1c_2 \cdots c_{j - 1}x_{u_i,v_i}c_{j + 1} \cdots c_1$. Then the path $P_i = u_ic_{j + 1} \cdots c_1c_2 \cdots c_{j - 1}v_i$ is the desired subdivision path with length $n_i$. Similarly, we can obtain all the remaining long subdivision paths, which completes the proof.
\end{proof}

\section*{Acknowledgements}
{We are grateful to the reviewers for the careful readings and suggestions that improve the presentation.}

\end{document}


\reduceabovedisplayskip
\reducebelowdisplayskip
\title{Appendix of Semi-Degree Condition for Arbitrary $H$-Linked Oriented Graphs\thanks{The author's work is supported by National Natural Science Foundation of China (No.12071260)}}

\author{Jia Zhou, Jin Yan\thanks{Corresponding author. E-mail adress: yanj@sdu.edu.cn}  \unskip\\[2mm]
School of Mathematics, Shandong University, Jinan 250100, China}

\date{}
\maketitle

\appendix

\section{Appendix: Additional calculations}

\begin{property}
Let $ h$ be a positive integer and $\mu \ll 1/10^{10}$ be a real number. There is an integer $n_0:=n_0(\mu, h)$ such that the following holds. Suppose $D$ is an extremal case of order $ n \geq n_0 $ with parameters $(31$, $10^3$, $1/200$, $1-10^3\sqrt{\mu })$. Let $ B\subseteq V(D) $ be a subset of order $h$. Then the following statements hold.
\begin{itemize}[itemsep=0pt, topsep=0.5pt,parsep=1pt]
\item[$(I)$] $ D \setminus B $ admits an extremal partition $(D_1^*, D_2^*, D_3^*, D_4^*)$ with parameters $(32$, $10^4$, $1/300$, $1-10^4\sqrt{\mu })$.
\item[$(II)$] By relocating at most $ 65\mu n $ vertices from $ (D^*_1, D^*_2, D^*_3, D^*_4) $, there exists an extremal partition $ (D_1, D_2, D_3, D_4)  $ with parameters $ (100, 10^5, 1/350, 1 - 10^5\sqrt{\mu}) $. Then contract a collection of disjoint paths $ Q_1, \ldots, Q_r $ with $ \left|\bigcup_{i=1}^r V(Q_i)\right| \leq 650\mu n $, there is an extremal partition $(D'_1, D'_2, D'_3, D'_4) $ with parameters $ (10^3, 10^6, 1/400, 1 - 10^6\sqrt{\mu}) $.
\end{itemize}
\end{property}

\begin{proof}
For (I), by the assumption, there is an extremal partition $ (D^0_1, D^0_2, D^0_3, D^0_4) $ of $V(D)$ with parameters $(31$, $10^3$, $1/200$, $1-10^3\sqrt{\mu })$ satisfying:
\begin{itemize}[itemsep=0pt, topsep=0.5pt,parsep=1pt]
\item[$(i)$] $|D^0_i|=(1/4\pm 31\mu)n$, for $i\in [4]$.
\item[$(ii)$] $a(D^0_i, D^*_{i+1}) > (1-10^3\mu)n^2/16$, for $i\in [4]$ $(mod \  4)$; $a(D^0_i)>(1/2- 10^3\mu)n^2/16$, for $i\in \{1,3\}$;  $a(D^0_i, D^0_j) > (1/2- 10^3\mu)n^2/16$, for $i,j\in \{2,4\}$ with $i\neq j$; $a(D^0_4) + a(D^0_4, D^0_3)\leq 10^3\mu n^2$;
\item[$(iii)$] Every vertex of $D$ is \(1/200\)-acceptable, and the number of vertices that are not \(1-10^3\sqrt{\mu })\)-circular is at most \(\sqrt{\mu} n\) in \(D\).
\end{itemize} 
Since $|B|=h$, let $n_0$ be a positive integer such that $\mu n_0>h$, there is 
\begin{itemize}[itemsep=0pt, topsep=0.5pt,parsep=1pt]
\item[$(i)$] $(1/4- 32\mu)n\leq (1/4- 31\mu)n-h\leq |D^0_i\setminus B|\leq (1/4+ 31\mu)n$, for $i\in [4]$.
\item[$(ii)$] $a(D^0_i\setminus B, D^0_{i+1}\setminus B) > (1-10^3\mu)n^2/16-hn>(1-10^4\mu)n^2/16$, for $i\in [4]$ $(mod \  4)$; $a(D^0_i\setminus B)>(1/2- 10^3\mu)n^2/16-hn>(1/2-10^4\mu)n^2/16$, for $i\in \{1,3\}$;  $a(D^0_i\setminus B, D^0_j\setminus B) > (1/2- 10^3\mu)n^2/16-hn>(1/2-10^4\mu)n^2/16$, for $i,j\in \{2,4\}$ with $i\neq j$; $a(D^0_4\setminus B) + a(D^0_4\setminus B, D^0_3\setminus B)\leq 10^3\mu n^2+hn\leq 10^4\mu n^2$.
\end{itemize}
For every vertex $v\in V(D)\setminus B$, $v$ is \(1/200\)-acceptable for $ (D^0_1, D^0_2, D^0_3, D^0_4) $, then the acceptablity of $v$ for $ (D^0_1\setminus B, D^0_2\setminus B, D^0_3\setminus B, D^0_4\setminus B) $ is at least  $(|D^0_i|/200 -h)/|D^0_i\setminus B|\geq \frac{|D^0_i\setminus B|}{300|D^0_i\setminus B|}=1/300$. Also, for every vertex $v\in V(D)\setminus B$,  if $v$ is \((1-10^3\sqrt{\mu })\)-circular for $ (D^0_1, D^0_2, D^0_3, D^0_4) $, then the circularity of $v$ for $ (D^0_1\setminus B, D^0_2\setminus B, D^0_3\setminus B, D^0_4\setminus B) $ is at least $\frac{ (1-10^3\sqrt{\mu })|D^0_i| -h}{|D^0_i\setminus B|}\geq (1-10^4\sqrt{\mu })$. Thus, let $D^*_i=D^0_i\setminus B$ for all $i\in [4]$. This completes the proof of (I).

\textbf{Proof of (II):} Since $ D \setminus B $ admits an extremal partition $ (D^*_1, D^*_2, D^*_3, D^*_4) $ with parameters $(32$, $10^4$, $1/300$, $1-10^4\sqrt{\mu })$, it follows the following conclusion:
\begin{itemize}[itemsep=0pt, topsep=0.5pt,parsep=1pt]
\item[$(i)$] $|D^*_i|=(1/4\pm 32\mu)n$, for $i\in [4]$.
\item[$(ii)$] $a(D^*_i, D^*_{i+1}) > (1-10^4\mu)n^2/16$, for $i\in [4]$ $(mod \  4)$; $a(D^*_i)>(1/2- 10^4\mu)n^2/16$, for $i\in \{1,3\}$;  $a(D^*_i, D^*_j) > (1/2- 10^4\mu)n^2/16$, for $i,j\in \{2,4\}$ with $i\neq j$; $a(D^*_4) + a(D^*_4, D^*_3)\leq 10^4\mu n^2$;
\item[$(iii)$] Every vertex of $D\setminus B$ is \(1/300\)-acceptable, and the number of vertices that are not \(1-10^4\sqrt{\mu })\)-circular is at most \(\sqrt{\mu} n\) in \(D\setminus B\).
\end{itemize} 
After relocating at most $ 65\mu n $ vertices,  the adjusted partition satisfies $ |D_i| = \left( \frac{1}{4} \pm 100\mu \right)n$ for $  i \in [4]. $ Also, the arc modification is bounded as follows: At most \(64\mu n\) vertices are moved, and the number of arcs that are destroyed or newly added is $
  \Delta_1 = 64\mu n \cdot n = 64\mu n^2.$ Hence, 
  \begin{itemize}[itemsep=0pt, topsep=0.5pt,parsep=1pt]
\item $
a(D_i, D_{i+1}) > \left(1 - 10^4\mu\right)\frac{n^2}{16} - 65\mu n^2>\left(1 - 10^5\mu\right)\frac{n^2}{16},
$ for $i\in [4]$ (mod  4);
\item $a(D_i)>(1/2- 10^4\mu)n^2/16 - 65\mu n^2>\left(1/2 - 10^5\mu\right)\frac{n^2}{16}$, for $i\in \{1,3\}$;  
\item $a(D_i, D_j) > (1/2- 10^4\mu)n^2/16- 65\mu n^2>\left(1 - 10^5\mu\right)\frac{n^2}{16}$, for $i,j\in \{2,4\}$ with $i\neq j$;
\item  $a(D_4) + a(D_4, D_3)\leq 10^4\mu n^2+65\mu n^2<10^5 \mu n$.
\end{itemize}    
Moving vertices may disrupt the original adjacency structure, but the affected quantity is extremely small, at most \(65\mu n\). Therefore, the acceptability of all vertices is still greater than \(|D_i^*|/300 - 65\mu n > |D_i|/350\), and the first item in condition (iii) holds. Similarly, the vertices that were \((1 - 10^4\mu)\)-circular previously have a circularity of at least \((1 - 10^4\mu)|D_{i}^*|-65\mu n>(1 - 10^5\mu)|D_{i}|\) now. It should be noted that during our proof process, the moved vertices maintain $1/400$-acceptable, and the vertices obtained from the contracted paths maintain $(1 - 10^5\mu)$-circular. Namely, $ (D_1, D_2, D_3, D_4)  $ is an extremal partition with parameters $ (100, 10^5, 1/350, 1 - 10^5\sqrt{\mu}) $ as desired.

After contracting paths that cover at most $ \leq 650\mu n $ vertices,  the adjusted partition satisfies $ |D'_i| = \left( \frac{1}{4} \pm 10^3\mu \right)n$ for $  i \in [4]. $ Also, the arc modification is bounded as follows: The number of arcs that are destroyed or newly added is $
  \Delta_1 = 650\mu n \cdot n = 650\mu n^2.$ Hence, 
  \begin{itemize}[itemsep=0pt, topsep=0.5pt,parsep=1pt]
\item $
a(D_i, D_{i+1}) > \left(1 - 10^5\mu\right)\frac{n^2}{16} - 650\mu n^2>\left(1 - 10^6\mu\right)\frac{n^2}{16},
$ for $i\in [4]$ (mod  4);
\item $a(D_i)>(1/2- 10^5\mu)n^2/16 - 650\mu n^2>\left(1/2 - 10^6\mu\right)\frac{n^2}{16}$, for $i\in \{1,3\}$;  
\item $a(D_i, D_j) > (1/2- 10^5\mu)n^2/16- 650\mu n^2>\left(1 - 10^6\mu\right)\frac{n^2}{16}$, for $i,j\in \{2,4\}$ with $i\neq j$;
\item  $a(D_4) + a(D_4, D_3)\leq 10^5\mu n^2+650\mu n^2<10^6 \mu n$.
\end{itemize}    
Contracting paths may disrupt the original adjacency structure, but the affected quantity is extremely small, at most \(650\mu n\). Therefore, the acceptability of all vertices is still greater than \(|D_i^*|/350 - 650\mu n > |D_i|/400\), and the first item in condition (iii) holds. Similarly, the vertices that were \((1 - 10^5\mu)\)-circular previously have a circularity of at least \((1 - 10^5\mu)|D_{i}^*|-650\mu n>(1 - 10^6\mu)|D_{i}|\) now. It should be noted that during our proof process, the moved vertices maintain $1/400$-acceptable, and the vertices obtained from the contracted paths maintain $(1 - 10^6\mu)$-circular. Namely, $ (D'_1, D'_2, D'_3, D'_4)  $ is an extremal partition with parameters $ (10^3, 10^6, 1/400, 1 - 10^6\sqrt{\mu}) $ as desired. 
\end{proof}

\begin{lemma}[Lemma 4.2 (ii)]
Let $ h=1$ and $\mu \ll 1$ be a real number. There is an integer $n_0:=n_0(\mu, h)$ such that the following holds. Let $ D $ be an oriented graph of order $ n \geq n_0 $. Suppose $ B \subseteq V(D) $ is a subset of order $h$. If $ D \setminus B $ admits an extremal partition $ (D^*_1, D^*_2, D^*_3, D^*_4) $ with parameters $(32$, $10^4$, $1/300$, $1-10^4\sqrt{\mu })$. If $ \delta^0(D) \geq \frac{3n  - 4}{8} $, then there exists an extremal partition $ (D_1, D_2, D_3, D_4) $ with parameters $ (10^3, 10^6, 1/400,\ 1-10^6\sqrt{\mu}) $ such that $ |D_2| = |D_4| $. This partition is obtained by adjusting the location of at most $ 65\mu n $ vertices from $ (D^*_1\cup B, D^*_2, D^*_3, D^*_4) $ and contracting a collection of disjoint paths $ Q_1, \ldots, Q_r $ in $ D $ with $ \left|\bigcup_{i=1}^r V(Q_i)\right| \leq 650\mu n $.
\end{lemma}

\begin{proof}
In this case, let $b=0$ and the vertex in $B$ can be both adjusted and contracted. Additionally, some computations also differ from those in (i). First, we assign the vertex in \(B\) to the partition $ (D^*_1, D^*_2, D^*_3, D^*_4) $ of \(D\setminus B\) to obtain a partition of \(V(D)\).  By Lemma 4.1 (ii), for the vertex $u\in B$, there exist two indices $i, j\in [4]$ with $i\equiv j + 1$ or $i\equiv j + 2$ (mod 4) such that $d^+_{D^*_i\setminus S}(u)\geq n/48$ and $d^-_{D^*_j\setminus S}(u)\geq n/48$. Then, \(u\) can be placed into a set \(D^*_k\) such that \(u\) is a \(1/100\)-acceptable vertex, and the acceptability of the remaining vertices remains unchanged. Finally, update the partition \((D^*_1, D^*_2, D^*_3, D^*_4)\), that is, it is a partition of \(V(D)\). 

Provided that $|D^*_2|\neq |D^*_4|$, since otherwise, the conclusion holds trivially. Our aim is to rearrange the positions of a few vertices to reduce the difference between \(|D^*_2|\) and \(|D^*_4|\), while ensuring that the newly obtained partition remains an extremal partition with a slight increase in its parameters. We classify vertices as follows with $\gamma = 1/400$: if $|D^*_2|> |D^*_4|$, we call a vertex \textbf{good} if it is $\gamma$-acceptable, and it belongs to $D^*_4$ or has one of the properties
$D^*_1:(D^*_2)_{<\gamma}(D^*_3)_{<\gamma}$, $D^*_2:(D^*_1)^{<\gamma}(D^*_2)^{<\gamma}_{<\gamma}(D^*_3)_{<\gamma}$, $D^*_3 :(D^*_1)^{<\gamma}(D^*_2)^{<\gamma}$. In this case, $D^*_4$ is called the \textbf{good set} and $D^*_2$ is called the \textbf{bad set}. And if $|D^*_2|< |D^*_4|$, we call a vertex \textbf{good} if it is $\gamma$-acceptable, and it belongs to $D^*_2$ or has one of the properties
$D^*_1:(D^*_4)^{<\gamma}(D^*_3)^{<\gamma}$, $D^*_3 :(D^*_1)_{<\gamma}(D^*_4)_{<\gamma}$, $D^*_4:(D^*_1)_{<\gamma}(D^*_4)^{<\gamma}_{<\gamma}(D^*_3)^{<\gamma}$. In this case, $D^*_4$ is called the \textbf{bad set} and $D^*_2$ is called the \textbf{good set}. In either case, we call a vertex \textbf{bad} if it is not good.

 The following shows that the goal can be achieved by relocating the bad vertices.

\begin{claim}\label{claim5}
All bad vertices in the bad set can become \(\gamma\)-acceptable by arranging them in \(D^*_1\cup D^*_3\).
\end{claim}

\begin{proof}
W.l.o.g., assume $D^*_2$ is the bad set and $v$ is a bad vertex in $D^*_2$.
Then $v$ satisfies at least one of properties $D^*_2:(D^*_1)^{>\gamma}$, $D^*_2:(D^*_2)^{>\gamma}$, $D^*_2:(D^*_2)_{>\gamma}$ and $D^*_2:(D^*_3)_{>\gamma}$. By the $1/300$-acceptability of $v$, we get that $v$ has properties $D^*_2:(D^*_1)_{>1/300}$ or $D^*_2:(D^*_4)_{>1/300}$ and $D^*_2:(D^*_4)^{>1/300}$ or $D^*_2:(D^*_3)^{>1/300}$. Hence, if $v$ has properties $D^*_2:(D^*_1)^{>\gamma}$ or $D^*_2:(D^*_2)^{>\gamma}$, then we move the vertex $v$ to the part $D^*_1$. This means that $v$ has property $D^*_1:(D^*_1)^{>\gamma}(D^*_1)_{>1/300}$ or $D^*_1:(D^*_1)^{>\gamma}(D^*_4)_{>1/300}$ or $D^*_1:(D^*_2)^{>\gamma}(D^*_4)_{>1/300}$ or $D^*_1:(D^*_2)^{>\gamma}(D^*_1)_{>1/300}$. By the definition of $\gamma$-acceptable, $v$ is $\gamma$-acceptable after moving. And if $v$ has properties $D^*_2:(D^*_2)_{>\gamma}$ or $D^*_2:(D^*_3)_{>\gamma}$, then $v$ becomes $\gamma$-acceptable after moving the vertex $v$ into $D^*_3$. The case where \( D^*_4 \) is the bad set is analogous.
\end{proof}

\begin{claim}\label{claim4}
We can arrange every bad vertex in $D^*_1\cup D^*_3$ to be a good vertex by moving it to the good set.
\end{claim}

\begin{proof}
Here, we present the proof for the case where \( D^*_4 \) is the good set, as a similar analysis can be applied if \( D^*_2 \) is the good set. If $v$ is a bad vertex in $D^*_1$, we can obtain that $D^*_1:(D^*_2)_{>\gamma} $ or $D^*_1:(D^*_3)_{>\gamma} $. Since $v$ is a $1/300$-acceptable vertex, it follows that $D^*_1:(D^*_1)^{>1/300} $ or $D^*_1:(D^*_2)^{>1/300} $. It is easy to check that we move the vertex $v$ to $D^*_4$ then $v$ becomes $\gamma$-acceptable. Furthermore, it is good. The same argument applies for the case $v\in D^*_3$.
\end{proof}

 Define $r'=\left||D^*_2| - |D^*_4| \right|$, then $r'\leq 64\mu n$. According to the definition of the good set,  removing vertices from the bad set or adding vertices to the good set will reduce the gap between $|D^*_2|$ and $|D^*_4| $. Together with Claims \ref{claim5}-\ref{claim4}, by redistributing at most $r'$ the bad vertices in the union of the bad set and $D^*_1 \cup D^*_3$, we obtain a new partition $(D_1,D_2,D_3,D_4)$ and all vertices in $V(D)$ are good. By Definition 4.2, it is easy to check that $(D_1,D_2,D_3,D_4)$ is a extremal partition with parameters $ (100, 10^5, 1/350,1-10^5\sqrt{\mu}) $ {(for a detailed calculation, see Property A.1 (II) in Appendix)}. Then the conclusion follows if $|D_2| = |D_4|$. Therefore, suppose that $|D_2| \neq |D_4|$.

Now, define $r=\left||D_2| - |D_4|\right|$, then $r\leq 65\mu n$. Denote $A(X, Y)$ to be the set of arcs from $X$ to $Y$ and $a(X,Y)=|A(X,Y)|$. Indeed, some arcs can also be extended into disjoint paths. Furthermore, contracting these paths would reduce the difference between $|D_2|$ and $|D_4|$. We construct an oriented graph $H^{\prime}$ with $V(H^{\prime})= V(D\setminus B)$ as follows:

\indent $ \bullet$ $A(H^{\prime})=A(D_2, D_1)\cup A(D_2)\cup A(D_3, D_1)\cup A(D_3, D_2)$, when $|D_2|> |D_4|$;

\indent $ \bullet$ $A(H^{\prime})=A(D_1, D_3)\cup A(D_1, D_4)\cup A(D_4, D_3)\cup A(D_4)$, when $|D_2|< |D_4|$.

\noindent In either case, choose $M$ to be a maximum matching in $H^{\prime}$. Let $L_i = V(M)\cap D_i$ for each $i\in [4]$.  Suppose that $a(M)\leq  r$.
Then the $r$ arcs $e_i$ of $M$ can be extended to $r$ disjoint paths $Q_i, i\in [r]$, where $Q_i$ contains exactly one arc $e_i$ in $M$, starting and ending at $(1 - 10^5\sqrt{\mu})$-circular vertices of the bad set and having the shortest possible length. This yields the result that each $Q_i$ has length at most 8, and uses two vertices from the bad set and the good set respectively. Thus $|\bigcup_{i\in [r]}V(Q_i)|< 10r<650 \mu n$. Further, after contracting these paths, there is a new partition $(D'_1,D'_2,D'_3,D'_4)$ such that $|D'_2|=|D'_4|$. By taking $\mu$ small enough and $n_0$ large enough, it is easy to check that $(D'_1,D'_2,D'_3,D'_4)$ is an extremal partition with parameters $ (10^3, 10^6, 1/400,1-10^6\sqrt{\mu}) $ {(for a detailed calculation, see Property A.1 (II) in Appendix)}. Then these disjoint paths $Q_1,\ldots,Q_r$ and the partition $(D'_1,D'_2,D'_3,D'_4)$ as desired. Now assume $a(M)>  r$. Next, we show that this contradicts the assumption $\delta ^0(D)\geq \frac{3n-4}{8}$. Now, we divide the proof into two cases: $|D_2| > |D_4|$ or $|D_2| <|D_4| $.

\textbf{Case 2.1: $|D_2| >|D_4|$}

The maximality of $M$, $\sum_{i\in [4]}|L_i|\leq 2a(M)$, and the definition of the good vertex yield that
\begin{equation*}\label{9}
\begin{aligned}
a(D_3, D_1)&\leq a(L_3, D_1) + a(D_3, L_1 )\leq (|L_3| + |L_1|)\gamma (n/4 + 32\mu n)\leq a(M)n/300.
\end{aligned}
\end{equation*}
Similarly, we also can obtain that $a(D_2,D_1)\leq a(L_2, D_1) + a(D_2, L_1) < a(M)n/300,$ $a(D_3,D_2)\leq a(L_3, D_2) + a(D_3, L_2 ) < a(M)n/300,$ and $a(D_2)\leq 2|L_2|\cdot \gamma (n/4+32\mu n)<a(M)n/150.$
Hence, it follows that
$$\sum_{v\in D_1} d_{D}^-(v)\leq \frac{|D_1|\cdot (|D_1-1|)}{2}+ a(M)n/150+ |D_1|\cdot |D_4|,$$
$$\sum_{v\in D_2} d_{D}(v)\leq (|D_1|+|D_3|)\cdot |D_2| + 2a(M)n/150+|D_2|\cdot |D_4|,$$
$$\sum_{v\in D_3} d_{D}^+(v)\leq \frac{|D_3|\cdot (|D_3-1|)}{2}+ a(M)n/150+ |D_3|\cdot |D_4|.$$

 By pigeonhole principle, there are three vertices $u\in D_1, v\in D_2, w\in D_3$ such that
$$d_{D}^-(u)\leq \frac{|D_1|-1}{2}+|D_4|+\frac{a(M)n}{150 |D_1|}<\frac{|D_1|-1}{2}+|D_4|+\frac{a(M)}{30},$$
$$d_{D}(v)<|D_1|+|D_3|+|D_4|+ \frac{a(M)}{15},\ \text{ and } d_{D\setminus B}^+(w)<\frac{|D_3|-1}{2} + |D_4| + \frac{a(M)}{30}.$$
Together with $|D_1|+|D_2|+|D_3|+|D_4|= n$, it yields that
\begin{equation*}
\begin{aligned}
4\times \frac{3n-4}{8}&\leq d_{D}^-(u) + d_{D}^+(w) + d_{D}(v) <\frac{3}{2}(n-|D_2| + |D_4|) + \frac{2a(M)}{15}-1.
\end{aligned}
\end{equation*}
 Thus $|D_2| - |D_4|=r>a(M)$ gives that $41a(M)/30<-1/2.$ This is impossible.

\textbf{Case 2.2: $|D_2|<|D_4|$.}

The similar calculations as above yields that $a(D_1, D_3)< a(M)n/300, $ $a(D_1,D_4)< a(M)n/300,$ $a(D_4,D_3) < a(M)n/300,$ and $a(D_4) <a(M)n/150.$ By Pigeonhole principle, there are three vertices $u\in D_1, v\in D_4, w\in D_3$ such that
$$d_{D}^+(u)<\frac{|D_1|-1}{2}+|D_2|+\frac{a(M)}{30},$$
$$d_{D}(v)<|D_1|+|D_3|+|D_2|+ \frac{a(M)}{15},\ \text{ and }d_{D}^-(w)<\frac{|D_3|-1}{2} + |D_2| + \frac{a(M)}{30}. $$
Together $|D_4|-|D_2|=r>a(M)$ with $|D_1|+|D_2|+|D_3|+|D_4|= n$, it yields that
\begin{equation*}
\begin{aligned}
4\times \frac{3n-4}{8}&\leq d_{D}^-(u) + d_{D}^+(w) + d_{D}(v) <\frac{3}{2}(n+|D_2| - |D_4|) + \frac{2a(M)}{15}-1.
\end{aligned}
\end{equation*}
This gives $|D_4|-|D_2| < 2/3+4 a(M) / 45.$ On the other hand, we previously observed that $r=|D_4|-|D_2| \geqslant a(M)+1$. Combining this gives $41a(M)/45< 0$, which is impossible. 
\end{proof}

\begin{lemma}[Lemma 4.2 (i)]\label{4.2}
Let $ h,q,b$ be positive integers with $b \leq q$ and $\mu \ll 1$ be a real number. There is an integer $n_0:=n_0(\mu, q)$ such that the following holds. Let $ D $ be an oriented graph of order $ n \geq n_0 $. Suppose $ B \subseteq V(D) $ is a subset of order $h$. If $ D \setminus B $ admits an extremal partition $ (D^*_1, D^*_2, D^*_3, D^*_4) $ with parameters $(32$, $10^4$, $1/300$, $1-10^4\sqrt{\mu })$, and $ \delta^0(D) \geq \frac{3n + 3h + 3q - 5}{8} $, then there exists an extremal partition $ (D_1, D_2, D_3, D_4) $ of $V(D)\setminus B$ with parameters $ (10^3, 10^6, 1/400,\ 1-10^6\sqrt{\mu}) $ such that {$ |D_4| - |D_2| = b $}. This partition is obtained by adjusting the location of at most $ 65\mu n $ vertices from $ (D^*_1, D^*_2, D^*_3, D^*_4) $ and contracting a collection of disjoint paths $ Q_1, \ldots, Q_r $ in $ D \setminus B $ with $ \left|\bigcup_{i=1}^r V(Q_i)\right| \leq 650\mu n $.
\end{lemma}

\begin{proof}\textbf{of Lemma \ref{4.2}.}
 The lemma is trivial if {\(|D^*_4| - |D^*_2| = b\)}. Consequently, we assume that \(|D^*_4| - |D^*_2| > b\) or \(|D^*_4| - |D^*_2| < b\). Our aim is to rearrange the positions of a few vertices to reduce the difference between \(|D^*_4|\) and \(|D^*_2| + b\), while ensuring that the newly obtained partition remains extremal with a slight change in its parameters. 
 
Given a partition $(O_1,O_2,O_3,O_4)$ of $V(D)$, we classify vertices as follows with $\gamma = 1/400$: if $|O_4| - |O_2| > b$, we call a vertex \textbf{good} if it is $\gamma$-acceptable, and it belongs to $O_2$ or has one of the properties
$O_1:(O_4)^{<\gamma}(O_3)^{<\gamma}$, $O_3 :(O_1)_{<\gamma}(O_4)_{<\gamma}$, $O_4:(O_1)_{<\gamma}(O_4)^{<\gamma}_{<\gamma}(O_3)^{<\gamma}$. In this case, $O_4$ is called the \textbf{bad set} and $O_2$ is called the \textbf{good set}. And if $|O_4| - |O_2| < b$, we call a vertex \textbf{good} if it is $\gamma$-acceptable, and it belongs to $O_4$ or has one of the properties
$O_1:(O_2)_{<\gamma}(O_3)_{<\gamma}$, $O_2:(O_1)^{<\gamma}(O_2)^{<\gamma}_{<\gamma}(O_3)_{<\gamma}$, $O_3 :(O_1)^{<\gamma}(O_2)^{<\gamma}$. In this case, $O_4$ is called the \textbf{good set} and $O_2$ is called the \textbf{bad set}. In either case, we call a vertex \textbf{bad} if it is not good.

 The following shows that the goal can be achieved by relocating the bad vertices. Now we consider the bad vertices in the initial partition $(D^*_1, D^*_2, D^*_3, D^*_4)$ of $V(D)$.

\begin{claim}\label{claim5}
Each bad vertex in the bad set can become \(\gamma\)-acceptable by arranging it in \(D^*_1\cup D^*_3\).
\end{claim}

\begin{proof}
W.l.o.g., assume $D^*_2$ is the bad set and $v$ is a bad vertex in $D^*_2$.
Then $v$ satisfies at least one of properties $D^*_2:(D^*_1)^{>\gamma}$, $D^*_2:(D^*_2)^{>\gamma}$, $D^*_2:(D^*_2)_{>\gamma}$ and $D^*_2:(D^*_3)_{>\gamma}$. By the $1/300$-acceptability of $v$, we get that $v$ has properties $D^*_2:(D^*_1)_{>1/300}$ or $D^*_2:(D^*_4)_{>1/300}$ and $D^*_2:(D^*_4)^{>1/300}$ or $D^*_2:(D^*_3)^{>1/300}$. Hence, if $v$ has properties $D^*_2:(D^*_1)^{>\gamma}$ or $D^*_2:(D^*_2)^{>\gamma}$, then we move the vertex $v$ to the part $D^*_1$. This means that $v$ has property $D^*_1:(D^*_1)^{>\gamma}(D^*_1)_{>1/300}$ or $D^*_1:(D^*_1)^{>\gamma}(D^*_4)_{>1/300}$ or $D^*_1:(D^*_2)^{>\gamma}(D^*_4)_{>1/300}$ or $D^*_1:(D^*_2)^{>\gamma}(D^*_1)_{>1/300}$. By the definition of $\gamma$-acceptable, $v$ is $\gamma$-acceptable after moving. And if $v$ has properties $D^*_2:(D^*_2)_{>\gamma}$ or $D^*_2:(D^*_3)_{>\gamma}$, then $v$ becomes $\gamma$-acceptable after moving the vertex $v$ into $D^*_3$. The case where \( D^*_4 \) is the bad set is analogous.
\end{proof}

\begin{claim}\label{claim4}
We can arrange every bad vertex in $D^*_1\cup D^*_3$ to be a good vertex by moving it to the good set.
\end{claim}

\begin{proof}
Here, we present the proof for the case where \( D^*_4 \) is the good set, as a similar analysis can be applied if \( D^*_2 \) is the good set. If $v$ is a bad vertex in $D^*_1$, we can obtain that $D^*_1:(D^*_2)_{>\gamma} $ or $D^*_1:(D^*_3)_{>\gamma} $. Since $v$ is a $1/300$-acceptable vertex, it follows that $D^*_1:(D^*_1)^{>1/300} $ or $D^*_1:(D^*_2)^{>1/300} $. It is easy to check that we move the vertex $v$ to $D^*_4$ then $v$ becomes $\gamma$-acceptable. Furthermore, it is good. The same argument applies for the case $v\in D^*_3$.
\end{proof}

Let $n_0$ be an integer such that $\mu n_0\geq q $. Define $r^*=\left||D^*_4| - (|D^*_2| + b)\right|$, then $r^*\leq 64\mu n+q \leq 65\mu n$ as $b\leq q$ and $\mu n\geq \mu n_0\geq hq$. {According to the definition of a good set, removing vertices from a bad set or adding vertices to a good set reduces the gap between \(|D^*_4|\) and \(|D^*_2|+b\). Specifically, each application of either Claim \ref{claim5} or Claim \ref{claim4} decreases the gap between \(|D^*_4|\) and \(|D^*_2|+b\) by 1. Thus, if the number of bad vertices in the union of the bad set and \(D^*_1\cup D^*_3\) is at least \(r^*\), then applying Claims \ref{claim5}–\ref{claim4} a total of \(r^*\) times—equivalently, redistributing exactly \(r^*\) bad vertices in the union—yields a new partition \((D_1,D_2,D_3,D_4)\) such that \(|D_4|=|D_2|+b\). It is straightforward to verify that this partition is extremal with parameters \((100, 10^5, 1/350,1-10^5\sqrt{\mu})\)\footnote{For detailed calculations, refer to Property A.1 (II) in the Appendix}, as desired. We therefore assume that the number of bad vertices in the union of the bad set and \(D^*_1\cup D^*_3\) is at most \(r^*\). Redistributing all bad vertices in this union via the application of Claims \ref{claim5}–\ref{claim4} produces a new extremal partition \((D_1,D_2,D_3,D_4)\) with parameters \((100, 10^5, 1/350,1-10^5\sqrt{\mu})\)\footnote{For detailed calculations, refer to Property A.1 (II) in the Appendix}, and every vertex in \(D\setminus B\) is good. Furthermore, we assume that $|D_4|\neq |D_2|+b$, as the desired conclusion then follows immediately.}

Now, define $r=\left||D_4| - (|D_2| + b)\right|$, then $r\leq 65\mu n$. {Next, we claim that there exist $r$ disjoint paths such that contracting these paths reduces the gap between \(|D_4|\) and \(|D_2|+b\) by $r$, where each path is obtained by extending a `special' arc. To this end,} we construct an oriented graph $H^{\prime}$ with $V(H^{\prime})= V(D\setminus B)$ as follows:

\indent $ \bullet$ $A(H^{\prime})=A(D_1, D_3)\cup A(D_1, D_4)\cup A(D_4, D_3)\cup A(D_4)$, when $|D_4|- |D_2|> b$;

\indent $ \bullet$ $A(H^{\prime})=A(D_2, D_1)\cup A(D_2)\cup A(D_3, D_1)\cup A(D_3, D_2)$, when $|D_4|- |D_2|< b$.

\begin{figure}[htbp]  
    \centering  
    \begin{minipage}[t]{0.43\textwidth}
        \centering
        \includegraphics[width=\linewidth]{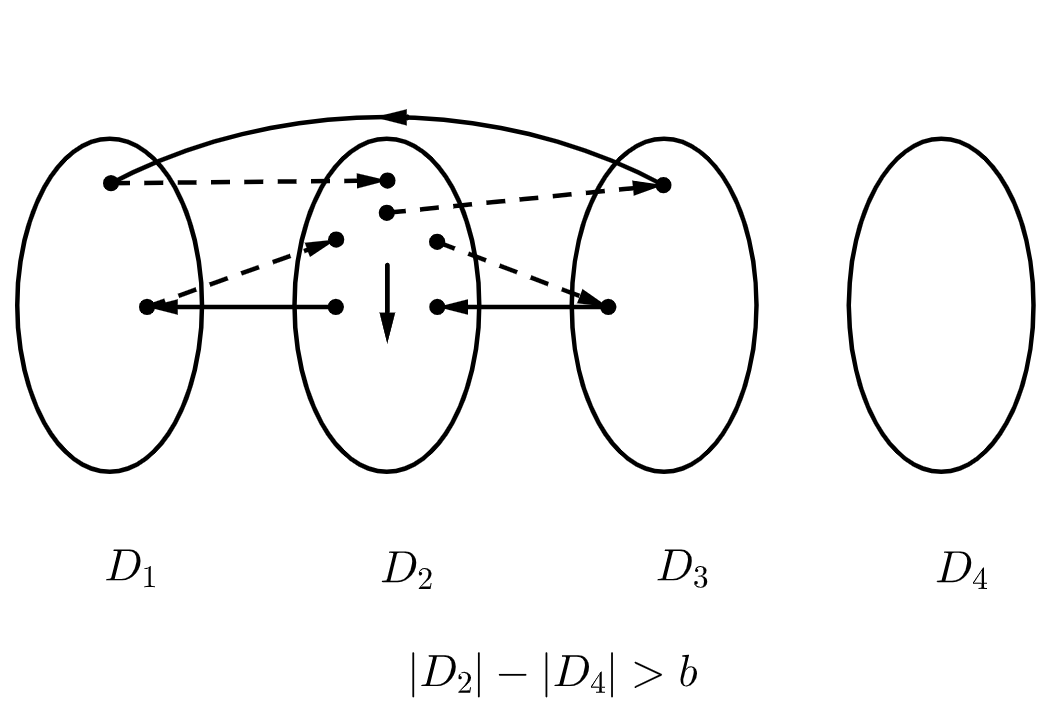}  
    \end{minipage}
    \hspace{0.1\textwidth}  
    \begin{minipage}[t]{0.43\textwidth}
        \centering
        \includegraphics[width=\linewidth]{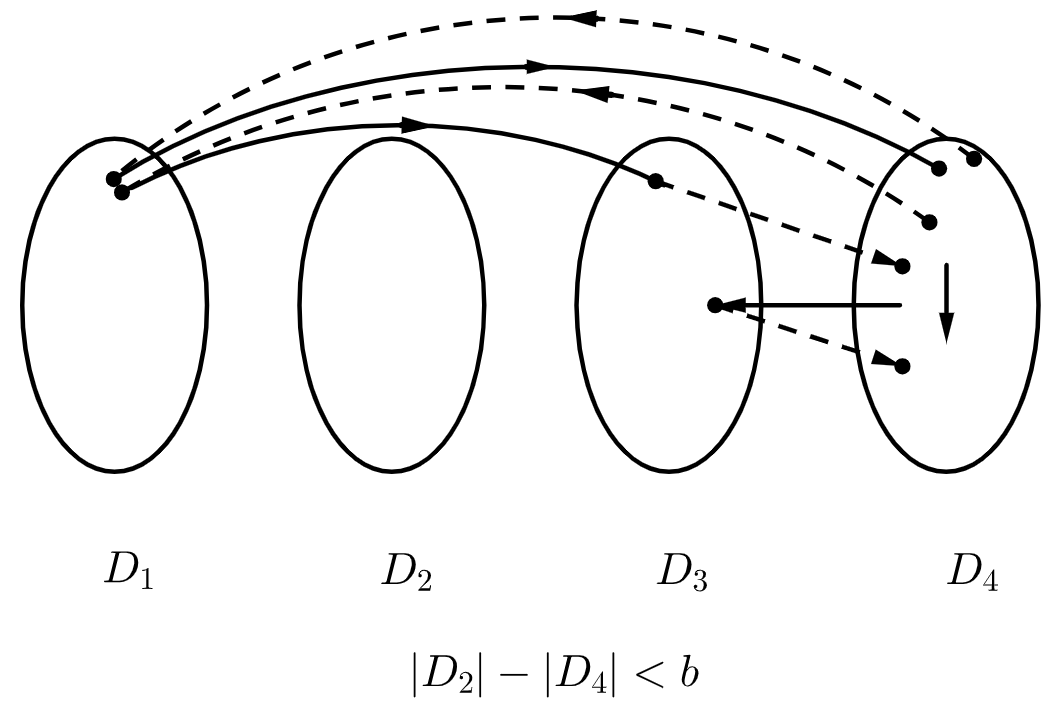}  
    \end{minipage}
    
    \caption{{Illustration of \(H'\): Solid lines denote arcs in \(H'\) (if such arcs exist in $D$), while dashed lines denote arcs that lie in the digraph $D$ but not in \(H'\); the paths formed by the concatenation of solid and dashed lines are exactly the paths \(\{Q'_i \mid i\in [r]\}\) we seek.}}
    \label{tu5}  
\end{figure}
\noindent In either case, choose $M$ to be a maximum matching in $H^{\prime}$. {Then we consider the case $ a(M)\leq  r$ and $a(M)>  r$ separately.} Suppose that $a(M)\geq  r$. {Hereafter, the bad set and good set referred to are defined with respect to the partition $(D_1,D_2,D_3,D_4)$. Since every vertex in $\bigcup_{i=1}^4 D_i$ is $1/350$-acceptable, it is readily verified that there exist $r$ disjoint paths $\{Q'_i \mid i\in[r]\}$), each containing exactly one arc in $M$ and with endvertices lying in the bad set (see Fig. \ref{tu5}). Subject to this, we choose each path $Q'_i$ to be minimal, and wherever feasible, select its endvertices to be $(1-10^5\sqrt{\mu})$-circular vertices. It should be noted that it may happen that not all endvertices of $Q'_i$ are $(1-10^5\sqrt{\mu})$-circular vertices; when this occurs, we extend $Q'_i$ to a minimal path $Q_i$ in a manner that winds around $D_1,D_2,D_3,D_4$, such that both endvertices of $Q_i$ are $(1-10^5\sqrt{\mu})$-circular vertices in the bad set. This construction implies that each path $Q_i$ has length at most 8, using exactly two more vertices from the bad set than from the good set.} Thus $|\bigcup_{i\in [r]}V(Q_i)|< 10r<650 \mu n$. Further, after contracting these paths, there is a new partition $(D'_1,D'_2,D'_3,D'_4)$ such that $|D'_4|=|D'_2| + b$. By taking $\mu$ small enough and $n_0$ large enough, it is easy to check that $(D'_1,D'_2,D'_3,D'_4)$ is extremal with parameters $ (10^3, 10^6, 1/400,1-10^6\sqrt{\mu}) $\footnote{for a detailed calculation, see Property A.1 (II) in Appendix}. Then these disjoint paths $Q_1,\ldots,Q_r$ and the partition $(D'_1,D'_2,D'_3,D'_4)$ as desired. 

Now assume $a(M)<  r$. Next, we show that this contradicts the assumption $\delta ^0(D)\geq \frac{3n + 3h + 3q - 5}{8}$. Let $L_i = V(M)\cap D_i$ for each $i\in [4]$. We divide the proof into two cases: $|D_4| - |D_2| > b$ or $|D_4| - |D_2| < b$.

\textbf{Case 2.1: $|D_4| - |D_2| > b$}

The maximality of $M$, $\sum_{i\in [4]}|L_i|\leq 2a(M)$, and the definition of the good vertex yield that
\begin{equation*}\label{9}
\begin{aligned}
a( D_1,D_3)&\leq a(D_1,L_3 ) + a( L_1,D_3 )\leq (|L_3| + |L_1|)\gamma (n/4 + 32\mu n)\leq a(M)n/300.
\end{aligned}
\end{equation*}
Similarly, we also can obtain that $a(D_1,D_4)\leq a(D_1,L_4 ) + a(L_1, D_4 ) < a(M)n/300,$ $a(D_4,D_3)\leq a(D_4,L_3 ) + a(L_4,D_3  ) < a(M)n/300,$ and $a(D_4)\leq 2|L_4|\cdot \gamma (n/4+32\mu n)<a(M)n/150.$
Hence, it follows that
$$\sum_{v\in D_1} d_{D\setminus B}^+(v)\leq \frac{|D_1|\cdot (|D_1-1|)}{2}+ a(M)n/150+ |D_1|\cdot |D_2|,$$
$$\sum_{v\in D_4} d_{D\setminus B}(v)\leq (|D_1|+|D_3|)\cdot |D_4| + 2a(M)n/150+|D_2|\cdot |D_4|,$$
$$\sum_{v\in D_3} d_{D\setminus B}^-(v)\leq \frac{|D_3|\cdot (|D_3-1|)}{2}+ a(M)n/150+ |D_3|\cdot |D_2|.$$

 By pigeonhole principle, there are three vertices $u\in D_1, v\in D_2, w\in D_3$ such that
$$d_{D\setminus B}^+(u)\leq \frac{|D_1|-1}{2}+|D_2|+\frac{a(M)n}{150 |D_1|}<\frac{|D_1|-1}{2}+|D_2|+\frac{a(M)}{30},$$
$$d_{D\setminus B}(v)<|D_1|+|D_3|+|D_2|+ \frac{a(M)}{15},\ \text{ and } d_{D\setminus B}^-(w)<\frac{|D_3|-1}{2} + |D_2| + \frac{a(M)}{30}.$$
Together $|D_4| - |D_2|-b=r>a(M)$ with $|D_1|+|D_2|+|D_3|+|D_4|= n-h$, it yields that
\begin{equation}\label{2}
\begin{aligned}
4\times \frac{3n+3h+3q-5}{8}-3h&\leq d_{D\setminus B}^+(u) + d_{D\setminus B}^-(w) + d_{D\setminus B}(v) \\
&<\frac{3}{2}(n-h+|D_2| - |D_4|)+ \frac{2a(M)}{15} -1.
\end{aligned}
\end{equation}
 Thus (\ref{2}) gives that $a(M)+b< |D_4|-|D_2|<2a/15-3q/2+3/2$, so $a(M)<0$. This is impossible.

\textbf{Case 2.2: $0\leq |D_4|- |D_2|< b$.}

The similar calculations as above yields that $a(D_1, D_3)< a(M)n/300, $ $a(D_1,D_4)< a(M)n/300,$ $a(D_4,D_3) < a(M)n/300,$ and $a(D_4) <a(M)n/150.$ By Pigeonhole principle, there are three vertices $u\in D_1, v\in D_4, w\in D_3$ such that
$$d_{D\setminus B}^-(u)<\frac{|D_1|-1}{2}+|D_4|+\frac{a(M)}{30},$$
$$d_{D\setminus B}(v)<|D_1|+|D_3|+|D_4|+ \frac{a(M)}{15},\ \text{ and }d_{D\setminus B}^+(w)<\frac{|D_3|-1}{2} + |D_4| + \frac{a(M)}{30}. $$
Together $|D_2|+b-|D_4|=r>a(M)$ with $|D_1|+|D_2|+|D_3|+|D_4|= n-h$, it yields that
\begin{equation*}
\begin{aligned}
4\times \frac{3n+3h+3q-5}{8}-3h&\leq d_{D\setminus B}^-(u) + d_{D\setminus B}^+(w) + d_{D\setminus B}(v) \\
&<\frac{3}{2}(n-h+|D_4| - |D_2|) + \frac{2a(M)}{15}-1.
\end{aligned}
\end{equation*}
This gives $|D_2|-|D_4| < 1-q+4 a(M) / 45.$ On the other hand, we previously observed that $r=|D_2| +b-|D_4| \geqslant a(M)+1$. Combining this gives $41a(M)/45< b-q\leq 0$ (as $b\leq q$), which is impossible. This completes the proof of $(i)$.
\end{proof}

\section{Appendix: Additional Proofs}

In the appendix section, for verification purposes, we strictly adhere to the notation in [10]. Specifically, the oriented graph is denoted by \( G \), and the partition of \( V(G) \) is represented as \((A, B, C, D)\). They also use the notation \(A:=P(1), B:=P(2), C:=P(3), D:=P(4)\).

\subsection{Appendix 1: Proof of Lemma 4.3}

\begin{lemma}[Lemma 4.3]\label{A.1}
There exist a real number \( \eta_0 < 1 \) and an integer $n_0$ such that for any $0< \xi  \ll \eta_0$, the following statement holds. If an oriented graph \( G \) of order \( n \geq {n_0} \) has a partition \( (A, B, C, D) \) such that
\begin{itemize}[itemsep=0pt, topsep=0.5pt,parsep=1pt]
\item[(i)] \( |B| = |D| \), and $|A|=|B|=|C|=|D|=(1/4 \pm \xi  ) |G|$ for all $i\in [4]$;
\item[(ii)] if $M$ is a subset of $A$ or $C$ with order at least $ 10^3\xi  n $, then $M$ contains a 3-path;
\item[(iii)] all but at most $800\xi  n$ vertices in $D$ have at least $|B|/32$ out-neighbours and at least $|B|/32$ in-neighbours in $D_2$;
\item[(iv)] for the partition \( (A, B, C, D) \), every vertex is $1 /600$-acceptable and the number of non-$(1 - 20q \xi )$-circular vertices is at most $\xi  n$,
\end{itemize}
 then \( G \) contains a Hamilton cycle.
\end{lemma}

The proof of Lemma \ref{A.1} needs the following Blow-up Lemma of Koml\'{o}s, S\'{a}rk\"{o}zy and Szemer\'{e}di \cite{Komlos(1997)}.
\begin{lemma}\label{a.2}
\cite{Komlos(1997)} \textbf{\emph{(}Blow-up Lemma\emph{)}} Given a graph $F$ on $[l]$ and positive reals $d', \Delta$, there is a positive real $\eta_0 =\eta_0(d', \Delta, k)$ such that the following holds for all positive numbers $l_1,\ldots , l_k$ and all $0 <\eta \leq \eta_0$. Let $F'$ be the graph obtained from $F$ by replacing each vertex $i\in F$ with a set $V_i$ of $l_i$ new vertices and joining all vertices in $V_i$ to all vertices in $V_j$ whenever $ij$ is an edge of $F$. Let $G'$ be a spanning subgraph of $F'$ such that for every edge $ij\in F$ the graph $(V_i, V_j)_{G'}$ is $(\eta, d')$-super-regular. Then $G'$ contains a copy of every subgraph $H$ of $F'$ with $\Delta(H)\leq \Delta$.
\end{lemma}

\textbf{Notation.} We will also call an edge \textbf{acceptable} if it is one of the type edge from $A$ to $B$, from $B$ to $C$, from $C$ to $D$, from $D$ to $A$, in $A$, in $C$ (so, for example, an edge from $B$ to $A$ is not). The paths $P$ have initial and final vertices in the same class and will be \textbf{acceptable}, meaning that every edge on $P$ is acceptable. Note that each such acceptable path $P$ must be \textbf{$B D$-balanced}, meaning that if we delete the initial vertex of $P$ we are left with a path that meets $B$ and $D$ the same number of times. This may be seen from the observations that visits of $P$ to $B \cup D$ alternate between $B$ and $D$, and if the path is in $A$ and then leaves, it must visit $B$ and $D$ an equal number of times before returning to $A$ (and similarly for $C$ ). Note that if we have $|B|=|D|$ and contract a $B D$-balanced path, then the resulting digraph will still satisfy $|B|=|D|$. The 'moreover' part of the following claim is used later in the proof to turn a graph with $|B|=|D|+1$ into one with $|B|=|D|$ under certain circumstances.%
\begin{proof}\textbf{of Lemma \ref{A.1}.}
We will use the 'standard' version of the Blow-up lemma (Lemma \ref{a.2}) to prove the assertion. For this, the idea is to find suitable paths which together contain all the {non-$(1 - 20q \xi )$-circular} vertices. We will contract these paths into vertices so that the resulting oriented graph $G_{1}$ consists entirely of {$(1 - 30q \xi )$-circular} vertices. Then we will find suitable paths whose contraction results in an oriented graph $G_{2}$ that satisfies $|A|=|B|=|C|=|D|$ and all of whose vertices are {$(1 - 30q \xi )$-circular}. We can then apply the Blow-up lemma to its underlying graph to find a directed Hamilton cycle in $G_{2}$ which 'winds around' $A, B, C, D$. The choice of our paths will then imply that this Hamilton cycle corresponds to one in $G$.

Let $v_{1}, \ldots, v_{t}$ be the vertices that are {$1 /600$-acceptable} but not {$(1 - 20q \xi )$-circular}. For each $v_{i}$, choose a {$(1 - 20q \xi )$-circular} out-neighbour $v_{i}^{+}$ and a {$(1 - 20q \xi )$-circular} in-neighbour $v_{i}^{-}$ so that all of these vertices are distinct, and so that the edges $v_{i} v_{i}^{+}$ and $v_{i}^{-} v_{i}$ are acceptable. Note that this can be done since {$t\leq \xi  n$}. Let $P_{i}^{\prime}$ be a path of length at most 3 starting at $v_{i}^{+}$ and ending at a {$(1 - 20q \xi )$-circular} vertex that lies in the same class as $v_{i}^{-}$, and where the successive vertices lie in successive classes, that is, the successor of a vertex $x \in V(P) \cap P(i)$ lies in $P(i+1)$. (Therefore if, for example, $v_{i}$ has the first of the {$1/600$-acceptable} properties of a vertex in $A$, then we can choose $v_{i}^{+} \in B, v_{i}^{-} \in D$ and so $P_{i}^{\prime}$ would have its final vertex in $D$. Also, if $v_{i}^{-}$ and $v_{i}^{+}$ lie in the same class then $P_{i}^{\prime}$ consists of the single vertex $v_{i}^{+}$.) Again, the paths $P_{i}^{\prime}$ can be chosen to be disjoint. Let $P_{i}:=v_{i}^{-} v_{i} v_{i}^{+} P_{i}^{\prime}$. Then the $P_{i}$ are acceptable and thus $B D$-balanced. Let $G_{1}$ be the oriented graph obtained from $G$ by contracting the paths $P_{i}$. Then every vertex of $G_{1}$ is $(1 - 30q \xi )$-circular. Moreover, the sets $A, B, C, D$ in $G_{1}$ still satisfy $|B|=|D|$, and we still have that the sizes of the other pairs of sets differ by at most {$4\xi  n$}.

Now suppose that $|A|<|C|$ and let $s:=|C|-|A|$. Greedily find a path $P_{C}$ of the form
$$
\underbrace{C C D A B \ldots C C D A B}_{s \text { times }} C
$$
consisting entirely of cyclic vertices. Therefore $P_{C}$ starts with an edge between two {$(1 - 30q \xi )$-circular} vertices in $C$. (This can be done as (ii).) Then the path $P_{C}$ uses one {$(1 - 30q \xi )$-circular} vertex in $D$, one in $A$ and so on. Let $G_{2}$ be the digraph obtained by contracting $P_{C}$. Then in $G_{2}$, we have $|A|=|C|$ and $|B|=|D|$. If $|A|>|C|$ we can achieve equality in the similar way by contracting a path $P_{A}$ of the form $A A B C D \ldots A A B C D A$. Note that since {$s\leq 4\xi  n$}, all vertices of $G_{2}$ are $(1 - 40q \xi )$-circular. Now suppose that in $G_{2}$ we have $|B|>|A|$. Let $s:=|B|-|A|$. This time, we greedily find a path $P_{B}$ of the form
$$\underbrace{B D A B C D B C D A \ldots B D A B C D B C D A}_{s \text { times }} B$$
consisting of {$(1 - 40q \xi )$-circular} vertices. The statement (iii) ensure that such a path exists. Therefore $P_{B}$ starts in such a {$(1 - 40q \xi )$-circular} vertex in $B$ and then uses {$(1 - 40q \xi )$-circular} vertices in $D, A, B$ and $C$. Then move to such a {$(1 - 40q \xi )$-circular} vertex in $D$ and use {$(1 - 40q \xi )$-circular} vertices in $B$, $C, D$ and $A$ etc. Note that $P_{B}$ is $B D$-balanced. By contracting $P_{B}$, we obtain an oriented graph (which we still call $G_{2}$ ) with $|A|=|B|=|C|=|D|$ and all of whose vertices are {$(1 - 50q \xi )$-circular}. Finally, suppose that in $G_{2}$ we have $|B|<|A|$. In this case we can equalize the sets by contracting two paths $P_{A}$ and $P_{C}$ as above.

Thus we have arranged that $|A|=|B|=|C|=|D|$ in $G_{2}$. Let $G_{2}^{\prime}$ be the underlying graph corresponding to the set of edges oriented from $P(i)$ to $P(i+1)$, for $1 \leqslant i \leqslant 4$. We chose $\eta$ such that {$q\xi  \ll \eta \leq \eta _0(1,2,4)$, where $\eta _0(1,2,4) $ is stated in Lemma \ref{a.2}}. Since all vertices of $G_{2}$ are {$(1 - 50q \xi )$-circular}, each pair $(P(i), P(i+1))$ is {$(\eta, 1)$}-super-regular in $G_{2}^{\prime}$. Also, $G_{2}^{\prime}$ is simple; that is, there are no multiple edges. Let $F^{\prime}$ be the 4-partite graph with vertex classes $A=P(1), B=P(2), C=P(3), D=P(4)$, where the 4 bipartite graphs induced by $(P(i), P(i+1))$ are all complete. Clearly $F^{\prime}$ has a Hamilton cycle, so we can apply Lemma 5 with $k=4, \Delta=2$ to find a Hamilton cycle $C_{\text {Ham }}$ in $G_{2}^{\prime}$.
\end{proof}

\subsection{Appendix 2: Proof of Lemma 4.5}
\begin{lemma}[Lemma 4.5]\label{A.3}
Let $ M' $, $ n_0 $ be positive numbers and $ \varepsilon, d, \mu $ positive constants satisfying
\begin{equation}\label{e2}
1/n_0 \ll 1/M' \ll \varepsilon \ll d \ll \mu  \ll  1.
\end{equation}
Let $ W \subseteq V(G) $ be a vertex subset with $ |W| \leq dn $, and $ G $ an oriented graph on $ n \geq n_0 $ vertices satisfying $
\delta^0(G) \geq (3/8 - 2d)n.$
Applying the Diregularity lemma to $ G \setminus W $ with $ M' $ as the minimal cluster number, we obtain the reduced digraph $ R' $ with parameters $ (\varepsilon, d) $. Suppose $ R $ is a spanning oriented subgraph of $ R' $ with parameters $ (\varepsilon, d) $ satisfying $
\delta^0(R) \geq (3/8 - 6d)|R|.$ If $ G $ is not an extremal case with parameters $(31$, $10^3$, $1/200$, $1-10^3\sqrt{\mu })$, then $ R $ is a robust $ (\mu, 1/3) $-outexpander.
\end{lemma}


The proof of Lemma \ref{A.3} needs the following lemma.
\begin{lemma}\label{a.4}
\cite{Keevash(2009)} For every $\varepsilon\in (0, 1)$, there exist numbers $M' = M'(\varepsilon)$ and $n_0 = n_0(\varepsilon)$ such that the following holds. Let $d\in [0, 1]$ with $\varepsilon \leq d/2$. Let $G$ be an oriented graph of order $n\geq n_0$ and let $R'$ be the reduced digraph with parameters $(\varepsilon, d)$ obtained by applying the Diregularity lemma to $G$ with $M'$ as the lower bound on the number of clusters. Then $R'$ has a spanning oriented subgraph $R$ such that

$(i)$ $\delta^0(R)\geq (\delta^0(G)/|G|- (3\varepsilon + d))|R|$;

$(ii)$ for all disjoint sets $S, T \subset V(R)$ with $a_G\left(S^*, T^*\right) \geqslant 3 d n^2$ we have $a_R(S, T)>d|R|^2$, where $S^*=\bigcup_{i \in S} V_i$ and $T^*=\cup_{i \in T} V_i$;

$(iii)$ for every set $S \subset V(R)$ with $a_G\left(S^*\right) \geqslant 3 d n^2$ we have $a_R(S)>d|R|^2$, where $S^*=\cup_{i \in S} V_i$.
\end{lemma}

\begin{proof}\textbf{of Lemma B.3.}
Apply the Diregularity lemma to $G\setminus W$ with $M'$ as the lower bound on the number of clusters, there is a partition $(V_0,V_1,\ldots,V_k)$ of $V(G\setminus W)$. Let $G$ and $R$ be oriented graphs as stated in the lemma. Let $|R|=k$, then
\begin{equation}\label{e1}
\delta^0(R)\geq (3/8-3d)k.
\end{equation} The lemma is trivial if $R$ is a robust $(\mu, 1/3)$-outexpander. Hence, assume that $R$ is not a robust $(\mu, 1/3)$-outexpander. Then we will first show that \(G\) has an extremal partition $(A,B,C,D)$ with parameters $(32$, $10^3$, $1/200$, $1-9000\sqrt{\mu })$.  Let
\[
A_R:=S\cap N_R^+(S),\ B_R:=N_R^+(S)\setminus S,\ C_R:=[k]\setminus(S\cup N_R^+(S)),\ D_R:=S\setminus N_R^+(S).
\]
Define \(A:=\cup_{i\in A_R}V_i\cup V_0
\cup W\), \(B:=\cup_{i\in B_R}V_i$, $C:=\cup_{i\in C_R}V_i$, and $D:=\cup_{i\in D_R}V_i\). By definition we have \(e_R(A_R,C_R)=\)
\(e_R(A_R,D_R)=e_R(D_R,C_R)=e_R(D_R)=0\). Since \(R\) has parameters \((\varepsilon,d)\), Lemma \ref{a.4} (ii)-(iii) imply that we have
\begin{equation}\label{A2}
e(A,C),e(A,D),e(D,C),e(D)<3dn^2\stackrel{(\ref{e2})}<10^3\mu n^2.
\end{equation}


\begin{claim}\label{C.1}
\( |A|,|B|,|C|,|D|=(1/4\pm 31\mu)n\).
\end{claim}
\begin{proof}
 First we show that \(A_R, B_R, C_R\) and \(D_R\) are non-empty. Since the average value of
\(\vert N_R^+(x)\cap S\vert\) over all \(x\in S\) is less than \(\vert S\vert/2\), we have
\[
\vert B_R\vert>\delta^+(R)-\vert S\vert/2\stackrel{(\ref{e1})}{>}(3/8 - 2d)k - k/3>k/30.
\]
Also \(\vert D_R\vert=\vert B_R\vert+\vert S\vert-\vert N_R^+(S)\vert>k/30 - \mu k>0\). Since \(e_R(D_R)=0\), for any \(x\in D_R\) we
have \(\vert N_R(x)\vert\leq\vert A_R\vert+\vert B_R\vert+\vert C_R\vert=\vert N_R^+(S)\vert+\vert C_R\vert\). Thus
\[
\vert C_R\vert\stackrel{(\ref{e1})}{>}2(3/8 - 2d)k-\vert N_R^+(S)\vert>2(3/8 - 2d)k-(2/3 + \mu)k>0
\]
and also \(\vert A_R\vert=\vert C_R\vert+\vert N_R^+(S)\vert+\vert S\vert - k>2(3/8 - 2d)k+\vert S\vert - k>0\).

Choose a vertex \(u_R\in D_R\) whose degree in \(R\) is minimal, a vertex \(v_R\in A_R\) whose outdegree
in \(R\) is minimal and a vertex \(w_R\in C_R\) of whose indegree in \(R\) is minimal. Since the minima are
at most the averages, inequality \((2)\) implies that \(2(3/8 - 2d)k<d(u_R)\leq\vert A_R\vert+\vert B_R\vert+\vert C_R\vert\),
\((3/8 - 2d)k<d^+(v_R)\leq\vert A_R\vert/2+\vert B_R\vert\) and \((3/8 - 2d)k<d^-(w_R)\leq\vert B_R\vert+\vert C_R\vert/2\). We also
have the inequality \(\vert B_R\vert-\vert D_R\vert=\vert N_R^+(S)\vert-\vert S\vert<\mu k\). Thus we may define positive reals \(r_A\),
\(r_B\), \(r_C\), \(r_D\) by
\begin{itemize}[itemsep=0pt, topsep=0.5pt,parsep=1pt]
    \item \(r_A:=\vert A_R\vert/2+\vert B_R\vert-(3/8 - 2d)k\);
    \item \(r_B:=(3/2)(\vert D_R\vert-\vert B_R\vert + \mu k)\);
    \item \(r_C:=\vert B_R\vert+\vert C_R\vert/2-(3/8 - 2d)k\);
    \item \(r_D:=\vert A_R\vert+\vert B_R\vert+\vert C_R\vert-2(3/8 - 2d)k\).
\end{itemize}
Then
\[
r_A + r_B+r_C + r_D=\frac{3}{2}(\vert A_R\vert+\vert B_R\vert+\vert C_R\vert+\vert D_R\vert)+3\mu k/2 - 4(3/8 - 2d)k\stackrel{(\ref{e2})}<2\mu k.
\]
This in turn implies that
\begin{itemize}[itemsep=0pt, topsep=0.5pt,parsep=1pt]
    \item \(\vert D_R\vert=k-(\vert A_R\vert+\vert B_R\vert+\vert C_R\vert)=k - 2(3/8 - 2d)k-r_D=k/4\pm 5\mu k\);
    \item \(\vert B_R\vert=\vert D_R\vert+2ck-\frac{2}{3}r_B=k/4\pm10\mu k\);
    \item \(\vert A_R\vert=2((3/8 - 2d)k-\vert B_R\vert+r_A)=k/4\pm30\mu k\) and
    \item \(\vert C_R\vert=2((3/8 - 2d)k-\vert B_R\vert+r_C)=k/4\pm30\mu k\).
\end{itemize}
Altogether, this gives \(\vert A\vert,\vert B\vert,\vert C\vert,\vert D\vert=(1/4\pm 31\mu )n\).
\end{proof}

\begin{claim}\label{C.2} The following statements hold.
\begin{itemize}[itemsep=0pt, topsep=0.5pt,parsep=1pt]
    \item \(e(A,B)>(1 - 10^3\mu)n^2/16\);\(e(B,C)>(1 - 10^3\mu)n^2/16\);\(e(C,D)>(1 - 10^3\mu)n^2/16\);\\ \(e(D,A)>(1 - 10^3\mu)n^2/16\);
    \item \(e(A)>(1/2 - 500\mu)n^2/16\); \(e(C)>(1/2 - 500\mu)n^2/16\);
    \item \(e(B,D)>(1/2 - 500\mu)n^2/16\);\(e(D,B)>(1/2 - 500\mu)n^2/16\).
\end{itemize}
\end{claim}
\begin{proof}
Since \(e(A,C),e(A,D)<3dn^2\) by \((6)\) we have
\[
\sum_{x\in A}d^+(x)\leq\vert A\vert^2/2+\vert A\vert\vert B\vert+6dn^2\stackrel{(\ref{e2})}\leq (3/2 + 400\mu)n^2/16.
\]
On the other hand, \(\sum_{x\in A}d^+(x)\geq\vert A\vert(3n - 4)/8\geq (3/2 - 200\mu)n^2/16\). Therefore we must
have \(e(A)>(1/2 - 10^3\mu)n^2/16\) and \(e(A,B)>(1 - 10^3\mu)n^2/16\). Also, since \(e(A,C),e(D,C)<
3dn^2\) we have
\[
(3/2 - 200\mu)n^2/16<\sum_{x\in C}d^-(x)<\vert B\vert\vert C\vert+\vert C\vert^2/2+6dn^2=(3/2 + 400\mu)n^2/16,
\]
so \(e(C)>(1/2 - 10^3\mu)n^2/16\) and \(e(B,C)>(1 - 10^3\mu)n^2/16\). Next, writing \(\overline{D}:=A\cup B\cup C\)
and using the inequalities \(e(D),e(D,C),e(A,D)<3dn^2\) gives
\begin{equation*}
\begin{aligned}
(3 - 500\mu)n^2/16&<e(D,\overline{D})+e(\overline{D},D)+2e(D)\\
&<e(D,A)+e(D,B)+e(B,D)+e(C,D)+12dn^2\\
&\leq\vert D\vert(\vert A\vert+\vert B\vert+\vert C\vert)+12dn^2=(3 + 800\mu)n^2/16,
\end{aligned}
\end{equation*}
so \(e(D,A)>(1 - 10^3\mu)n^2/16\) and \(e(C,D)>(1 - 10^3\mu)n^2/16\). Finally, since \(e(D,C),e(D)<
3dn^2\) we have
\[
(3/2 - 300\mu)n^2/16<\sum_{x\in D}d^+(x)<\vert A\vert\vert D\vert+e(D,B)+6dn^2<(1 + 400\mu)n^2/16+e(D,B)
\]
and so \(e(D,B)>(1/2 - 500\mu)n^2/16\). Since \(e(A,D),e(D)<3dn^2\) we have
\[
(3/2 - 300\mu)n^2/16<\sum_{x\in D}d^-(x)<e(B,D)+\vert C\vert\vert D\vert+6dn^2<(1 + 400\mu)n^2/16+e(B,D)
\]
and so \(e(B,D)>(1/2 - 500\mu)n^2/16\).
\end{proof}

Henceforth we will use only Claims \ref{C.1} and \ref{C.2} and make no further use of the information in \((\ref{A2})\). This has the advantage of making our picture invariant under the relabelling \(A\leftrightarrow C\),
\(B\leftrightarrow D\). 
Then we show that the number of non $(1 - 9000\sqrt{\mu })$-circular vertices is at most $\sqrt{\mu } n$, and we can arrange every vertex of $G$ to be $1/200$-acceptable. To calculate the number of non $(1 - 9000\sqrt{\mu })$-circular vertices, we assume that there are at least $x|P(i)|$ vertices in $P(i)$ with less than $(1 - 9000\sqrt{\mu })|P(i-1)|$ in-neighbours in $P(i-1)$. Thus the number of arcs from $P(i-1)$ to $P(i)$ is
$$e(P(i-1), P(i))\leq (1 - x)|P(i)||P(i-1)| + (1 - 9000\sqrt{\mu })\cdot x |P(i)||P(i-1)|. $$
Together with $e(P(i-1), P(i)) > (1 - 10^3\mu)n^2 / 16$, for $i\in [4]$ (mod 4), the value of $x$ is at most $\sqrt{\mu}/9$. Similarly, there are at most $\sqrt{\mu}|P(i)|/9$ vertices in $P(i)$ with less than $(1 - 9000\sqrt{\mu })|P(i+1)|$ out-neighbours in $P(i+1)$. In summary, $D$ has at most $\sqrt{\mu}n$ vertices that are not $(1 - 9000\sqrt{\mu })$-circular.

\begin{claim}
By reassigning vertices that are not {$(1 - 9000\sqrt{\mu })$-cyclic} to \((A,B,C, D)\) we can arrange that every vertex of \(G\) is {$1/200$-acceptable}.
\end{claim}

\begin{proof}
To satisfy the first statement of the claim, for any vertex $x\in G$ we let $P_{x}^{+}:=\{1\leq i\leq4:$
$|N^{+}(x)\cap P(i)| > n/400\}$, $P_{x}^{-}:=\{1\leq i\leq4:|N^{-}(x)\cap P(i)| > n/400\}$, and $P_{x}:=P_{x}^{+}\cup P_{x}^{-}$.
By the minimum semi-degree condition $|P_{x}^{+}|\geq2$, $|P_{x}^{-}|\geq2$ and $|P_{x}|\geq3$. If there is some $i$
such that $i + 1\in P_{x}^{+}$ and $i - 1\in P_{x}^{-}$ (where we use addition and subtraction mod 4) then
we can put $x$ into $P(i)$ and it will have property $P(i):P(i - 1)_{>1/200}P(i + 1)^{>1/200}$, that is, $x$ will become {$1/200$-acceptable}. Otherwise we must have either $P_{x}^{+}=\{1,3\}$ and $P_{x}^{-}=\{2,4\}$, or
$P_{x}^{-}=\{1,3\}$ and $P_{x}^{+}=\{2,4\}$. In either case we can put $x$ into $A = P(1)$: in the first case it
will have property $A:A^{>1/200}D_{>1/200}$ and in the second case property $A:A_{>1/200}B^{>1/200}$.
Therefore in both cases $x$ will become {$1/200$-acceptable}. As before, by increasing the coefficient of the $\sqrt{\mu}$-notation if necessary, we can ensure that the properties of all other vertices are
maintained.
\end{proof}

 That is, $G$ is an extremal case with parameters $(31$, $10^3$, $1/200$, $1-10^3\sqrt{\mu })$ by the definition, a contradiction. Thereby, $R$ is a robust $(\mu, 1/3)$-outexpander.
\end{proof}